\documentclass[reqno]{amsart}

\usepackage[colorlinks,citecolor=green,linkcolor=red]{hyperref}
\usepackage{a4wide}
\usepackage{color}
\usepackage{mathrsfs}
\usepackage{mathtools}
\usepackage{amsmath}
\usepackage{cleveref}
\usepackage{amsthm}
\usepackage{amssymb}
\usepackage{esint}
\usepackage{nicefrac}
\numberwithin{equation}{section}

\usepackage[latin1]{inputenc}

\newcommand{\N}{\mathbb{N}}
\newcommand{\R}{\mathbb{R}}
\newcommand{\sfd}{{\sf d}}
\renewcommand{\d}{{\mathrm d}}

\newcommand{\restr}[1]{\lower3pt\hbox{\(|_{#1}\)}}
\newcommand{\eps}{\varepsilon}
\newcommand{\nchi}{{\raise.3ex\hbox{\(\chi\)}}}

\newcommand{\mm}{{\mathfrak m}}
\newcommand{\haus}{{\mathscr H}}
\newcommand{\Tan}{{\rm Tan}}
\newcommand{\Leb}{\mathcal L}
\newcommand{\supp}{{\rm supp}}
\renewcommand{\P}{{\rm P}}
\renewcommand{\Cap}{{\rm Cap}}

\newcommand{\Lip}{{\rm Lip}}
\newcommand{\lip}{{\rm lip}}
\newcommand{\diam}{{\rm diam}}
\newcommand{\Test}{{\rm Test}}
\newcommand{\tr}{{\rm tr}}
\newcommand{\TestV}{{\rm TestV}}
\newcommand{\dive}{{\rm div}}
\newcommand{\Per}{{\rm Per}}
\newcommand{\abs}[1]{\left\lvert#1\right\rvert}
\newcommand{\norm}[1]{\left\lVert#1\right\rVert}
\newcommand{\Ch}{{\rm Ch}}
\newcommand{\res}{\mathop{\hbox{\vrule height 7pt width .5pt depth 0pt
			\vrule height .5pt width 6pt depth 0pt}}\nolimits}
\DeclareMathOperator{\Cbs}{C_{\rm bs}}

\DeclareMathOperator{\Lipb}{Lip_b}

\DeclareMathOperator{\CD}{CD}
\DeclareMathOperator{\RCD}{RCD}
\DeclareMathOperator{\ncRCD}{ncRCD}
\DeclareMathOperator{\Hess}{Hess}

\newtheorem{theorem}{Theorem}[section]
\newtheorem{corollary}[theorem]{Corollary}
\newtheorem{lemma}[theorem]{Lemma}
\newtheorem{proposition}[theorem]{Proposition}
\newtheorem{definition}[theorem]{Definition}

\newtheorem{remark}[theorem]{Remark}

\title[Rectifiability of the reduced boundary over \(\RCD\) spaces]{Rectifiability of the reduced boundary for sets of finite perimeter over \(\RCD(K,N)\) spaces}
\author{Elia Bru\`{e}, Enrico Pasqualetto, and Daniele Semola}

\address{Scuola Normale Superiore\\
         Piazza dei Cavalieri 7 \\
         56126 Pisa \\
         Italy}
\email{elia.brue@sns.it}

\address{University of Jyvaskyla\\
         Department of Mathematics and Statistics \\
         P.O. Box 35 (MaD) \\
         FI-40014 University of Jyvaskyla \\
         Finland}
\email{enrico.e.pasqualetto@jyu.fi}

\address{Scuola Normale Superiore\\
         Piazza dei Cavalieri 7 \\
         56126 Pisa \\
         Italy}
\email{daniele.semola@sns.it}

\begin{document}
\date{\today} 
\begin{abstract}
This note is devoted to the study of sets of finite perimeter over $\RCD(K,N)$ metric measure spaces. Its aim is to complete the picture about the generalization of De Giorgi's theorem within this framework. Starting from the results of \cite{ABS18} we obtain uniqueness of tangents and rectifiability for the reduced boundary of sets of finite perimeter. As an intermediate tool, of independent interest, we develop a Gauss--Green integration-by-parts formula tailored to this setting.	
These results are new and non-trivial even in the setting of Ricci limits.
\end{abstract}
\maketitle
\tableofcontents
\section*{Introduction}
In the last years the theory of $\RCD(K,N)$ metric measure spaces has undergone a fast and remarkable development. After the introduction of the so-called \textit{curvature-dimension} condition $\CD(K,N)$ in the seminal and independent works \cite{Sturm06I,Sturm06II} and \cite{Lott-Villani09}, the notion of $\RCD(K,N)$ space was proposed in \cite{Gigli12} after the study of its infinite-dimensional counterpart $\RCD(K,\infty)$ in \cite{AmbrosioGigliSavare11-2} (see also \cite{AmbrosioGigliMondinoRajala12} for the case of $\sigma$-finite reference measure). In the infinite-dimensional case the equivalence with the Bochner inequality was studied in \cite{AmbrosioGigliSavare12}, then \cite{Erbar-Kuwada-Sturm13} established equivalence with the dimensional Bochner inequality for the so-called class $\RCD^*(K,N)$ (see also \cite{AmbrosioMondinoSavare13}). Equivalence between $\RCD^*(K,N)$ and $\RCD(K,N)$ has been eventually achieved in \cite{CavallettiMilman16}.

We do know nowadays that, apart from smooth weighted Riemannian manifolds (with generalized Ricci tensor bounded from below), this class includes Ricci limits (see \cite{Cheeger-Colding97I,Cheeger-Colding97II,Cheeger-Colding97III}) and Alexandrov spaces \cite{Petrunin11}.  

One of the main research lines within this theory in recent times has been aimed at understanding the structure of $\RCD(K,N)$ spaces. After \cite{Mondino-Naber14,MK16,DPMR16,GP16-2} we know that they are rectifiable as metric measure spaces. Moreover in \cite{BS18} the first and the third named authors proved that they have constant dimension, in the almost everywhere sense.

This being the state of the art, we have reached a good understanding of the structure of $\RCD(K,N)$ spaces \textit{up to measure zero}. It sounds therefore quite natural to try to push the study further, investigating their structure, both from the analytic and from the geometric points of view, up to sets of positive codimension.\\ 
In this perspective in the last two years there have been some independent and remarkable developments. We wish to mention a few of them below, without the aim of being complete in this list.
\begin{itemize}
\item In the setting of non collapsed Ricci limit spaces, Cheeger-Jiang-Naber have obtained in \cite{CheegerJiangNaber18} rectifiability for the singular sets of any codimension. Let us also mention \cite{CheegerNaber15}, where some of the ideas developed in \cite{CheegerJiangNaber18} where already present, and \cite{ABS19}, where some estimates (actually much weaker than those in \cite{CheegerJiangNaber18}) are proved for the singular strata of non collapsed $\RCD$ spaces.
\item There have been some efforts aimed at defining a notion of boundary for metric measure spaces and relating it with the singular set of codimension $1$. See \cite{LKP17} and the very recent \cite{KM19}.
\item One of the main contributions of \cite{Gigli14} was the development of the language of tensor fields defined almost everywhere (with respect to the reference measure) on $\RCD$ spaces. In \cite{DGP19} the notion of tensor field defined ``$2$-capacity-almost everywhere'' is defined and it is proved that Sobolev vector fields on $\RCD$ spaces have a representative in this class.
\item In \cite{ABS18}, the first and third named authors together with Ambrosio initiated a fine study of sets of finite perimeter over $\RCD(K,N)$ spaces, with the project of generalizing the Euclidean De Giorgi theorem to this framework. 
\end{itemize} 

One of the main results in \cite{ABS18} was the existence of a Euclidean half-space as tangent space to a set of finite perimeter at almost every point (with respect to the perimeter measure). This conclusion could be improved to a uniqueness statement (up to negligible sets) only in the case of a non collapsed ambient space. 
The state of the theory of sets of finite perimeter was at that stage comparable to that of the structure theory after \cite{Gigli-Mondino-Rajala15}, where existence of Euclidean tangent spaces almost everywhere with respect to the reference measure was proved. Uniqueness of tangents in the possibly collapsed case and rectifiability for the boundary were conjectured by analogy with the Euclidean theory, but left as open questions in \cite{ABS18}. Let us point out that, up to our knowledge, no general rectifiability criterion is known at this stage for (subsets of) metric measure spaces.\\ 
Aim of this note is to provide a positive answer to these questions, providing a counterpart in codimension 1 of \cite{Mondino-Naber14} and of De Giorgi's theorem in this setting.\\ 
Together with uniqueness of tangents (cf.\ \Cref{th:uniqueness}) and rectifiability (cf.\ \Cref{th: rectifiability}) we also establish a representation formula for the perimeter measure in terms of the codimension 1 Hausdorff measure (cf.\ \Cref{cor:reprper}). As an intermediate tool which, however, we find to have independent interest we prove in \Cref{thm:Gauss-Green} a Gauss--Green integration-by-parts formula for Sobolev vector fields. 

The proof of uniqueness for blow-ups of sets of finite perimeter follows a strategy quite similar to that of the uniqueness theorem for tangents to $\RCD(K,N)$ spaces adopted in \cite{Mondino-Naber14}. As in that case, closeness to a rigid configuration (half-space in Euclidean space) at a certain location and along a certain scale, which is what we learn from \cite{ABS18}, can be turned into closeness to the same configuration at almost any location and at any scale, yielding uniqueness.\\
To encode the ``closeness information'' in analytic terms we rely on the use of harmonic $\delta$-splitting maps, which were introduced in \cite{Cheeger-Colding96} and turned to be an extremely powerful tool in the study of Ricci limits (see \cite{Cheeger-Colding97I,Cheeger-Colding97II,Cheeger-Colding97III} and the more recent \cite{CheegerNaber15,CheegerJiangNaber18}). To the best of our knowledge this is the first time they are explicitly used in the $\RCD$-theory, even though their use is implicit in \cite{AHPT18}, and we establish some of their properties within this framework.\\
Propagation of regularity almost at every location and at any scale, which was a consequence of a maximal function argument in \cite{Mondino-Naber14}, this time follows from a weighted maximal function argument suitably adapted to the codimension 1 framework. This argument heavily relies on the interplay between the fact that the perimeter measure is a codimension 1 measure (which was proved in a fairly more general context in \cite{A02}) and the fact that harmonic functions satisfy $L^2$ Hessian bounds on $\RCD(K,N)$ spaces.

In order to explain the strategy and the difficulties in the proof of rectifiability for the reduced boundary, let us recall how things work on $\R^n$. Therein a crucial role is played by the exterior normal to the set of finite perimeter, which is an almost everywhere unit valued vector field providing the representation $D\nchi_E=\nu_E\abs{D\nchi_E}$ for the distributional derivative of the set of finite perimeter $E$. Relying on the properties of the exterior normal one can obtain a characterization of blow-ups in a much simpler way than in \cite{ABS18} and even get rectifiability of the boundary, proving that sets where the unit normal is not oscillating too much are bi-Lipschitz to subsets of $\R^{n-1}$.\\
When trying to reproduce the Euclidean approach in the \textit{non smooth} and \textit{non flat} realm of $\RCD$ spaces, one faces two main difficulties. The first one due to the fact that the theory of tangent modules, as developed in \cite{Gigli14}, allows to talk about vector fields only up to negligible sets with respect to the reference measure (as the reduced boundary of a set of finite perimeter is not). The second one is that controlling the behaviour of the normal vector cannot be enough to control the behaviour of the set in this framework, since the space itself might ``oscillate''. This is a common feature of geometry on metric measure spaces (see also the introduction of \cite{CheegerJiangNaber18}), which can be understood looking at the following example: let $(X,\sfd,\mm)$ be any $\RCD(K,N)$ m.m.s.\ and take its product with the Euclidean line. Then consider the ``generalized half-space'' $\{t<0\}$, where $t$ denotes the coordinate along the line: it is easily seen that it is a set of locally finite perimeter and one can identify its reduced boundary with $X$. Moreover, whatever notion of unit normal we have in mind, this will be non oscillating in this case. Still, rectifiability of $(X,\sfd,\mm)$ is highly non trivial and requires \cite{Mondino-Naber14} to be achieved.      

To handle the first difficulty we mentioned above, we rely on the very recent \cite{DGP19}, where a notion of cotangent module with respect to the $2$-capacity is introduced and studied. Building upon the fact that the $2$-capacity controls the perimeter measure in great generality, we introduce the notion of tangent module over the boundary of a set of finite perimeter (cf.\ \Cref{thm:tg_mod_over_bdry}).\\ 
Furthermore we prove that there is a well-defined unit normal to a set of finite perimeter as an element of this module, that it satisfies the Gauss--Green integration-by-parts formula and, relying on functional analysis tools, that it can be approximated by regular vector fields (cf.\ \Cref{thm:Gauss-Green} for a rigorous statement).   

The results obtained in the study of the unit normal are then combined in a new way with the theory of $\delta$-splitting maps to prove rectifiability of the reduced boundary for sets of finite perimeter.\\ 
We introduce a notion of $\delta$-orthogonality to the unit normal for $\delta$-splitting maps. Then we prove on the one hand that $\delta$-splitting maps $\delta$-orthogonal to the unit normal control both the geometry of the space and that of the boundary of the set of finite perimeter (and vice-versa). On the other hand the combination of $\delta$-orthogonality and $\delta$-splitting is seen to be suitable for propagation at many locations and any scale with maximal function arguments (cf.\ \Cref{prop:good aproximation}  and \Cref{prop:rectifiability}).

We wish to emphasize the fact that, on the one hand the coarea formula (which holds in the great generality of metric measure spaces) provides plenty of sets of finite perimeter even in the non-smooth context, on the other one there is no hope to have a notion of smooth hypersurface within this setting. Therefore we expect the range of applications to be large in the development of the theory of spaces satisfying lower curvature bounds, both for the techniques we develop in the paper and for our main results that, to the best of our knowledge, are new also for Ricci limits.

A number of open questions remains open and suitable for future investigation after the study pursued in this paper. In particular, we wish to point out that neither the constancy of the dimension result of \cite{BS18}, nor the absolute continuity of the reference measure with respect to the Hausdorff measure (\cite{MK16,GP16-2,DPMR16}), play a role in the proofs of our results. It might be interesting to investigate whether one can prove constancy of the dimension for tangents also in the case of sets of finite perimeter and sharpen the representation formula for the perimeter measure (maybe relying on the good understanding we have of the top dimensional case). In this regard let us point out that, in none of these cases, the techniques adopted to solve the analogous problems in codimension 0 seem to work when dealing with sets of finite perimeter. 

This note is organised as follows: in \Cref{sec:preliminaries} we collect all the preliminary material to be used in the paper. We dedicate \Cref{section:Gauss-Green formula} to the construction of the tangent module over the boundary of a set of finite perimeter and to establishing a Gauss--Green integration-by-parts formula. Uniqueness of blow-ups is the main outcome of \Cref{sec:uniqueness}, while rectifiability for the reduced boundary is obtained in \Cref{sec:bdryrect}.
\medskip

\textbf{Acknowledgements.} The authors wish to thank Luigi Ambrosio and Nicola Gigli for several useful comments on an earlier version of this note and Aaron Naber for an enlightening discussion about propagation of regularity in codimension 1.

The second named author was partially supported by the Academy of
Finland, projects no.\ 274372, 307333, 312488, and 314789.
\section{Preliminaries and notations}\label{sec:preliminaries}
\subsection{Calculus tools}
Throughout this paper a \textit{metric measure space} is a triple $(X,\sfd,\mm)$, where $(X,\sfd)$ is a complete and separable metric space and $\mm$ is a nonnegative Borel measure on $X$ finite on bounded sets. From now on when we write \emph{m.m.s.}\ we mean \emph{metric measure space(s)}.\\ 
In order to simplify the notation, numerical constants depending only on the parameters entering into play, will be denoted with the same letter $C$ even if they do vary. Often we will make explicit their dependence on the parameters writing for instance $C_N,C_{N,K}$.
    
We will denote by $B_r(x)=\{\sfd(\cdot,x)<r\}$ and $\bar{B}_r(x)=\{\sfd(\cdot,x)\leq r\}$ the open and closed balls respectively, by $\Lip(X,\sfd)$ (resp.\ $\Lipb(X,\sfd)$, $\Lip_c(X,\sfd)$, $\Lip_{\rm bs}(X,\sfd)$, $\Lip_{\text{loc}}(X,\sfd)$) the space of Lipschitz
(resp.\ bounded Lipschitz, compactly supported Lipschitz, Lipschitz with bounded support, Lipschitz on bounded sets) functions and for any $f\in\Lip(X,\sfd)$ we shall denote its slope by
\begin{equation*}
\lip f(x):=\limsup_{y\to x}\frac{\abs{f(x)-f(y)}}{\sfd(x,y)}.
\end{equation*}
We shall use the standard notation $L^p(X,\mm)=L^p(\mm)$, for the $L^p$ spaces and $\Leb^n$ for the $n$-dimensional Lebesgue measure on $\R^n$. We shall denote by $\omega_n$ the Lebesgue measure of the unit ball in $\R^n$. If $f\in L^1_{\rm loc}(X,\mm)$ and $U\subset X$ is such that $0<\mm(U)<+\infty$, then $\fint_Uf\d \mm$ denotes the average of $f$ over $U$.

The Cheeger energy $\Ch:L^2(X,\mm)\to[0,+\infty]$ is the convex and lower semicontinuous functional defined through
\begin{equation}\label{eq:cheeger}
\Ch(f):= \inf\left\lbrace\liminf_{n\to\infty}\int_X(\lip f_n)^2\d\mm:\quad f_n\in\Lipb(X)\cap L^2(X,\mm),\ \norm{f_n-f}_2\to 0 \right\rbrace
\end{equation}
and its finiteness domain will be denoted by $H^{1,2}(X,\sfd,\mm)$, sometimes we write $H^{1,2}(X)$ omitting the dependence on $\sfd$ and $\mm$ when it is clear from the context. Looking at the optimal approximating sequence in \eqref{eq:cheeger}, it is possible to identify a canonical object $\abs{\nabla f}$, called minimal relaxed slope, providing the integral representation
\begin{equation*}
\Ch(f):= \int_X\abs{\nabla f}^2\d\mm\qquad\forall f\in H^{1,2}(X,\sfd,\mm).
\end{equation*}

Any metric measure space such that $\Ch$ is a quadratic form is said to be \textit{infinitesimally Hilbertian}. Let us recall from \cite{AmbrosioGigliSavare11-2,Gigli12} that, under this assumption, the function 
\begin{equation*}
\nabla f_1\cdot\nabla f_2:= \lim_{\eps\to 0}\frac{\abs{\nabla(f_1+\eps f_2)}^2-\abs{\nabla f_1}^2}{2\eps}
\end{equation*}
defines a symmetric bilinear form on $H^{1,2}(X,\sfd,\mm)\times H^{1,2}(X,\sfd,\mm)$ with values into $L^1(X,\mm)$.

It is possible to define a Laplacian operator $\Delta:\mathcal{D}(\Delta)\subset L^{2}(X,\mm)\to L^2(X,\mm)$ in the following way. We let $\mathcal{D}(\Delta)$ be the set of those $f\in H^{1,2}(X,\sfd,\mm)$ such that, for some $h\in L^2(X,\mm)$, one has
\begin{equation}\label{eq:amb1}
\int_X \nabla f\cdot\nabla g\, \d\mm=-\int_X hg\,\d\mm\qquad\forall g\in H^{1,2}(X,\sfd,\mm) 
\end{equation} 
and, in that case, we put $\Delta f=h$. 
It is easy to check that the definition is well-posed and that the Laplacian is linear (because $\Ch$ is a quadratic form). 

The heat flow $P_t$ is defined as the $L^2(X,\mm)$-gradient flow of $\frac{1}{2}\Ch$. Its existence and uniqueness
follow from the Komura-Brezis theory. It can be equivalently characterized by saying that for any
$u\in L^2(X,\mm)$ the curve $t\mapsto P_tu\in L^2(X,\mm)$ is locally absolutely continuous in $(0,+\infty)$ and satisfies
\begin{equation*}
\frac{\d}{\d t}P_tu=\Delta P_tu \quad\text{for }\Leb^1\text{-a.e. }t\in(0,+\infty),\qquad
\lim_{t\downarrow 0}P_tu=u\quad\text{in $L^2(X,\mm)$.}
\end{equation*}
Under the infinitesimal Hilbertianity assumption the heat flow provides a linear, continuous and self-adjoint contraction semigroup in $L^2(X,\mm)$. Moreover $P_t$ extends to a linear, continuous and mass preserving operator,
still denoted by $P_t$, in all the $L^p$ spaces for $1\le p < +\infty$. 

\begin{definition}[Function of bounded variation]\label{def:bvfunction}
	We say that $f\in L^1(X,\mm)$ belongs to the space ${\rm BV }(X,\sfd,\mm)$ of functions of bounded variation if there exist
	locally Lipschitz functions $f_i$ converging to $f$ in $L^1(X,\mm)$ such that
	\begin{equation*}
	\limsup_{i\to\infty}\int_X\abs{\nabla f_i}\d \mm<+\infty.
	\end{equation*}
	\end{definition}  
	If $f\in {\rm BV }(X,\sfd,\mm)$ one can define 
	\begin{equation*}
	\abs{Df}(A):= \inf\left\lbrace \liminf_{i\to\infty}\int_A\abs{\nabla f_i}\d \mm: f_i\in\Lip_{\text{loc}}(A),\quad f_i\to f \text{ in } L^1(A,\mm)\right\rbrace  
	\end{equation*}
	for any open $A\subset X$. In \cite{Ambrosio-DiMarino14} (see also \cite{MR03} for the case of locally compact spaces) it is proven that this set function 
	is the restriction to open sets of a finite Borel measure that we call \emph{total variation of $f$} and still denote by $\abs{Df}$.

Dropping the global integrability condition on $f=\nchi_E$, let us recall now the analogous definition of set of finite perimeter 
in a metric measure space (see again \cite{A02,MR03,Ambrosio-DiMarino14}).

\begin{definition}[Perimeter and sets of finite perimeter]
	\label{def:setoffiniteperimeter}
	Given a Borel set $E\subset X$ and an open set $A$ the perimeter $\Per(E,A)$ is defined in the following way:
	\begin{equation*}
	\Per(E,A):=\inf\left\lbrace \liminf_{n\to\infty}\int_A\abs{\nabla u_n}\d\mm: u_n\in\Lip_{{\rm loc}}(A),\quad u_n\to\nchi_E\quad \text{in } L^1_{{\rm loc}}(A,\mm)\right\rbrace .
	\end{equation*}
	We say that $E$ has finite perimeter if $\Per(E,X)<+\infty$. In that case it can be proved that the set function $A\mapsto\Per(E,A)$ is the restriction to open sets of a finite Borel measure $\Per(E,\cdot)$ defined by
	\begin{equation*}
	\Per(E,B):=\inf\left\lbrace \Per(E,A): B\subset A,\text{ } A \text{ open}\right\rbrace.
	\end{equation*}
\end{definition}

Let us remark for the sake of clarity that $E\subset X$ with finite $\mm$-measure is a set of finite perimeter if and only if $\nchi_E\in{\rm BV}(X,\sfd,\mm)$ and that $\Per(E,\cdot)=\abs{D\nchi_E}(\cdot)$. In the following we will say that $E\subset X$ is a set of locally finite perimeter if $\nchi_E$ is a function of locally bounded variation, that is to say $\eta\nchi_E\in{\rm BV}(X,\sfd,\mm)$ for any $\eta\in \Lip_{\rm bs}(X,\sfd)$.
%with bounded support. 

\subsubsection{Total variation of ${\rm BV}$ functions via integration by parts}
Let us present an equivalent approach to the study of ${\rm BV}$ functions in m.m.s.\ introduced by Di Marino in \cite{DM14}. Before stating \Cref{thm:repr_tv} we need to recall the notion of derivation.
\begin{definition}[Derivation]
	Let \((X,\sfd,\mm)\) be a metric measure space. Then a \emph{derivation}
	on \(X\) is a linear map \(\boldsymbol b:\,\Lip_{\rm bs}(X)\to L^0(\mm)\) such
	that the following properties are satisfied:
	\begin{itemize}
		\item[\(\rm i)\)] \textsc{Leibniz rule.} \(\boldsymbol b(fg)=\boldsymbol b(f)g
		+f\boldsymbol b(g)\) for every \(f,g\in\Lip_{\rm bs}(X)\).
		\item[\(\rm ii)\)] \textsc{Weak locality.} There exists \(G\in L^0(\mm)\) such that
		\[
		|\boldsymbol b(f)|\leq G\,\lip_a(f)\footnote{where $\lip_a(f)(x):=\lim_{r\to 0}\sup_{\sfd(x,y)<r }\frac{|f(x)-f(y)|}{\sfd(x,y)}$ is the so-called asymptotic Lipschitz constant.}		\;\;\;\mm\text{-a.e.}
		\quad\text{ for every }f\in\Lip_{\rm bs}(X).
		\]
		The least function \(G\) (in the \(\mm\)-a.e.\ sense) with this property is denoted
		by \(|\boldsymbol b|\).
	\end{itemize}
\end{definition}
The space of all derivations on \(X\) is denoted by \({\rm Der}(X)\).
Given any derivation \(\boldsymbol b\in{\rm Der}(X)\), we define its \emph{support}
\({\rm supp}(\boldsymbol b)\subset X\) as the essential closure of
\(\{|\boldsymbol b|\neq 0\}\). For any open set \(U\subset X\), we write
\({\rm supp}(\boldsymbol b)\Subset U\) if \({\rm supp}(\boldsymbol b)\) is bounded
and \({\rm dist}({\rm supp}(\boldsymbol b),X\setminus U)>0\).
Given any \(\boldsymbol b\in{\rm Der}(X)\) with \(|\boldsymbol b|\in L^1_{\rm loc}(X)\),
we say that \({\rm div}(\boldsymbol b)\in L^p(\mm)\) -- for some exponent
\(p\in[1,\infty]\) -- provided there exists a function \(h\in L^p(\mm)\) such that
\begin{equation}\label{eq:def_div_der}
-\int\boldsymbol b(f)\,\d\mm=\int fh\,\d\mm
\quad\text{ for every }f\in\Lip_{\rm bs}(X).
\end{equation}
The function \(h\) is uniquely determined, thus it can be unambiguously denoted by
\({\rm div}(\boldsymbol b)\). We set
\[\begin{split}
{\rm Der}^p(X)&:=\big\{\boldsymbol b\in{\rm Der}(X)\;\big|\;
|\boldsymbol b|\in L^p(\mm)\big\},\\
{\rm Der}^{p,p}(X)&:=\big\{\boldsymbol b\in{\rm Der}^p(X)\;\big|\;
{\rm div}(\boldsymbol b)\in L^p(\mm)\big\}
\end{split}\]
for any \(p\in[1,\infty]\). The set \({\rm Der}^p(X)\) is a Banach 
space if endowed with the norm
\(\|\boldsymbol b\|_p:=\||\boldsymbol b|\|_{L^p(\mm)}\).
\begin{remark}\label{rmk:div_local}
		We claim that for every \(\boldsymbol b\in{\rm Der}^{p,p}(X)\) --
		where \(p\in[1,\infty]\) -- it holds that
		\begin{equation}\label{eq:bound_supp_div}
		{\rm supp}\big({\rm div}(\boldsymbol b)\big)\subset{\rm supp}(\boldsymbol b).
		\end{equation}
		In order to prove it, fix any open bounded subset \(U\) of
		\(X\setminus{\rm supp}(\boldsymbol b)\). Then formula \eqref{eq:def_div_der} guarantees
		that \(\int f\,{\rm div}(\boldsymbol b)\,\d\mm=-\int\boldsymbol b(f)\,\d\mm=0\) for
		every \(f\in\Lip_{\rm bs}(U)\), whence accordingly \({\rm div}(\boldsymbol b)=0\)
		holds \(\mm\)-a.e.\ on \(U\). By arbitrariness of \(U\), we conclude that
		the claim \eqref{eq:bound_supp_div} is verified.
		\end{remark}
    In the next proposition the notions of tangent module \(L^2(TX)\) and, more generally,
     of Hilbert \(L^2(\mm)\)-normed
    \(L^\infty(\mm)\)-module, play a role. We will denote by
    \(\nabla\colon H^{1,2}(X)\to L^2(TX)\) the gradient map. 
    We refer to \Cref{subsection:modules} below for the definition of
    these objects.
\begin{proposition}\label{prop:identif_tangent_mod}
	Let \((X,\sfd,\mm)\) an infinitesimally Hilbertian metric measure space.
	Let us denote by \(\overline{\mathbb D}\) the closure in \({\rm Der}^2(X)\)
	of the pre-Hilbert space \(\mathbb D:=\big({\rm Der}^{2,2}(X),\|\cdot\|_2\big)\).
	Then \(\overline{\mathbb D}\) has a natural structure of Hilbert \(L^2(\mm)\)-normed
	\(L^\infty(\mm)\)-module and the map
    \(A\colon L^2(TX)\to\overline{\mathbb D}\), defined as
\[
A(v)(f)\coloneqq v\cdot\nabla f\quad\text{ for every }
v\in L^2(TX)\text{ and }f\in\Lip_{bs}(X).
\]

	is a normed module isomorphism between \(L^2(TX)\) and \(\overline{\mathbb D}\).
	Moreover, it holds \(A\big(D({\rm div})\big)=\mathbb D\) and
	\[
	{\rm div}(A(v))={\rm div}(v)\quad\text{ for every }v\in D({\rm div}).
	\]
\end{proposition}
\begin{proof}
	Cf.\ the proof of \cite[Proposition 6.5]{DMGPS18}.
\end{proof}
\begin{remark}\label{rmk:ineq_Df_deriv}
		Given an infinitesimally Hilbertian space \((X,\sfd,\mm)\)
		and any \(f\in{\rm BV}(X,\sfd,\mm)\), it holds
		\[
		\int f\,{\rm div}(v)\,\d\mm\leq|Df|(X)\quad\text{ for every }v\in D({\rm div})
		\text{ with }|v|\leq 1\;\;\mm\text{-a.e.\ and }{\rm div}(v)\in L^\infty(\mm).
		\]
		Such inequality readily follows from \cite[Theorem 3.3]{DM14} and \Cref{prop:identif_tangent_mod}.
		\end{remark}
\begin{theorem}[Representation formula for \(|Df|\)]\label{thm:repr_tv}
	Let \((X,\sfd,\mm)\) be an infinitesimally Hilbertian metric measure space.
	Let \(f\in{\rm BV}(X,\sfd,\mm)\) be given. Then for every open set \(U\subset X\)
	it holds that
	\[
	|Df|(U)=\sup\bigg\{\int_U f\,{\rm div}(v)\,\d\mm\;\bigg|\;v\in D({\rm div}),\;|v|\leq 1
	\;\;\mm\text{-a.e.},\;{\rm div}(v)\in L^\infty(\mm),\,{\rm supp}(v)\Subset U\bigg\}.
	\]
\end{theorem}
\begin{proof}
	Combine \cite[Theorem 3.4]{DM14} with \Cref{prop:identif_tangent_mod}
	(recall that we have \(\boldsymbol b\in{\rm Der}^{2,2}(X)\) for every
	\(\boldsymbol b\in{\rm Der}^{\infty,\infty}(X)\) such that \({\rm supp}(\boldsymbol b)\)
	is bounded, thanks to \Cref{rmk:div_local}).
\end{proof}

\subsubsection{PI spaces}
Let us recall that $(X,\sfd,\mm)$ satisfies a weak local \((1,2)\)-Poincar\'{e} inequality with constants \(C_P>0\) and \(\lambda\geq 1\) if it holds
	\begin{equation}\label{eq:Poincare_ineq}
	\fint_{B_r(x)}\big|f-(f)_{x,r}\big|\,\d\mm\leq
	C_P\,r\,\left(\fint_{B_{\lambda r}(x)}|Df|^2\,\d\mm\right)^{\nicefrac{1}{2}}
	\quad\text{ for all }f\in H^{1,2}(X),\,x\in X,\,r>0,
	\end{equation}

where
\begin{equation}\label{eq:average}
	(f)_{x,r}:=\fint_{B_r(x)}f\,\d\mm.
\end{equation}

Before giving the definition of PI space we need to recall the notion of locally doubling m.m.s.: we say that $(X,\sfd,\mm)$ is locally doubling if for any $R>0$ there exists $C_D>0$ depending only on $R$ such that
\begin{equation}\label{eq:locallydoubling}
	\mm(B_{2r}(x))\le C_D\mm(B_r(x))\quad \forall\ 0<r<R,\ x\in X.
\end{equation}

\begin{definition}\label{def:PI spaces}
A \emph{PI space} is a locally doubling metric
measure space supporting a weak local \((1,2)\)-Poincar\'{e} inequality.
\end{definition}
\subsubsection{Capacity and Hausdorff measures}
We briefly recall the notion of capacity and its main properties in this setting, referring
to \cite{DGP19} for a detailed discussion on the topic.
The \emph{capacity} of a given set \(E\subset X\) is defined as
\[
\Cap(E):=\inf\Big\{\|f\|^2_{H^{1,2}(X)}\;\Big|\;f\in H^{1,2}(X,\sfd,\mm),\ f\geq 1\ \mm\text{-a.e.\ on some neighbourhood of }E\Big\}.
\]
It turns out that \(\Cap\) is a submodular outer measure on \(X\),
finite on all bounded sets, such that the inequality \(\mm(E)\leq\Cap(E)\) holds
for any Borel set \(E\subset X\). Any function \(f:\,X\to[0,+\infty]\)
can be integrated with respect to the capacity via Cavalieri's formula:
\[
\int f\,\d\Cap:=\int_0^{+\infty}\Cap\big(\{f>t\}\big)\,\d t.
\]
(The function \(t\mapsto\Cap\big(\{f>t\}\big)\) is non-increasing,
thus in particular it is Lebesgue measurable.) The integral operator
\(f\mapsto\int f\,\d\Cap\) is subadditive as a consequence of the submodularity
of \(\Cap\). Given any set \(E\subset X\), we shall use the shorthand
notation \(\int_E f\,\d\Cap:=\int\nchi_E f\,\d\Cap\).
\medskip

Let us now introduce the \emph{codimension-\(\alpha\) Hausdorff measure}. We refer to \cite{A02} for a more detailed introduction to this topic.

\begin{definition}\label{def:codimension measure}
	Given a locally doubling metric measure space $(X,\sfd,\mm)$, for any $\alpha>0$ we set
	\begin{equation*}
	h_{\alpha}(B_r(x)):=\frac{\mm(B_r(x))}{r^{\alpha}}
	\qquad \text{for any}\ x\in X,\ r\in (0,\infty).
	\end{equation*}
	The codimension-\(\alpha\) Hausdorff measure $\haus^{h_{\alpha}}$ is the Borel regular outer measure built from $h_{\alpha}$ with the Carath\'{e}odory construction. We will denote by  $\haus_{\delta}^{h_{\alpha}}$ the pre-measure with parameter $\delta$.
\end{definition}
The codimension-$1$ measure plays a crucial role in the theory of sets of finite perimeter over PI spaces, since $\Per(E,\cdot)\ll \haus^{h_1}$ for any set of finite perimeter $E$. This result has been proved by Ambrosio in \cite[Lemma 5.2]{A02}.
\begin{lemma}\label{lemma:AssolutacontinuitaHauss}
	Let $(X,\sfd,\mm)$ be a PI space. For any set of locally finite perimeter $E\subset X$ it holds
	\begin{equation*}
		\haus^{h_1}(B)=0\implies \Per(E,B)=0
		\quad\text{for any Borel set }B\subset X.
	\end{equation*}
\end{lemma}
Let us now prove two results connecting the \emph{codimension-\(\alpha\) Hausdorff measure} and the capacity. Their proofs are inspired by those given for the analogous results in the Euclidean setting in \cite{EvansGariepy}.
\begin{lemma}\label{lem:H_theta_null}
Let \((X,\sfd,\mm)\) be a locally doubling m.m.s.. Let
\(f\in L^1(X,\mm)\), \(f\geq 0\) be given. Then for any exponent \(\alpha>0\)
it holds that
\[
\mathscr H^{h_\alpha}(\Lambda_\alpha)=0,\quad\text{ where we set }
\Lambda_\alpha:=\Big\{x\in X\;\Big|\;
\limsup_{r\searrow 0}r^\alpha(f)_{x,r}>0\Big\}.
\]
\end{lemma}
\begin{proof}
By Lebesgue differentiation theorem we know that the limit
\(\lim_{r\searrow 0}(f)_{x,r}\) exists and is finite for
\(\mm\)-a.e.\ \(x\in X\), thus for any \(\alpha>0\) we have that
\(\limsup_{r\searrow 0}r^\alpha(f)_{x,r}=0\) holds for \(\mm\)-a.e.\ \(x\in X\).
This means that \(\mm(\Lambda_\alpha)=0\). Calling
\[
\Lambda_\alpha^k:=\Big\{x\in X\;\Big|\;
\limsup_{r\searrow 0}r^\alpha(f)_{x,r}\geq 1/k\Big\}\quad\text{ for every }k\in\N,
\]
we see that \(\Lambda_\alpha=\bigcup_k\Lambda_\alpha^k\),
thus in particular \(\mm(\Lambda_\alpha^k)=0\) for every \(k\in\N\).
Given that \(f\in L^1(X,\mm)\), for any \(\eps>0\) there exists \(\delta>0\)
such that \(\int_A f\,\d\mm\leq\eps\) for any Borel set \(A\subset X\)
satisfying \(\mm(A)<\delta\). Fix \(k\in\N\) and pick an open set
\(U\subset X\) such that \(\Lambda_\alpha^k\subset U\) and \(\mm(U)<\delta\).
Let us define
\[
\mathcal F:=\bigg\{B_r(x)\;\bigg|\;x\in\Lambda_\alpha^k,\;r\in(0,\eps),\;
B_r(x)\subset U,\;\int_{B_r(x)}f\,\d\mm\geq\mm\big(B_r(x)\big)/(r^\alpha k)\bigg\}.
\]
Therefore by applying the Vitali covering theorem we can find a sequence
\((B_i)_{i\in\N}\subset\mathcal F\) of pairwise disjoint balls
\(B_i=B_{r_i}(x_i)\) such that \(\Lambda_\alpha^k\subset\bigcup_i B_{5 r_i}(x_i)\).
Being \(\mm\) locally doubling, there exists a constant \(C_D\geq 1\) such that
\(\mm\big(B_{5r}(x)\big)\leq C_D\,\mm\big(B_r(x)\big)\) holds for every \(x\in X\)
and \(r<\eps\). Consequently, one has that
\[\begin{split}
\mathscr H^{h_\alpha}_{10\eps}(\Lambda_\alpha^k)
&\leq\frac{1}{5^\alpha}\sum_{i=1}^\infty\frac{\mm\big(B_{5r_i}(x_i)\big)}{r_i^\alpha}
\leq\frac{C_D}{5^\alpha}\sum_{i=1}^\infty\frac{\mm(B_i)}{r_i^\alpha}
\leq\frac{C_D k}{5^\alpha}\sum_{i=1}^\infty\int_{B_i}f\,\d\mm
\leq\frac{C_D k}{5^\alpha}\int_U f\,\d\mm\\
&\leq\frac{C_D k}{5^\alpha}\eps.
\end{split}\]
By letting \(\eps\searrow 0\) we conclude that
\(\mathscr H^{h_\alpha}(\Lambda_\alpha^k)=0\), whence
\(\mathscr H^{h_\alpha}(\Lambda_\alpha)=\lim_k\mathscr H^{h_\alpha}(\Lambda_\alpha^k)=0\).
\end{proof}

\begin{theorem}\label{thm:H_ac_Cap}
Let \((X,\sfd,\mm)\) be a PI space. Then it holds that
\(\mathscr H^{h_\alpha}\ll\Cap\) for every \(\alpha\in(0,2)\).
\end{theorem}
\begin{proof}
Fix \(\alpha\in(0,2)\) and a set \(A\subset X\) with \(\Cap(A)=0\).
We aim to prove that \(\mathscr H^{h_\alpha}(A)=0\). By definition of
capacity, we can find a sequence \((f_i)_i\subset H^{1,2}(X)\)
such that \(f_i\geq 1\) on some neighbourhood of \(A\) and
\({\|f_i\|}_{H^{1,2}(X)}\leq 1/2^i\) for every \(i\in\N\).
Since \(\sum_{i=1}^\infty{\|f_i\|}_{H^{1,2}(X)}<+\infty\), one has that
\(g:=\sum_{i=1}^\infty f_i\) is a well-defined element of the Banach
space \(H^{1,2}(X)\). For any \(k\in\N\) it clearly holds that
\(g\geq k\) on some neighbourhood of \(A\), whence for any \(x\in A\)
we have \((g)_{x,r}\geq k\) for every \(r<{\rm dist}\big(x,\{g<k\}\big)\)
and accordingly
\begin{equation}\label{eq:H_ac_Cap_claim1}
\lim_{r\searrow 0}(g)_{x,r}=+\infty\quad\text{ for every }x\in A.
\end{equation}
Furthermore, we claim that
\begin{equation}\label{eq:H_ac_Cap_claim2}
\limsup_{r\searrow 0}r^\alpha\fint_{B_r(x)}|Dg|^2\,\d\mm=+\infty
\quad\text{ for every }x\in A.
\end{equation}
In order to prove it, we argue by contradiction: suppose that
\(\limsup_{r\searrow 0}r^\alpha\fint_{B_r(x)}|Dg|^2\,\d\mm<+\infty\)
for some \(x\in A\), so that there exists a constant \(M>0\) such that
\begin{equation}\label{eq:H_ac_Cap_claim3}
r^\alpha\fint_{B_r(x)}|Dg|^2\,\d\mm\leq M\quad\text{ for every }r\in(0,1).
\end{equation}
Call \(C_D\) the doubling constant of \(\mm\) (for \(r<1/2\)).
Therefore, for every \(r<1/(2\lambda)\) we have that
\[\begin{split}
\big|(g)_{x,r}-(g)_{x,2r}\big|
&\overset{\phantom{\eqref{eq:H_ac_Cap_claim3}}}=
\frac{1}{\mm\big(B_r(x)\big)}\bigg|\int_{B_r(x)}g-(g)_{x,2r}\,\d\mm\bigg|\\
&\overset{\phantom{\eqref{eq:H_ac_Cap_claim3}}}\leq
C_D\fint_{B_{2r}(x)}\big|g-(g)_{x,2r}\big|\,\d\mm\\
&\overset{\eqref{eq:Poincare_ineq}}\leq
2\,C_D\,C_P\,r\,\bigg(\fint_{B_{2\lambda r}(x)}|Dg|^2\,\d\mm\bigg)^{\nicefrac{1}{2}}\\
&\overset{\eqref{eq:H_ac_Cap_claim3}}\leq
(2^{1-\nicefrac{\alpha}{2}}\,C_D\,C_P\,\lambda^{-\nicefrac{\alpha}{2}}\,
M^{\nicefrac{1}{2}})\,r^{1-\nicefrac{\alpha}{2}}.
\end{split}\]
Let us set
\(C:=2^{1-\nicefrac{\alpha}{2}}\,C_D\,C_P\,
\lambda^{-\nicefrac{\alpha}{2}}\,M^{\nicefrac{1}{2}}\)
and \(\theta:=1-\alpha/2\in(0,1)\). Then the previous computation gives
\(\sum_{i=2}^\infty\big|(g)_{x,2^{-i}}-(g)_{x,2^{-i+1}}\big|\leq
C\sum_{i=2}^\infty(2^\theta)^{-i}<+\infty\), contradicting
\eqref{eq:H_ac_Cap_claim1}. This proves \eqref{eq:H_ac_Cap_claim2}.

Finally, it immediately follows from \eqref{eq:H_ac_Cap_claim2} that \(A\)
is contained in the set of all points \(x\in X\) that satisfy
\(\limsup_{r\searrow 0}r^\alpha\fint_{B_r(x)}|Dg|^2\,\d\mm>0\), which
is \(\mathscr H^{h_\alpha}\)-negligible by \Cref{lem:H_theta_null}.
Therefore, we conclude that \(\mathscr H^{h_\alpha}(A)=0\), thus completing
the proof of the statement.
\end{proof}
\subsection{\texorpdfstring{$\RCD$}{RCD} metric measure spaces}\label{section: preliminaries RCD}
The main object of our investigation in this note are $\RCD(K,N)$ metric measure spaces, that is infinitesimally Hilbertian spaces satisfying a lower Ricci curvature bound and an upper dimension bound in synthetic sense according to \cite{Sturm06I,Sturm06II,Lott-Villani09}. Before passing to the description of the main properties of $\RCD(K,N)$ spaces that will be relevant for the sake of this note, let us briefly focus on the adimensional case.

The class of $\RCD(K,\infty)$  spaces was introduced in \cite{AmbrosioGigliSavare12} (see also \cite{AmbrosioGigliMondinoRajala12} for the extension to the case of $\sigma$-finite reference measures) adding to the $\CD(K,\infty)$ condition, formulated in terms of $K$-convexity properties of the logarithmic entropy over the Wasserstein space $(\P_{2},W_2)$, the infinitesimal Hilbertianity assumption. 

Under the $\RCD(K,\infty)$ condition it was proved that the dual heat semigroup $P_t^{*}:\P_2(X)\to\P_2(X)$, defined by
\begin{equation*}
\int_Xf\,\d P_t^*\mu=\int_XP_tf\,\d\mu\qquad\forall\mu\in\P_2(X),\quad\forall f \in \Lip_{\rm bs}(X,\sfd)
\end{equation*}
is $K$-contractive w.r.t.\ the $W_2$-distance and, for $t>0$, maps probability measures into probability measures absolutely continuous w.r.t.\ $\mm$. Then, for any $t>0$, it is possible to define the \textit{heat kernel} $p_t:X\times X\to[0,+\infty)$ by 
\begin{equation}\label{eq:heat kernel}
p_t(x,\cdot)\mm=P_t^*\delta_x.
\end{equation}
We go on stating a few regularization properties of $\RCD(K,\infty)$ spaces, referring again to \cite{AmbrosioGigliSavare12,AmbrosioGigliMondinoRajala12} for a more detailed discussion and for the proofs of these results.

First we have the \textit{Bakry-\'Emery} contraction estimate:
\begin{equation}\label{eq:BE2}
\abs{\nabla P_tf}^2\le e^{-2Kt}P_t\abs{\nabla f}^2\quad \mm\text{-a.e.,}
\end{equation}
for any $t>0$ and for any $f\in H^{1,2}(X,\sfd,\mm)$. This contraction estimate can be generalized to the whole range of exponents $1<p<+\infty$. Furthermore in \cite{GigliHan14} it has been proved that on any proper $\RCD(K,\infty)$ m.m.s.\ it holds
\begin{equation}\label{eq:BE1}
\abs{DP_t f}\le e^{-Kt}P_t^*\abs{Df},
\end{equation}
for any $t>0$ and for any $f\in{\rm BV}(X,\sfd,\mm)$.

Next we have the so-called \textit{Sobolev-to-Lipschitz} property, stating that any $f\in H^{1,2}(X,\sfd,\mm)$ such that $\abs{\nabla f}\in L^{\infty}(X,\mm)$ admits a representative $\tilde{f}\in \Lip(X,\sfd)$ with Lipschitz constant bounded from above by $\norm{\abs{\nabla f}}_{L^{\infty}}$.

Let us introduce the space of test functions $\Test(X,\sfd,\mm)$ following \cite{Gigli14}:
\begin{equation}\label{eq:test}
\Test(X,\sfd,\mm):=\{f\in D(\Delta)\cap L^{\infty}(X,\mm): \abs{\nabla f}\in L^{\infty}(X,\mm)\text{ and }\Delta f\in H^{1,2}(X,\sfd,\mm) \}.
\end{equation}

The notion of $\RCD(K,N)$ m.m.s.\ was proposed and extensively studied in \cite{Gigli12,AmbrosioMondinoSavare13,Erbar-Kuwada-Sturm13} (see also \cite{CavallettiMilman16} for the equivalence between the $\RCD$ and the $\RCD^*$ condition\footnote{In \cite{CavallettiMilman16} the case of finite measure is considered but, due to the local nature of their arguments, it is thought that the identification result extends to the general case.}), as a finite dimensional counterpart to $\RCD(K,\infty)$ m.m.s.\ which were introduced and firstly studied in \cite{AmbrosioGigliSavare12}. Here we just recall that they can be characterized asking for the quadraticity of $\Ch$, the volume growth condition $\mm(B_r(x))\le c_1\exp(c_2r^2)$ for some (and thus for all) $x\in X$, the validity of the \textit{Sobolev-to-Lipschitz property} and of a weak form of Bochner's inequality
\begin{equation*}
\frac{1}{2}\Delta\abs{\nabla f}^2-\nabla f\cdot\nabla\Delta f\ge \frac{\left(\Delta f\right)^2}{N}+K\abs{\nabla f}^2 
\quad\text{for any $f\in \Test(X,\sfd,\mm)$.}
\end{equation*}   
We refer to \cite{AmbrosioMondinoSavare13,Erbar-Kuwada-Sturm13} for a more detailed discussion and equivalent characterizations of the $\RCD(K,N)$ condition.

Note that, if $(X,\sfd,\mm)$ is an $\RCD(K,N)$ m.m.s., then so is $(\supp\, \mm,\sfd,\mm)$, hence in the following we will always tacitly assume $\supp\, \mm = X$.

We recall that any $\RCD(K,N)$ m.m.s.\ $(X,\sfd,\mm)$ satisfies the Bishop-Gromov inequality:
\begin{equation}\label{eq:BishopGromovInequality}
\frac{\mm(B_R(x))}{v_{K,N}(R)}\le\frac{\mm(B_r(x))}{v_{K,N}(r)}
\quad \text{for any $0<r<R$ and $x\in X$},
\end{equation}
where $v_{K,N}(r)$ is the volume of the ball with radius $r$ in the model space with dimension $N$ and Ricci curvature $K$. We refer to \cite[Theorem 30.11]{Villani09} for the proof of \eqref{eq:BishopGromovInequality}.
In particular $(X,\sfd,\mm)$ is locally uniformly doubling. Furthermore, it was proved in \cite{Rajala12} that it satisfies a local Poincar\'{e} inequality. Therefore $\RCD(K,N)$ spaces fit in the framework of PI spaces that we introduced above. 
% %

We assume the reader to be familiar with the notion of pointed measured Gromov-Hausdorff convergence (pmGH-convergence for short), referring to \cite[Chapter 27]{Villani09} for an overview on the subject. 

\begin{remark}\label{remark:stability}
	A fundamental property of $\RCD(K,N)$ spaces, that will be used several times in this paper, is the stability w.r.t.\ pmGH-convergence, meaning that a pmGH-limit of a sequence of (pointed) $\RCD(K_n,N_n)$ spaces for some $K_n\to K$ and $N_n\to N$ is an $\RCD(K,N)$ m.m.s..  
\end{remark}

Let us finally recall the construction of good cut-off functions over $\RCD(K,N)$ metric measure spaces, see \cite[Lemma 3.1]{Mondino-Naber14} for a proof.

\begin{lemma}\label{lemma:good cut-off}
	Let $(X,\sfd,\mm)$ be an $\RCD(K,N)$ m.m.\ space. For any $0<2r<R$ and $x\in X$ there exists a test function $\eta:X\to \R$ satisfying
	\begin{itemize}
		\item[(i)] $0\le \eta\le 1$ on $X$, $\eta=1$ on $B_r(x)$ and $\eta=0$ on $X\setminus B_{2r}(x)$;
		\item[(ii)] $r^2|\Delta \eta|+r|\nabla \eta|\le C_{N,K,R}$.
	\end{itemize}
\end{lemma}

\subsubsection{Structure theory}
Let us briefly review the main results concerning the state of the art about the so-called structure theory of $\RCD(K,N)$ spaces.

Given a m.m.s.\ $(X,\sfd,\mm)$, $x\in X$ and $r\in(0,1)$, we consider the rescaled and normalized pointed m.m.s.\ $(X,r^{-1}\sfd,\mm_r^{x},x)$, where 
\begin{equation*}
\mm_r^x:= \left( \int_{B_r(x)} \left(1-\frac{\sfd(x,y)}{r}\right) \d \mm(y)\right)^{-1}\mm=C(x,r)^{-1}\mm.
\end{equation*}

\begin{definition}
	We say that a pointed m.m.s.\ $(Y,\sfd_Y,\eta,y)$ is tangent to $(X,\sfd,\mm)$ at $x$ if there exists a sequence $r_i\downarrow 0$ such that $(X,r_i^{-1}\sfd,\mm_{r_i}^x,x)\rightarrow(Y,\sfd_Y,\eta,y)$ in the pmGH-topology. The collection of all the tangent spaces of $(X,\sfd,\mm)$ at $x$ is denoted by $\Tan_x(X,\sfd,\mm)$.
\end{definition}

A compactness argument, which is due to Gromov, together with the rescaling and stability properties of the $\RCD(K,N)$ condition (see \Cref{remark:stability}), yields that $\Tan_x(X,\sfd,\mm)$ is non-empty for every $x\in X$ and its elements are all $\RCD(0,N)$ pointed m.m.\ spaces.\\
Let us recall below the notion of $k$-regular point and $k$-regular set. 

\begin{definition}\label{def:regular point}
	Given any natural $1\le k\le N$, we say that $x\in X$ is a $k$-regular point if
	\begin{equation*}
		\Tan_x(X,\sfd,\mm)=\left\lbrace (\R^k,\sfd_{eucl},c_k\Leb^k,0)  \right\rbrace.
	\end{equation*}
	We shall denote by $\mathcal{R}_k$ the set of $k$-regular points in $X$.
\end{definition}
Observe that, by explicit computation, in \Cref{def:regular point} the
constant $c_k$ equals $\frac{\omega_k}{k+1}$.
	\begin{remark}\label{remark: limit of different normalization}
		Observe that, if $x\in \mathcal{R}_k$, then one has
		\begin{equation}\label{eq:asympotic of different normalization}
		\lim_{r\to 0} \frac{\int_{B_r(x)}\left( 1-\frac{\sfd(x,y)}{r}\right)\, \d \mm(y)}{\mm(B_r(x))}=\frac{1}{k+1}.
		\end{equation}
		Moreover it can be easily checked that $x\in \mathcal{R}_k$ if and only if
		\begin{equation*}
		\lim_{r\to 0}\sfd_{pmGH}\left( \left(X,r^{-1}\sfd,\frac{\mm}{\mm(B_r(x))},x\right), \left( \R^k,\sfd_{\text{eucl}}, \frac{1}{\omega_k}\Leb^k, 0^k\right)\right)=0.
		\end{equation*}
	\end{remark}
	
After the works \cite{Gigli-Mondino-Rajala15,Mondino-Naber14,GP16-2,DPMR16,MK16} and \cite{BS18} we have the following structure theorem for $\RCD(K,N)$ spaces.

\begin{theorem}\label{thm:structure theory}
	Let $(X,\sfd,\mm)$ be an $\RCD(K,N)$ m.m.s.\ with $K\in\R$ and $N\ge 1$. Then there exists a natural number $1\le n\le N$, called essential dimension of $X$, such that $\mm(X\setminus \mathcal{R}_n)=0$. Moreover $\mathcal{R}_n$ is $(\mm,n)$-rectifiable and $\mm$ is representable as $\theta\haus^n\res {\mathcal{R}_n}$ for some nonnegative density $\theta\in L^1_{\rm loc}(X,\haus^n\res\mathcal{R}_n)$.
\end{theorem}
	Recall that $X$ is said to be $(\mm,n)$-rectifiable if there exists a family $\left\lbrace A_i \right\rbrace_{i\in\N}$ of Borel subsets of $X$ such that each $A_i$ is bi-Lipschitz to a Borel subset of $\R^n$ and $\mm(X\setminus \cup_{i\in\N}A_i)=0$.

\subsubsection{Sobolev functions and Laplacian on balls}\label{subsubsection:laplacian on balls}
Following a standard approach let us give a notion of Sobolev functions and Laplacian on balls, we refer to \cite{AmbrosioHonda18} for more detailed presentation.

 We define the space $H_0^{1,2}(B_r(x),\sfd,\mm)$ considering the closure of $\Lip_{\rm c}(B_r(x),\sfd)$ in $H^{1,2}(X,\sfd,\mm)$. Let us also define $H_{{\rm loc}}^{1,2}(B_r(x),\sfd,\mm)$ as the space of those $f:B_r(x)\to\R$ such that $\eta f\in H^{1,2}(X,\sfd,\mm)$ for any $\eta\in \Lip_c(B_r(x),\sfd)$. Exploiting the locality of the minimal relaxed slope one can easily define $|\nabla f|$ for any $f\in H_{{\rm loc}}^{1,2}(B_r(x),\sfd,\mm)$. This allows us to introduce $H^{1,2}(B_r(x),\sfd,\mm)$ as the space of $f\in  H_{{\rm loc}}^{1,2}(B_r(x),\sfd,\mm)$ such that $f, |\nabla f|\in L^2(X,\mm)$.
 
\begin{definition}\label{def:laplacian on balls}
	A function $f\in H^{1,2}(B_r(x),\sfd,\mm)$ belongs to $D(\Delta, B_r(x))$ if there exists $g\in L^2(B_r(x),\mm)$ satisfying
	\begin{equation*}
		\int_{B_r(x)} \nabla f\cdot \nabla  h\, \d \mm=-\int_{B_r(x)} f\, g\, \d \mm
		\quad \text{for any }h\in H_0^{1,2}(B_r(x),\sfd,\mm).
	\end{equation*}
	With a slight abuse of notation we write $\Delta f=g$ in $B_r(x)$.
\end{definition}
It is easily seen that, if $f\in D(\Delta,B_r(x))$ and $\eta\in \Lip_c(B_r(x),\sfd)\cap D(\Delta)$, $\Delta \eta\in L^{\infty}(X,\mm)$ then $\eta f\in D(\Delta)$.

\subsubsection{Stability and convergence results}\label{subsubsection:stability results}
Let us fix a pointed measured Gromov-Hausdorff convergent sequence
\begin{equation*}
	(X_i,\sfd_i,\mm_i,x_i)\to (Y,\varrho,\mu,y)
\end{equation*}
of $\RCD(K,N)$ m.m.\ spaces. Recall that, in the setting of uniformly locally doubling spaces, the pointed measured Gromov-Hausdorff convergence can be equivalently characterized asking for the existence of a proper metric space $(Z,\sfd_Z)$ where $(X_i,\sfd_i)$ and $(Y,\varrho)$ are isometrically embedded, $x_i\to y$ and $\mm_i\to\mu$ in duality with $\Cbs(Z)$ (the space of continuous functions with bounded supports in $Z$). This is the so-called extrinsic approach and it is convenient to formulate various notions of convergence.

\begin{definition}\label{def:convpuntunif}
	Let $(X_i, \sfd_i, \mm_i, x_i)$, $(Y,\varrho,\mu, y)$, $(Z,\sfd_Z)$ be as above and $f_i:X_i\to\R$, $f:Y\to\R$. We say that $f_i\to f$ pointwise if
	$f_i(z_i)\to f(z)$ for every sequence of points $z_i\in X_i$ such that $z_i\to z$ in $Z$.
	If for every $\varepsilon>0$ there exists $\delta>0$ such that
	$\abs{f_i(z_i)-f(z)}\le\varepsilon$
	for every $i\ge\delta^{-1}$ and $z_i\in X_i$, $z\in Y$ with $\sfd_Z(z_i,z)\le\delta$,
	then we say that $f_i\to f$ uniformly.
\end{definition}
The next proposition is a version of the Ascoli--Arzel\`{a} compactness theorem for sequences of functions defined on varying spaces. We omit the proof, that can be obtained arguing as in the case of a fixed space.
\begin{proposition}\label{thm:AscoliArzela}
	Let $(X_i, \sfd_i,\mm_i,x_i)$ and $(Y,\rho,\mu,y)$ be as above and $R>0$, $L>0$ fixed.
	Then for any sequence of $L$-Lipschitz functions $f_i:B_R(x_i)\to \R$
such that \(\sup_i\big|f_i(x_i)\big|<+\infty\)	
	there exists a subsequence that converges uniformly to some $L$-Lipschitz function $f:B_R(y)\to \R$.
\end{proposition}

We recall below the notions of convergence in $L^p$ and Sobolev spaces for functions defined over converging sequences of metric measure spaces. We will be concerned only with the cases $p=2$ and $p=1$ in the rest of the note. We refer again to \cite{AmbrosioHonda17,AmbrosioHonda18} for a more general treatment and the proofs of the results we state below.

\begin{definition}\label{def:L2convergence}
	We say that $f_i\in L^2(X_i,\mm_i)$ converge in $L^2$-weak to $f\in L^2(Y,\mu)$ if $f_i\mm_i\to f\mu$ in duality with $\Cbs(Z)$ 
	and $\sup_i\norm{f_i}_{L^2(X_i,\mm_i)}<+\infty$.
	
	We say that $f_i\in L^2(X_i,\mm_i)$ converge in $L^2$-strong to $f\in L^2(Y,\mu)$ if $f_i\mm_i\to f\mu$ in duality with $\Cbs(Z)$ 
	and $\lim_i\norm{f_i}_{L^2(X_i,\mm_i)}=\norm{f}_{L^2(Y,\mu)}$.
\end{definition}

\begin{definition}\label{def:L1convergence}
	We say that a sequence $(f_i)\subset L^1(X_i,\mm_i)$ converges $L^1$-strongly to $f\in L^1(Y,\mu)$ if 
	\begin{equation*}
	\sigma\circ f_i\,\mm_i\to \sigma\circ f\mu
	\qquad
	\text{and} 
	\qquad
	\int_{X_i}|f_i|\,\d\mm_i\to \int_Y |f|\,\d\mu,
	\end{equation*}
	where $\sigma(z):={\rm sign}(z)\sqrt{|z|}$ and the weak convergence is understood in duality with $\Cbs(Z)$. Equivalently, if $\sigma\circ f_i$ $L^2$-strongly converge to $\sigma\circ f$.
\end{definition}

Dealing with characteristic functions one has the following equivalent notion of $L^1$-convergence.

\begin{definition}\label{def:L1 convergence of sets}
	We say that a sequence of Borel sets $E_i\subset X_i$ such that $\mm_i(E_i)<\infty$ for any $i\in\N$ converges in $L^1$-strong to a Borel set $F\subset Y$ with $\mu(F)<\infty$ if $\nchi_{E_i}\mm_i\to \nchi_F\mu$ in duality with $\Cbs(Z)$ and $\mm_i(E_i)\to \mu(F)$.
	
	We also say that a sequence of Borel sets $E_i\subset X_i$ converges in $L^1_{{\rm loc}}$ to a Borel set $F\subset Y$ if $E_i\cap B_R(x_i)\to F\cap B_R(y)$ in $L^1$-strong for any $R>0$.
\end{definition}

\begin{remark}\label{remark:convergence of sets}
	It follows from the very definition of $L^1$-convergence that, if a sequence of sets $E_i\to F$ in $L^1$, then
	$\nchi_{E_i}\to \nchi_F$ in $L^2$-strong.
\end{remark}

\begin{definition}\label{def:localBVconvergence}
    We say that a sequence of sets with locally finite perimeter $E_i\subset X_i$ converges locally strongly in $\text{\rm BV}$ to a set of locally finite perimeter $F\subset Y$ if $E_i\to F$ in $L^1_{\rm loc}$ and $|D\nchi_{E_i}|\to |D\nchi_F|$ in duality with $\Cbs(Z)$.
\end{definition}

A proof of the technical result below can be found in \cite{AmbrosioHonda17}.
\begin{proposition}\label{prop:sum L1 convergence}
	Let us fix $p=1,2$.
	\begin{itemize}
		\item[(i)] 	For any $f_i, g_i\in L^p(X_i,\mm_i)$ such that $f_i\to f\in L^p(Y,\mu)$ and $g_i\to g\in L^p(Y,\mu)$ strongly in $L^p$ one has $f_i+g_i\to f+g$ strongly in $L^p$.
		\item[(ii)] If $f_i\to f$ and $g_i\to g$ in $L^2$-strong then $f_ig_i\to fg$ in $L^1$-strong.
		\item[(ii)] If $f_i\to f$ in $L^1$-strong and $\sup_{i\in\N} \norm{f_i}_{L^{\infty}(X_i,\mm_i)}<\infty$
		then $\norm{f_i}_{L^2(X_i,\mm_i)}\to \norm{f}_{L^2(Y,\mu)}$. In particular $f_i\to f$ in $L^2$-strong.
	\end{itemize}
\end{proposition}

Let us present a compactness result for sets with finite perimeter that is partially taken from \cite{ABS18}.
\begin{proposition}\label{prop:compactness sets with finite perimeter}
	Let $E_i\subset X_i$ be sets of finite perimeter satisfying
	\begin{equation*}
	\sup_{i\in\N}	\Per(E_i, B_1(x_i))<\infty.
	\end{equation*}
	Then there exists $F\subset Y$ of finite perimeter such that, up to extract a subsequence, $E_i\cap B_1(x_i)\to F\cap B_1(y)$ in $L^1$-strong and
	\begin{equation}\label{z25}
		\liminf_{i\to\infty} \int g \,\d |D\nchi_{E_i}|\ge \int g \,\d |D\nchi_F|,
		\quad\text{for any $g\in C(Z)$, nonnegative with $\supp(g)\subset \bar B_{1/2}(y)$}.
	\end{equation}
	If we further assume that  
	\begin{equation}\label{z27}
		\lim_{i\to\infty}\abs{D\nchi_{E_i}}(B_{1/2}(x_i))=\abs{D\nchi_F}(B_{1/2}(y)),
	\end{equation}
	then \eqref{z25} improves to
	\begin{equation}\label{z26}
	\lim_{i\to\infty} \int g \,\d |D\nchi_{E_i}|= \int g \,\d |D\nchi_F|,
	\quad\text{for any $g\in C(Z)$ with $\supp(g)\subset B_{1/2}(y)$}.
	\end{equation}
\end{proposition}

\begin{proof}
	The convergence $E_i\cap B_1(x_i)\to F\cap B_1(y)$ in $L^1$-strong up to subsequence can be obtained arguing as in the proof of \cite[Corollary 3.4]{ABS18}. 
	
	Inequality \eqref{z25} follows from \cite[Proposition 3.6]{ABS18} along with a localization argument that we sketch briefly. For any $i\in \N$, using \Cref{lemma:good cut-off} we build a good cut-off function $\eta_i\in \Lip(X_i,\sfd_i)$ satisfying $\eta_i=1$ in $B_{1/2}(x_i)$ and $\eta_i=0$ in $X_i\setminus B_{3/4}(x_i)$. By \Cref{thm:AscoliArzela}, up to extract a subsequence, we can assume that $\eta_i\to \eta_{\infty}\in \Lip(Y,\rho)$ uniformly and in $L^2$-strong. It is easily seen that $\eta_{\infty}=1$ in $B_{1/2}(y)$ and $\eta_{\infty}=0$ in $Y\setminus B_1(y)$. The sequence $(\eta_i\nchi_{E_i})_{i}$ satisfies
	\begin{equation*}
		\eta_i\nchi_{E_i}\to \eta_{\infty}\nchi_F
		\ \text{in $L^1$-strong}
		\quad\text{and}
		\quad
		\sup_{i\in\N}|D (\eta_i\nchi_{E_i})|(X_i)<\infty,
	\end{equation*}
	thanks to \Cref{prop:sum L1 convergence}(ii) and standard calculus rules. Applying \cite[Proposition 3.6]{ABS18} to the sequence $(\eta_i \nchi_{E_i})_i$ we get \eqref{z25}.
	
	Inequality \eqref{z26}) is a weak convergence result in the ball $B_{1/2}(y)\subset Z$, which can be proved arguing as in the proof of \cite[Corollary 3.7]{ABS18} taking into account \eqref{z25} and \eqref{z27}.
\end{proof}

Let us now introduce a notion of $H^{1,2}$-convergence along with its local counterpart.

\begin{definition}\label{def:H12convergence}
	We say that $f_i\in H^{1,2}(X_i,\sfd_i,\mm_i)$ are weakly convergent to $f\in H^{1,2}(Y,\varrho,\mu)$ if they converge in $L^2$-weak and $\sup_i\Ch^i(f_i)<+\infty$. Strong $H^{1,2}$-convergence is defined asking that $f_i$ converge to $f$ in $L^2$-strong and $\lim_i\Ch^i(f_i)=\Ch(f)$.  
\end{definition}

\begin{definition}\label{def:H12convergencelocal}
	We say that $f_i\in H^{1,2}(B_R(x_i),\sfd_i,\mm_i)$ are weakly convergent in $H^{1,2}$ to $f\in H^{1,2}(B_R(y),\varrho,\mu)$ on $B_R(y)$ if $f_i$ are $L^2$-weakly (or $L^2$-strongly, equivalently) to $f$ on $B_R(y)$ with $\sup_{i\in\N}\norm{f_i}_{H^{1,2}}<\infty$. Strong convergence in $H^{1,2}$ on $B_R(y)$ is defined by requiring 
	\begin{equation*}
			\lim_{i\to\infty} \int_{B_R(x_i)}|\nabla f_i|^2\,\d \mm_i=\int_{B_R(y)}|\nabla f|^2\,\d \mu.
	\end{equation*}
\end{definition}

Let us now collect results from \cite{AmbrosioHonda18} that will play a role in this note.

\begin{lemma}[{\cite[Lemma 2.10]{AmbrosioHonda18}}]\label{lemma:approxWithLipschitz}
	For any $f\in \Lip_c(B_R(y),\varrho)$ there exist $f_i\in \Lip_c(B_R(x_i), \sfd_i)$ satisfying 
	\begin{equation*}
		\sup_{i\in\N}\norm{|\nabla f_i|}_{L^{\infty}(X_i,\mm_i)}<\infty
	\end{equation*}
    and strongly convergent to $f$ in $H^{1,2}$.   
\end{lemma}

\begin{theorem}[{\cite[Theorem 4.4]{AmbrosioHonda18}}]\label{thm:stabilityLaplacian}
	Let $f_i\in D(\Delta, B_R(x_i))$  with
	\begin{equation*}
		\sup_{i\in\N} \int_{B_R(x_i)}(|f_i|^2+|\nabla f_i|^2+(\Delta f_i)^2)\,\d \mm_i<\infty,
	\end{equation*}
	and let $f$ be an $L^2$-strong limit function of $f_i$ on $B_R(y)$.
	Then:
	\begin{itemize}
		\item[(i)] $f\in D(\Delta, B_R(y))$;
		\item[(ii)] $\Delta f_i\to \Delta f$ on $B_R(y)$ weakly in $L^2$;
		\item[(iii)] $|\nabla f_i|^2\to  |\nabla f|^2$ on $B_R(y)$ strongly in $L^1$.
	\end{itemize}
\end{theorem}

\begin{proposition}[{\cite[Corollary 4.12]{AmbrosioHonda18}}]\label{prop:harmonic approximation}
	Let $f\in H^{1,2}(B_R(y),\varrho,\mu)$ be a harmonic function
	(i.e., $f\in D(\Delta,B_R(y))$ with \(\Delta f=0\)). Then, for any $0<r<R$ there exist $f_i\in H^{1,2}(B_r(x_i),\sfd_i,\mm_i)$ harmonic such that $f_i\to f$ on $B_r(y)$ strongly in $H^{1,2}$.
\end{proposition}
\subsection{Normed modules}\label{subsection:modules}
Let \((X,\sfd,\mm)\) be a metric measure space.
We begin by briefly recalling the definitions of normed module
over \((X,\sfd,\mm)\), which have been introduced in \cite{Gigli14} and are in turn inspired by the theory developed in \cite{Weaver01}.
\medskip

Let \(R\) be either \(L^\infty(\mm)\) or \(L^0(\mm)\). Let \(\mathscr M\)
be a module over the commutative ring \(R\). Then an \emph{$L^p$-pointwise norm}
on \(\mathscr M\), for some $p\in\{0\}\cup[1,\infty)$, is any mapping \(|\cdot|:\,\mathscr M\to L^p(\mm)\) such that
\begin{equation}\label{eq:def_normed_module}\begin{split}
|v|\geq 0&\quad\text{ for every }v\in\mathscr M,
\text{ with equality if and only if }v=0,\\
|v+w|\leq|v|+|w|&\quad\text{ for every }v,w\in\mathscr M,\\
|fv|=|f||v|&\quad\text{ for every }f\in R\text{ and }v\in\mathscr M,
\end{split}\end{equation}
where all (in)equalities are in the \(\mm\)-a.e.\ sense.
We shall consider two classes of normed modules:
\begin{itemize}
\item \textsc{\(L^p(\mm)\)-normed \(L^\infty(\mm)\)-modules,
with \(p\in[1,\infty)\).} A module \(\mathscr M^p\) over \(L^\infty(\mm)\)
endowed with an $L^p$-pointwise norm \(|\cdot|\) such that
\(\|v\|_{\mathscr M^p}:=\big\||v|\big\|_{L^p(\mm)}\)
is a complete norm on \(\mathscr M^p\).
\item \textsc{\(L^0(\mm)\)-normed \(L^0(\mm)\)-modules.}
A module \(\mathscr M^0\) over \(L^0(\mm)\) endowed with an $L^0$-pointwise
norm \(|\cdot|\) such that
\(\sfd_{\mathscr M^0}(v,w):=\int\min\big\{|v-w|,1\big\}\,\d\mm'\)
(where \(\mm'\) is any probability measure that is mutually absolutely
continuous with \(\mm\)) is a complete distance on \(\mathscr M^0\).
\end{itemize}
We refer to \cite{Gigli17} for an account of the abstract normed modules theory
on metric measure spaces.
\medskip

Assume \((X,\sfd,\mm)\) is \emph{infinitesimally Hilbertian},
i.e., its Sobolev space \(H^{1,2}(X,\sfd,\mm)\) is Hilbert.
Then a key example of normed module on \(X\) is represented by the
\emph{tangent module} \(L^0(TX)\), which is characterized as follows:
there is a unique couple \(\big(L^0(TX),\nabla)\), where \(L^0(TX)\)
is an \(L^0(\mm)\)-normed \(L^0(\mm)\)-module and
\(\nabla:\,H^{1,2}(X)\to L^0(TX)\) is a linear
\emph{gradient} map, such that the following hold:
\[\begin{split}
|\nabla f|\text{ coincides with the minimal relaxed slope of }f
& \text{ for every }f\in H^{1,2}(X),\\
\bigg\{\sum_{i=1}^n\nchi_{E_i}\nabla f_i\;\bigg|\;(E_i)_{i=1}^n
\text{ Borel partition of }X,\ (f_i)_{i=1}^n\subset H^{1,2}(X)\bigg\}
&\ \text{ is dense in }L^0(TX).
\end{split}\]
For any exponent \(p\in[1,\infty]\), we set
\(L^p(TX):=\big\{v\in L^0(TX)\,:\,|v|\in L^p(\mm)\big\}\).
It can be readily checked that the space \(L^p(TX)\) has a natural
\(L^p(\mm)\)-normed \(L^\infty(\mm)\)-module structure (for $p<\infty$).

\subsubsection{Second order calculus over $\RCD$ spaces}
Gigli in \cite{Gigli14} has developed a second order calculus for $\RCD(K,\infty)$ metric measure spaces. The notions of Hessian and covariant derivative have been introduced as bilinear forms on $L^2(TX)$, along with the spaces $H^{2,2}(X,\sfd,\mm)\subset H^{1,2}(X,\sfd,\mm)$ and $H^{1,2}_C(TX)\subset L^2(TX)$, see \cite[Definition 3.3.1, Definition 3.4.1, Definition 3.3.17, Definition 3.4.3]{Gigli14}.

Let us recall that, as proved in \cite[Proposition 3.3.18]{Gigli14}, we have the inclusion
\begin{equation}\label{eq:dominioLaplacianoH22}
	D(\Delta)\subset H^{2,2}(X,\sfd,\mm).
\end{equation}
Moreover, assuming $(X,\sfd,\mm)$ to be $\RCD(K,N)$ m.m.s., one has the local estimate
\begin{equation}\label{eq:local Hessian estimate}
	\int_{B_1(x)} |\Hess f|^2\, \d\mm  \le C_{N,K}\left(\int_{B_2(x)}|\Delta f|^2\,\d \mm
	+\inf_{m\in\R}\int_{B_2(x)}\big||\nabla f|^2-m\big|\,\d\mm\right) 
	 -K\int_{B_2(x)} |\nabla f|^2\,\d\mm,
\end{equation}
that can be checked integrating the improved Bochner inequality proved in
 \cite{Han14} against a good cut-off function (see \Cref{lemma:good cut-off} above).

Let us recall that the Hessian enjoys the following locality property that has been proved in \cite[Proposition 3.3.24]{Gigli14}.
\begin{proposition}\label{prop:localityHessian}
	Given $f_1, f_2\in H^{2,2}(X,\sfd,\mm)$ it holds
	\begin{equation*}
		|\Hess f_1|=|\Hess f_2|
		\quad \text{$\mm$-a.e.\ in }\left\lbrace f_1=f_2 \right\rbrace.
	\end{equation*}
\end{proposition}

In addition we shall use the following inequality that has been proved in \cite[Proposition 3.3.22]{Gigli14}:
\begin{equation}\label{eq:calculusrule}
	|\nabla (\nabla f\cdot \nabla g)|\le |\Hess f|\, |\nabla g|+|\Hess g|\, |\nabla f|
	\quad\text{for any }f,g\in H^{2,2}(X,\sfd,\mm).
\end{equation}

\subsubsection{Module with respect to the capacity measure}
We recall a variant of the notion of \(L^0\)-normed \(L^0\)-module
-- where the Borel measure \(\mm\) is replaced by the capacity -- which
has been proposed in \cite{DGP19}. Fix a metric measure space \((X,\sfd,\mm)\).
The space of all Borel functions on \(X\) -- considered up to
\(\Cap\)-a.e.\ equality -- is denoted by \(L^0(\Cap)\). If continuous
functions are strongly dense in \(H^{1,2}(X)\) (this condition is met, for instance,
if the space is infinitesimally Hilbertian), then there exists a unique
``quasi-continuous representative'' map \({\sf QCR}:\,H^{1,2}(X)\to L^0(\Cap)\)
that is characterized as follows: \(\sf QCR\) is a continuous map, and for any
\(f\in H^{1,2}(X)\) it holds that \({\sf QCR}(f)\) is (the equivalence class of)
a quasi-continuous function that is \(\mm\)-a.e.\ coinciding with \(f\) itself.
	Let us recall that a function $f:X\to \R$ is said to be quasi-continuous if for any $\eps>0$ there exists a set $E\subset X$ with $\Cap (E)<\eps$ such that $f:X\setminus E\to \R$ is continuous.
We refer to \cite[Theorem 1.20]{DGP19} for a proof of this result.
\medskip

Given a module \(\mathscr M_\Cap\) over the ring \(L^0(\Cap)\),
we say that a mapping \(|\cdot|:\,\mathscr M_\Cap\to L^0(\Cap)\)
is a \emph{pointwise norm} provided it satisfies the (in)equalities in
\eqref{eq:def_normed_module} in the \(\Cap\)-a.e.\ sense for any choice of
\(v,w\in\mathscr M_\Cap\) and \(f\in L^0(\Cap)\). Then the space
\(\mathscr M_\Cap\) is said to be an \emph{\(L^0(\Cap)\)-normed \(L^0(\Cap)\)-module}
if it is complete when endowed with the distance
\[
\sfd_{\mathscr M_\Cap}(v,w):=\sum_{k\in\N}\frac{1}{2^k\max\big\{\Cap(A_k),1\big\}}
\int_{A_k}\min\big\{|v-w|,1\big\}\,\d\Cap,
\]
where \((A_k)_k\) is any increasing sequence of open subsets of \(X\) having
finite capacity that is chosen in such a way that any bounded set \(B\subset X\) 
is contained in \(A_k\) for some \(k\in\N\) sufficiently big.

Let us recall, since this fact plays a crucial role in the discussion below, that $\abs{\nabla f}^2\in H^{1,2}(X)$ for any $f\in \Test(X)$ (see \cite{Savare13}), and thus \(|\nabla f|\in H^{1,2}(X)\) as well (see \cite{DGP19}). In particular, for any $f\in \Test(X)$, $\abs{\nabla f}$ admits a quasi-continuous representative.
\begin{theorem}[Tangent \(L^0(\Cap)\)-module \cite{DGP19}]
\label{thm:tangent_Cap-module}
Let \((X,\sfd,\mm)\) be an \(\RCD(K,\infty)\) space. Then there exists
a unique couple \(\big(L^0_\Cap(TX),\tilde\nabla\big)\), where \(L^0_\Cap(TX)\)
is an \(L^0(\Cap)\)-normed \(L^0(\Cap)\)-module and
\(\tilde\nabla:\,\Test(X)\to L^0_\Cap(TX)\) is a linear operator, such that
the following hold:
\[\begin{split}
|\tilde\nabla f|={\sf QCR}(|\nabla f|)\;\;\;\text{in the }\Cap\text{-a.e.\ sense}
&\quad\text{ for every }f\in\Test(X),\\
\bigg\{\sum_{n\in\N}\nchi_{E_n}\tilde\nabla f_n\;\bigg|\;(E_n)_n
\text{ Borel partition of }X,\ (f_n)_n\subset\Test(X)\bigg\}
&\quad\text{ is dense in }L^0_\Cap(TX).
\end{split}\]
The space \(L^0_\Cap(TX)\) is called \emph{capacitary tangent module} on
\(X\), while \(\tilde\nabla\) is the \emph{capacitary gradient}.
\end{theorem}

Fix any Radon measure \(\mu\) on a m.m.s.\ \((X,\sfd,\mm)\) and suppose
that \(\mu\ll\Cap\). Then there is a natural projection
\(\pi_\mu:\,L^0(\Cap)\to L^0(\mu)\). Given an \(L^0(\Cap)\)-normed
\(L^0(\Cap)\)-module \(\mathscr M_\Cap\),  we define an equivalence relation
\(\sim_\mu\) on \(\mathscr M_\Cap\) as follows:
given any \(v,w\in\mathscr M_\Cap\), we declare that
\[
v\sim_\mu w\quad\Longleftrightarrow\quad|v-w|=0\ \text{ holds }\mu\text{-a.e.\ on }X.
\]
Then the quotient \(\mathscr M^0_\mu:=\mathscr M_\Cap/\sim_\mu\)
inherits a natural structure of \(L^0(\mu)\)-normed \(L^0(\mu)\)-module.
Call \(\bar\pi_\mu:\,\mathscr M_\Cap\to\mathscr M^0_\mu\) the canonical projection.
Moreover, for any exponent \(p\in[1,\infty)\) we define
\begin{equation}\label{eq:restr_mod}
\mathscr M^p_\mu:=\big\{v\in\mathscr M^0_\mu\;\big|\;|v|\in L^p(\mu)\big\}.
\end{equation}
It turns out that \(\mathscr M^p_\mu\) is an \(L^p(\mu)\)-normed \(L^\infty(\mu)\)-module.
Notice that \(|\bar\pi_\mu(v)|=\pi_\mu(|v|)\) holds in the \(\mu\)-a.e.\ sense
for every \(v\in\mathscr M_\Cap\).
\begin{lemma}\label{lem:gen_set_restr}
Let \((X,\sfd,\mm)\) be a m.m.s., \(\mathscr M_\Cap\)
an \(L^0(\Cap)\)-normed \(L^0(\Cap)\)-module. Fix a finite
%\footnote{\color{red}Vale anche per \(\mu\) non finita se chiediamo che gli
%insiemi \(E_i\) siano limitati.}
Borel measure \(\mu\geq 0\) on \(X\) such that \(\mu\ll\Cap\). Let \(V\) be a linear
subspace of \(\mathscr M_\Cap\) such that \(|v|\) admits a bounded $\Cap$-a.e.\ representative for every \(v\in V\) and
\[
\mathcal V:=\bigg\{\sum_{n\in\N}\nchi_{E_n} v_n\;\bigg|\;
(E_n)_{n\in\N}\text{ Borel partition of }X,\;(v_n)_{n\in\N}\subset V\bigg\}
\quad\text{ is dense in }\mathscr M_\Cap.
\]
Then for any \(p\in[1,\infty)\) it holds that
\[
\mathcal W:=\bigg\{\sum_{i=1}^n\nchi_{E_i}\bar\pi_\mu(v_i)\;\bigg|\;
n\in\N,\;(E_i)_{i=1}^n\text{ Borel partition of }X,\;(v_i)_{i=1}^n\subset V\bigg\}
\quad\text{ is dense in }\mathscr M^p_\mu.
\]
\end{lemma}
\begin{proof}
Fix \(v\in\mathscr M^p_\mu\) and \(\eps>0\). Since \(|v|^p\in L^1(\mu)\),
there is \(\delta>0\) such that
\(\big(\int_E|v|^p\,\d\mu\big)^{\nicefrac{1}{p}}\leq\eps/3\) holds for any
Borel set \(E\subset X\) with \(\mu(E)<\delta\). Choose any
\(\bar v\in\mathscr M_\Cap\) such that \(\bar\pi_\mu(\bar v)=v\).
We can find \((\bar v_k)_k\subset\mathcal V\) so that
\(|\bar v_k-\bar v|\to 0\) in \(L^0(\Cap)\). Hence
\(\big|\bar\pi_\mu(\bar v_k)-\bar\pi_\mu(\bar v)\big|=
\pi_\mu\big(|\bar v_k-\bar v|\big)\to 0\) in \(L^0(\mu)\). Thanks to Egorov theorem,
there exists a compact set \(K\subset X\) with \(\mu(X\setminus K)<\delta\)
such that (possibly taking a not relabeled subsequence) it holds that
\(\big|\bar\pi_\mu(\bar v_k)-v\big|\to 0\) uniformly on \(K\). Consequently,
by dominated convergence theorem we see that \(\nchi_K\bar\pi_\mu(\bar v_k)\to\nchi_K v\)
in \(\mathscr M^p_\mu\). Then we can pick \(k\in\N\) so that the element
\(\bar w:=\bar v_k\) satisfies
\(\big\|\nchi_K\bar\pi_\mu(\bar w)-\nchi_K v\big\|_{\mathscr M^p_\mu}\leq\eps/3\). If
\(\bar w\) is written as \(\sum_{n\in\N}\nchi_{E_n}\bar w_n\), then we have
\(\nchi_K\bar\pi_\mu(\bar w)=\sum_{n\in\N}\nchi_{K\cap E_n}\bar\pi_\mu(\bar w_n)\).
By dominated convergence theorem we know that for \(N\in\N\) sufficiently big
the element \(z:=\sum_{n=1}^N\nchi_{K\cap E_n}\bar\pi_\mu(\bar w_n)\in\mathcal W\)
satisfies \(\big\|z-\nchi_K\bar\pi_\mu(\bar w)\big\|_{\mathscr M^p_\mu}\leq\eps/3\).
Therefore, we conclude that
\[
\|z-v\|_{\mathscr M^p_\mu}\leq\big\|z-\nchi_K\bar\pi_\mu(\bar w)\big\|_{\mathscr M^p_\mu}
+\big\|\nchi_K\bar\pi_\mu(\bar w)-\nchi_K v\big\|_{\mathscr M^p_\mu}
+\|\nchi_{X\setminus K}v\|_{\mathscr M^p_\mu}\leq\eps,
\]
thus proving the statement.
\end{proof}
\subsection{Hodge Laplacian of vector fields on \texorpdfstring{\(\RCD\)}{RCD} spaces}\label{subsection:HodgeLaplacian}
Let \((X,\sfd,\mm)\) be an \(\RCD(K,\infty)\) space. Consider the
space \(H^{1,2}_{\rm H}(TX)\) and the Hodge Laplacian
\(\Delta_{\rm H}:\,D(\Delta_{\rm H})\subset H^{1,2}_{\rm H}(TX)\to L^2(TX)\),
which have been defined in \cite[Definition 3.5.13]{Gigli14} and
\cite[Definition 3.5.15]{Gigli14}, respectively (cf.\ the first paragraph of
\cite[Section 2.6]{Gigli17} for the identification between vector and covector fields).
\medskip

It follows from its definition that the Hodge Laplacian is self-adjoint,
namely that
\begin{equation}\label{eq:Hodge_Lapl_sa}
\int\langle\Delta_{\rm H}v,w\rangle\,\d\mm=\int\langle v,\Delta_{\rm H}w\rangle\,\d\mm
\quad\text{ for every }v,w\in D(\Delta_{\rm H}).
\end{equation}
Let us consider the \emph{augmented Hodge energy functional}
\(\tilde{\mathcal E}_{\rm H}:\,L^2(TX)\to[0,+\infty]\), which is defined
in \cite[eq.\ (3.5.16)]{Gigli14} (up to identifying \(L^2(T^*X)\) with \(L^2(TX)\)
via the musical isomorphism). Then we denote by
\(({\sf h}_{{\rm H},t})_{t\geq 0}\)
%\footnote{\color{red}Meglio \(P_{{\rm H},t}\)?}
the gradient flow in \(L^2(TX)\) of the functional  \(\tilde{\mathcal E}_{\rm H}\).
This means that for any vector field \(v\in L^2(TX)\) it holds that
\(t\mapsto{\sf h}_{{\rm H},t}(v)\in L^2(TX)\) is the unique continuous curve on \([0,+\infty)\)
with \({\sf h}_{{\rm H},0}(v)=v\), which is locally absolutely continuous on \((0,+\infty)\)
and satisfies
\[
{\sf h}_{{\rm H},t}(v)\in D(\Delta_{\rm H})\;\;\;\text{and}\;\;\;
\frac{\d}{\d t}{\sf h}_{{\rm H},t}(v)=-\Delta_{\rm H}{\sf h}_{{\rm H},t}(v)
\quad\text{ for every }t>0.
\]
Cf.\ the discussion that precedes \cite[Proposition 3.6.10]{Gigli14}.
It also holds that
\begin{equation}\label{eq:commute_h_Ht}
{\sf h}_{{\rm H},t}(\nabla f)=\nabla P_t f
\quad\text{ for every }f\in H^{1,2}(X)\text{ and }t\geq 0.
\end{equation}
Finally, we recall that vector fields satisfy the following Bakry-\'{E}mery contraction estimate (see \cite[Proposition 3.6.10]{Gigli14}):
\begin{equation}\label{eq:BE_vf}
|{\sf h}_{{\rm H},t}(v)|^2\leq e^{-2Kt}P_t(|v|^2)\;\;\;\mm\text{-a.e.}
\quad\text{ for every }v\in L^2(TX)\text{ and }t\geq 0.
\end{equation}
\begin{lemma}[\({\sf h}_{{\rm H},t}\) is self-adjoint]
Let \((X,\sfd,\mm)\) be an \(\RCD(K,\infty)\) space. Then it holds that
\begin{equation}\label{eq:h_HT_sa}
\int\langle{\sf h}_{{\rm H},t}(v),w\rangle\,\d\mm=
\int\langle v,{\sf h}_{{\rm H},t}(w)\rangle\,\d\mm
\quad\text{ for every }v,w\in L^2(TX)\text{ and }t\geq 0.
\end{equation}
\end{lemma}
\begin{proof}
Fix \(v,w\in L^2(TX)\) and \(t>0\). We define the function \(\varphi:\,[0,t]\to\R\) as
\[
\varphi(s):=\int\langle{\sf h}_{{\rm H},s}(v),{\sf h}_{{\rm H},t-s}(w)\rangle\,\d\mm
\quad\text{ for every }s\in[0,t].
\]
Therefore, the function \(\varphi\) is absolutely continuous and satisfies
\[
\varphi'(s)=
-\int\langle\Delta_{\rm H}{\sf h}_{{\rm H},s}(v),{\sf h}_{{\rm H},t-s}(w)\rangle\,\d\mm
+\int\langle{\sf h}_{{\rm H},s}(v),\Delta_{\rm H}{\sf h}_{{\rm H},t-s}(w)\rangle\,\d\mm
\overset{\eqref{eq:Hodge_Lapl_sa}}=0\quad\text{ for a.e.\ }t>0.
\]
Then \(\varphi\) is constant, thus in particular
\(\int\langle{\sf h}_{{\rm H},t}(v),w\rangle\,\d\mm=\varphi(t)
=\varphi(0)=\int\langle v,{\sf h}_{{\rm H},t}(w)\rangle\,\d\mm\).
\end{proof}
\begin{proposition}\label{prop:heat_flow_div}
Let \((X,\sfd,\mm)\) be an \(\RCD(K,\infty)\) space.
Then for any \(v\in D({\rm div})\) it holds that
\[
{\sf h}_{{\rm H},t}(v)\in H^{1,2}_C(TX)\cap D({\rm div})\;\;\;\text{and}\;\;\;
{\rm div}({\sf h}_{{\rm H},t}(v))=P_t({\rm div}(v))
\quad\text{ for every }t>0.
\]
\end{proposition}
\begin{proof}
First of all, observe that
\({\sf h}_{{\rm H},t}(v)\in H^{1,2}_{\rm H}(TX)\subset H^{1,2}_C(TX)\)
by \cite[Corollary 3.6.4]{Gigli14}. Moreover, let \(f\in H^{1,2}(X)\) be given.
Then it holds that
\[\begin{split}
\int\langle\nabla f,{\sf h}_{{\rm H},t}(v)\rangle\,\d\mm
&\overset{\eqref{eq:h_HT_sa}}=\int\langle{\sf h}_{{\rm H},t}(\nabla f),v\rangle\,\d\mm
\overset{\eqref{eq:commute_h_Ht}}=\int\langle\nabla P_t f,v\rangle\,\d\mm
=-\int P_t f\,{\rm div}(v)\,\d\mm\\
&\overset{\phantom{\eqref{eq:h_HT_sa}}}=-\int f P_t({\rm div}(v))\,\d\mm.
\end{split}\]
By arbitrariness of \(f\), we conclude that \({\sf h}_{{\rm H},t}(v)\in D({\rm div})\)
and \({\rm div}({\sf h}_{{\rm H},t}(v))=P_t({\rm div}(v))\).
\end{proof}

\section{A Gauss-Green formula on \texorpdfstring{\(\RCD\)}{RCD} spaces}
\label{section:Gauss-Green formula}

Let \((X,\sfd,\mm)\) be an \(\RCD(K,N)\) m.m.\ space and \(E\subset X\)
a set of finite perimeter. We recall that, by \Cref{lemma:AssolutacontinuitaHauss}, one has \(|D\nchi_E|\ll\mathscr H^{h_1}\),
so accordingly \(|D\nchi_E|\ll\Cap\) by \Cref{thm:H_ac_Cap}.
It thus makes sense to consider the projection
\(\pi_{|D\nchi_E|}:\,L^0(\Cap)\to L^0(|D\nchi_E|)\).
Recall also that \({\sf QCR}:\,H^{1,2}(X)\to L^0(\Cap)\)
stands for the ``quasi-continuous representative'' operator. Then let us define
\[
\tr_E:\,H^{1,2}(X)\to L^0(|D\nchi_E|),\quad\tr_E:=\pi_{|D\nchi_E|}\circ{\sf QCR},
\]
the trace operator over the boundary of $E$.
Observe that \(\tr_E(f)\in L^\infty(|D\nchi_E|)\) holds for
every test function \(f\in\Test(X)\).

\bigskip

This being said, let us state the two main results of this section. The first one gives existence and uniqueness of the tangent module over the boundary of a set of finite perimeter.
The second theorem provides a Gauss--Green formula tailored for finite-dimensional $\RCD$ spaces along with a strong approximation result for the exterior normal of sets with finite perimeter. 
This approximation result, whose proof heavily relies on the abstract machinery of normed modules and on functional-analytic tools, plays a key role in the study of rectifiability properties for boundaries of sets with finite perimeter that we are going to perform in the last section of this note. 

Let us point out that in the very recent \cite{BuffaComiMiranda19} the problem of obtaining a Gauss--Green formula on $\RCD(K,\infty)$ spaces has been treated. A comparison between our stronger result, heavily relying on finite dimensionality, and those in \cite{BuffaComiMiranda19} is outside the scope of this note.  

\begin{theorem}[Tangent module over \(\partial E\)]\label{thm:tg_mod_over_bdry}
Let \((X,\sfd,\mm)\) be an \(\RCD(K,N)\) space. Let \(E\subset X\)
be a set of finite perimeter.
Then there exists a unique couple \(\big(L^2_E(TX),\bar\nabla\big)\) -- where
\(L^2_E(TX)\) is an \(L^2(|D\nchi_E|)\)-normed \(L^\infty(|D\nchi_E|)\)-module
and \(\bar\nabla:\,{\rm Test}(X)\to L^2_E(TX)\) is linear -- such that:
\begin{itemize}
\item[\(\rm i)\)] The equality \(|\bar\nabla f|=\tr_E(|\nabla f|)\) holds
\(|D\nchi_E|\)-a.e.\ for every \(f\in{\rm Test}(X)\).
\item[\(\rm ii)\)] \(\big\{\sum_{i=1}^n\nchi_{E_i}\bar\nabla f_i\;
\big|\;(E_i)_{i=1}^n\text{ Borel partition of }X,
\,(f_i)_{i=1}^n\subset{\rm Test}(X)\big\}\) is dense in \(L^2_E(TX)\).
\end{itemize}
Uniqueness is intended up to unique isomorphism: given another couple
\((\mathscr M,\bar\nabla')\) satisfying \(\rm i)\), \(\rm ii)\) above, there exists a unique
normed module isomorphism \(\Phi:\,L^2_E(TX)\to\mathscr M\) such that
\(\Phi\circ\bar\nabla=\bar\nabla'\). The space \(L^2_E(TX)\) is called
\emph{tangent module over the boundary of \(E\)} and \(\bar\nabla\)
is the \emph{gradient}.
\end{theorem}

We denote by \(\bar{\sf QCR}:\,H^{1,2}_C(TX)\to L^0_\Cap(TX)\)
the ``quasi-continuous representative'' map for Sobolev vector fields,
whose existence has been proven in \cite[Theorem 2.14]{DGP19} (see \cite[Definition 2.12]{DGP19} for a notion of ``quasi-continuous vector field'' suitable for this context). Moreover, with a slight abuse of notation we define
\[
\tr_E:\,H^{1,2}_C(TX)\cap L^\infty(TX)\to L^2_E(TX),
\quad\tr_E:=\bar\pi_{|D\nchi_E|}\circ\bar{\sf QCR}.
\]
Notice that \(|\tr_E(v)|=\tr_E(|v|)\) holds in the \(|D\nchi_E|\)-a.e.\ sense
for every \(v\in H^{1,2}_C(TX)\cap L^\infty(TX)\).

\begin{theorem}[Gauss--Green formula on \(\RCD\) spaces]\label{thm:Gauss-Green}
	Let \((X,\sfd,\mm)\) be an \(\RCD(K,N)\) space and
	\(E\subset X\) be a set of finite perimeter such that \(\mm(E)<+\infty\).
	Then there exists a unique vector field \(\nu_E\in L^2_E(TX)\) such that \(|\nu_E|=1\)
	holds \(|D\nchi_E|\)-a.e.\ and
	\begin{equation}\label{eq:Gauss-Green}
	\int_E{\rm div}(v)\,\d\mm=-\int\big\langle\tr_E(v),\nu_E\big\rangle\,\d|D\nchi_E|
	\quad\text{ for all }v\in H^{1,2}_C(TX)\cap D({\rm div})
	\text{ with }|v|\in L^\infty(\mm).
	\end{equation}
	Moreover, there exists a sequence \((v_n)_n\subset{\rm TestV}_E(X)\) of test
	vector fields over the boundary of \(E\) (see \Cref{lem:density_TestVE} below for the precise definition of this class) such that \(v_n\to\nu_E\) in the
	strong topology of \(L^2_E(TX)\).
\end{theorem}
\begin{remark}
In the case in which $X$ is a Riemannian manifold and $E\subset X$ is a domain with smooth boundary, it holds that $L^2_E(TX)$ is the space of all
Borel vector fields over \(X\) which are concentrated on the boundary of \(E\)
and $2$-integrable with respect to the surface measure and, in this case, $\bar{\nabla}$ is the classical gradient for smooth functions.
\end{remark}

\begin{remark}
	The tangent \(L^0(\Cap)\)-module \(L^0_\Cap(TX)\) is a Hilbert
	module; cf.\ \cite[Proposition 2.8]{DGP19}. Therefore, it is immediate to
	see by passing to the quotient that \(L^2_E(TX)\) is a Hilbert module as well.
\end{remark}

The remaining part of this section is dedicated to the proofs of \Cref{thm:tg_mod_over_bdry} and \Cref{thm:Gauss-Green}.

\begin{proof}[Proof of \Cref{thm:tg_mod_over_bdry}]
\textsc{Uniqueness.} Call \(\mathcal W\) the family of elements
of \(L^2_E(TX)\) considered in item ii). Given any
\(\omega=\sum_{i=1}^n\nchi_{E_i}\bar\nabla f_i\in\mathcal W\), we are
forced to set \(\Phi(\omega):=\sum_{i=1}^n\nchi_{E_i}\bar\nabla' f_i\).
Well-posedness of such definition stems from the \(|D\nchi_E|\)-a.e.\ identity
\[
\bigg|\sum_{i=1}^n\nchi_{E_i}\bar\nabla' f_i\bigg|
=\sum_{i=1}^n\nchi_{E_i}|\bar\nabla' f_i|
=\sum_{i=1}^n\nchi_{E_i}\tr_E(|\nabla f_i|)=\sum_{i=1}^n\nchi_{E_i}|\bar\nabla f_i|
=|\omega|,
\]
which also shows that \(\Phi\) preserves the pointwise norm. Then \(\Phi\) is
linear continuous, thus it can be uniquely extended to a linear continuous map
\(\Phi:\,L^2_E(TX)\to\mathscr M\) by density of \(\mathcal W\) in
\(L^2_E(TX)\). By an approximation argument, it is easy to see that the extended
\(\Phi\) preserves the pointwise norm and is an \(L^\infty(|D\nchi_E|)\)-module
morphism. Finally, the map \(\Phi\) is surjective, because its image is dense
(as \(\mathscr M\) satisfies ii)) and closed (as \(\Phi\) is an isometry).
Consequently, we have proved that there exists a unique normed module isomorphism
\(\Phi:\,L^2_E(TX)\to\mathscr M\) such that \(\Phi\circ\bar\nabla=\bar\nabla'\).\\
\textsc{Existence.} Let us consider the tangent \(L^0(\Cap)\)-module
\(L^0_\Cap(TX)\) and the relative capacitary gradient operator
\(\tilde\nabla:\,{\rm Test}(X)\to L^0_\Cap(TX)\) associated to the space
\((X,\sfd,\mm)\); cf.\ \Cref{thm:tangent_Cap-module}.
We define \(L^0_E(TX)\) as
\(L^0_\Cap(TX)/\sim_{|D\nchi_E|}\) and the \(L^2(|D\nchi_E|)\)-normed
\(L^\infty(|D\nchi_E|)\)-module \(L^2_E(TX)\) as in \eqref{eq:restr_mod}.
Moreover, we define the differential \(\bar\nabla:\,{\rm Test}(X)\to L^2_E(TX)\)
as \(\bar\nabla:=\bar\pi_{|D\nchi_E|}\circ\tilde\nabla\). Clearly, the map
\(\bar\nabla\) is linear by construction. Given any function \(f\in{\rm Test}(X)\),
it \(|D\nchi_E|\)-a.e.\ holds
\[
|\bar\nabla f|=\big|\bar\pi_{|D\nchi_E|}(\tilde\nabla f)\big|
=\pi_{|D\nchi_E|}(|\tilde\nabla f|)
=\pi_{|D\nchi_E|}\big({\sf QCR}(|\nabla f|)\big)=\tr_E(|\nabla f|),
\]
which shows that i) is satisfied. We also set \(V:={\rm Test}(X)\) and
the associated space \(\mathcal V\subset L^0_\Cap(TX)\) as in the statement
of \Cref{lem:gen_set_restr}. By the defining property of the cotangent
\(\Cap\)-module we know that \(\mathcal V\) is dense in \(L^0_\Cap(TX)\),
whence \Cref{lem:gen_set_restr} ensures that \(\mathcal W\) is dense
in \(L^2_E(TX)\). This means that property ii) holds.
Therefore, the existence part of the statement is proven.
\end{proof}

To prove \Cref{thm:Gauss-Green} we need some auxiliary results.
Let us begin with the following one, which was obtained as an
intermediate step in the proof of \cite[Theorem 4.2]{ABS18}.

\begin{lemma}\label{lemma:sci}
Let \((X,\sfd,\mm)\) be an \(\RCD(K,N)\) space. Let \(E\subset X\)
be a set of finite perimeter. Then
\begin{equation}\label{eq:sci}
\lim_{t\searrow 0}\int\bigg|1-e^{Kt}\frac{|\nabla P_t\nchi_E|}{P_t^*|D\nchi_E|}\bigg|
P_t^*|D\nchi_E|\,\d\mm=0.
\end{equation}
\end{lemma}
\begin{lemma}\label{lem:approx_hf_per}
Let \((X,\sfd,\mm)\) be an \(\RCD(K,N)\) space.
Let \(E\subset X\) be a set of finite perimeter.
Then
\begin{equation}\label{eq:equality_hf_per}
\int f\,P_t^*|D\nchi_E|\,\d\mm=\int\tr_E(P_t f)\,\d|D\nchi_E|
\quad\text{ for every }f\in H^{1,2}(X)\cap L^\infty(\mm)\text{ and }t>0.
\end{equation}
Moreover, it holds that
\begin{equation}\label{eq:approx_hf_per}
\lim_{t\searrow 0}\int\tr_E(P_t f)\,\d|D\nchi_E|=\int\tr_E(f)\,\d|D\nchi_E|
\quad\text{ for every }f\in H^{1,2}(X)\cap L^\infty(\mm).
\end{equation}
\end{lemma}
\begin{proof}
First of all, let us prove \eqref{eq:equality_hf_per}.
Fix any \(f\in H^{1,2}(X)\cap L^\infty(\mm)\) and \(t>0\). We claim that
\begin{equation}\label{eq:approx_hf_per_claim}
\exists\,(f_n)_n\subset\Lip_{\rm bs}(X,\sfd)\text{ bounded in }L^\infty(\mm):
\quad f_n\to f\text{ strongly in }H^{1,2}(X),\text{ weakly}^*\text{ in }L^\infty(\mm).
\end{equation}
To prove it, we argue as follows. Given any \(s>0\), the function \(P_s f\)
has a Lipschitz representative (still denoted by \(P_s f\)) thanks to the
\(L^\infty\)-\(\Lip\) regularisation of the heat flow.
Since \(\{P_s f\}_{s>0}\) is bounded in \(L^\infty(\mm)\)
by the weak maximum principle and \(P_s|\nabla f|^2\to|\nabla f|^2\) strongly
in \(L^1(\mm)\), we can find a function \(G\in L^1(\mm)\) and
a sequence \(s_n\searrow 0\) such that
\(P_{s_n}|\nabla f|^2\leq G\) holds \(\mm\)-a.e.\ for all \(n\)
and \(P_{s_n}f\to f\) weakly\(^*\) in \(L^\infty(\mm)\).
Fix \(\bar x\in X\) and for any \(n\in\N\) choose a compactly-supported
\(1\)-Lipschitz function \(\eta_n:\,X\to[0,1]\) such that \(\eta_n=1\) on \(B_n(\bar x)\).
Therefore, standard computations (based on the Leibniz rule
\(\nabla(\eta_n P_{s_n}f)=\eta_n\nabla P_{s_n}f+P_{s_n}f\nabla\eta_n\), the
dominated convergence theorem, and the Bakry-\'{E}mery contraction estimate) show that
\(f_n:=\eta_n P_{s_n}f\in\Lip_{\rm bs}(X,\sfd)\) satisfy \eqref{eq:approx_hf_per_claim}.
Now observe that \(P_t:\,H^{1,2}(X)\to H^{1,2}(X)\) is continuous, as a consequence of
the Bakry-\'{E}mery contraction estimate and the continuity of
\(P_t:\,L^2(\mm)\to L^2(\mm)\). This ensures that \(P_t f_n\to P_t f\) strongly
in \(H^{1,2}(X)\) as \(n\to\infty\), whence we know from
\cite[Propositions 1.12, 1.17 and 1.19]{DGP19} that (possibly passing to
a not relabeled subsequence) \({\sf QCR}(P_t f_n)\to{\sf QCR}(P_t f)\)
holds \(\rm Cap\)-a.e., and accordingly \(\tr_E(P_t f_n)\to\tr_E(P_t f)\)
holds \(|D\nchi_E|\)-a.e.. Moreover, since
\(|P_t f_n|\leq\sup_k\|f_k\|_{L^\infty(\mm)}=:C\) in the \(\mm\)-a.e.\ sense
for all \(n\in\N\), we deduce that \(\big|{\sf QCR}(P_t f_n)\big|\leq C\) holds
\(\rm Cap\)-a.e.\ for all \(n\in\N\), and thus \(\tr_E(P_t f_n)\leq C\) holds
\(|D\nchi_E|\)-a.e.\ for all \(n\in\N\). All in all, we obtain \eqref{eq:equality_hf_per}
by letting \(n\to\infty\) in
\(\int f_n\,P_t^*|D\nchi_E|\,\d\mm=\int\tr_E(P_t f_n)\,\d|D\nchi_E|\),
which is satisfied thanks to the defining property of \(P_t^*|D\nchi_E|\);
here we use the dominated convergence theorem and the \(L^\infty\)-weak\(^*\)
convergence \(f_n\to f\).

Let us now pass to the proof of \eqref{eq:approx_hf_per}.
Fix \(f\in H^{1,2}(X)\cap L^\infty(\mm)\).
By arguing as above, we see that
\(\big|\tr_E(P_t f)\big|\leq\|f\|_{L^\infty(\mm)}\) holds
\(|D\nchi_E|\)-a.e.\ for all \(t>0\), and that any given sequence \(t_n\searrow 0\)
admits a subsequence \(t_{n_i}\searrow 0\) such that
\(\tr_E(P_{t_{n_i}}f)\to\tr_E(f)\) holds \(|D\nchi_E|\)-a.e..
Therefore, by dominated convergence theorem we conclude that
\(\lim_i\int\tr_E(P_{t_{n_i}}f)\,\d|D\nchi_E|=\int\tr_E(f)\,\d|D\nchi_E|\),
which yields \eqref{eq:approx_hf_per}.
\end{proof}

\begin{lemma}[Test vector fields over \(\partial E\)]\label{lem:density_TestVE}
	Let \((X,\sfd,\mm)\) be an \(\RCD(K,N)\) space. Let \(E\subset X\)
	be a set of finite perimeter and finite mass. We define the class \({\rm TestV}_E(X)\subset L^2_E(TX)\)
	of \emph{test vector fields over the boundary of \(E\)} as
	\[
	{\rm TestV}_E(X):=\tr_E\big({\rm TestV}(X)\big)
	=\bigg\{\sum_{i=1}^n\tr_E(g_i)\bar\nabla f_i\;\bigg|\;
	n\in\N,\;(f_i)_{i=1}^n,(g_i)_{i=1}^n\subset{\rm Test}(X)\bigg\}.
	\]
	Then \({\rm TestV}_E(X)\) is dense in \(L^2_E(TX)\).
\end{lemma}
\begin{proof}
	By item ii) of \Cref{thm:tg_mod_over_bdry}, it suffices to show that each 
	\(v\in L^2_E(TX)\) of the form \(v=\nchi_E\bar\nabla f\) -- where \(E\subset X\)
	is a Borel set and \(f\in{\rm Test}(X)\) -- can be approximated by elements
	of \({\rm TestV}_E(X)\) with respect to the strong topology of \(L^2_E(TX)\).
	Fix \(\eps>0\) and choose a function \(h\in\Lip_c(X)\) such that
	\(\|h-\nchi_E\|_{L^2(|D\nchi_E|)}\leq\eps/(2\,\Lip(f))\). Moreover, by exploiting
	\cite[eq.\ (3.2.3)]{Gigli14} we can find a sequence \((g_n)_n\subset{\rm Test}(X)\)
	such that \(\sup_n\|g_n\|_{L^\infty(\mm)}<+\infty\) and \(g_n\to h\) in \(H^{1,2}(X)\).
	Hence, by using the results in \cite{DGP19} we see that (up to a not relabeled
	subsequence) it holds \(\tr_E(g_n)(x)\to h(x)\) for \(|D\nchi_E|\)-a.e.\ \(x\in X\).
	Accordingly, by applying the dominated convergence theorem we conclude that
	\(\big|(\tr_E(g_n)-h)\bar\nabla f\big|\to 0\) in \(L^2(|D\nchi_E|)\).
	Now choose \(n\in\N\) so big that \(g:=g_n\) satisfies
	\(\big\|(\tr_E(g)-h)\bar\nabla f\big\|_{L^2_E(TX)}<\eps/2\). Hence, one has that
	\[\begin{split}
	\big\|\tr_E(g)\bar\nabla f-v\big\|_{L^2_E(TX)}
	&\leq\big\|(\tr_E(g)-h)\bar\nabla f\big\|_{L^2_E(TX)}
	+\big\|(h-\nchi_E)\bar\nabla f\big\|_{L^2_E(TX)}\\
	&\leq\frac{\eps}{2}+\|h-\nchi_E\|_{L^2(|D\nchi_E|)}\,\Lip(f)<\eps.
	\end{split}\]
	Given that \(\tr_E(g)\bar\nabla f\in{\rm TestV}_E(X)\), the statement is achieved.
\end{proof}
The last ingredient we need is an improvement of \Cref{thm:repr_tv} in the special case of \(\RCD(K,\infty)\) spaces. As we are going to see in the
ensuing result, to obtain the total variation of a BV function it is
sufficient to restrict the attention only to those competitors that
are Sobolev regular. The proof is based on a parabolic approximation argument that builds upon the technical results developed in \Cref{subsection:HodgeLaplacian}.
\begin{theorem}[Representation formula for \(|Df|\) on \(\RCD\) spaces]
	\label{thm:repr_tv_RCD}
	Let \((X,\sfd,\mm)\) be an \(\RCD(K,\infty)\) space and \(f\in{\rm BV}(X)\). Then it holds that
	\[
	|Df|(X)=\sup\bigg\{\int f\,{\rm div}(v)\,\d\mm\;\bigg|\;v\in H^{1,2}_C(TX)\cap D({\rm div}),
	\;|v|\leq 1\;\;\mm\text{-a.e.},\;{\rm div}(v)\in L^\infty(\mm)\bigg\}.
	\]
\end{theorem}
\begin{proof}
	Call \(S\) the right hand side of the above formula. We know by \Cref{rmk:ineq_Df_deriv} that \(|Df|(X)\geq S\). In order
	to prove the converse inequality, fix any \(\eps>0\). \Cref{thm:repr_tv} guarantees the
	existence of a vector field \(v\in D({\rm div})\) -- with \(|v|\leq 1\) in the
	\(\mm\)-a.e.\ sense and \({\rm div}(v)\in L^\infty(\mm)\) -- such that
	\(\int f\,{\rm div}(v)\,\d\mm>|Df|(X)-\eps/2\). Now define
	\(v_t:=e^{Kt}\,{\sf h}_{{\rm H},t}(v)\) for every \(t>0\).
	Notice that \(v_t\in H^{1,2}_C(TX)\cap D({\rm div})\) by \Cref{prop:heat_flow_div}. Since \({\rm div}(v)\in L^\infty(\mm)\) and
	\({\rm div}(v_t)=e^{Kt}P_t({\rm div}(v))\), we deduce from the weak
	maximum principle that \({\rm div}(v_t)\in L^\infty(\mm)\) as well.
	More precisely, one has
	\(\|{\rm div}(v_t)\|_{L^\infty(\mm)}\leq e^{Kt}\,\|{\rm div}(v)\|_{L^\infty(\mm)}\)
	for all \(t>0\). Moreover, the weak maximum principle also guarantees that
	\[
	|v_t|=e^{Kt}\,|{\sf h}_{{\rm H},t}(v)|\overset{\eqref{eq:BE_vf}}\leq
	\sqrt{P_t(|v|^2)}\leq 1\quad\text{ in the }\mm\text{-a.e.\ sense.}
	\]
	Given that \(\lim_{t\searrow 0}{\rm div}(v_t)={\rm div}(v)\) in \(L^2(\mm)\),
	we can find \(t_n\searrow 0\) such that \({\rm div}(v_{t_n})(x)\to{\rm div}(v)(x)\)
	holds for \(\mm\)-a.e.\ \(x\in X\). Being \(\big({\rm div}(v_{t_n})\big)_n\) a bounded
	sequence in \(L^\infty(\mm)\), we can finally conclude that
	\(\lim_n\int f\,{\rm div}(v_{t_n})\,\d\mm=\int f\,{\rm div}(v)\,\d\mm\)
	by dominated convergence theorem. Therefore, there exists \(n\in\N\) such that \(w:=v_{t_n}\) satisfies
	\[
	\int f\,{\rm div}(w)\,\d\mm>\int f\,{\rm div}(v)\,\d\mm-\frac{\eps}{2}>|Df|(X)-\eps.
	\]
	This shows that \(|Df|(X)<S+\eps\), whence \(|Df|(X)\leq S\) by arbitrariness
	of \(\eps\), as desired.
\end{proof}

\begin{proof}[Proof of \Cref{thm:Gauss-Green}]
First of all, let us define \(\mu_t:=P_t^*|D\nchi_E|\mm\) for every \(t>0\).
Recall that \(\mu_t\rightharpoonup|D\nchi_E|\) in duality with \(C_b(X)\) as
\(t\searrow 0\). Let us also set
\[
\nu_t:=\nchi_{\{P_t^*|D\nchi_E|>0\}}\frac{\nabla P_t\nchi_E}{P_t^*|D\nchi_E|}
\in L^0(TX)\quad\text{ for every }t>0.
\]
It follows from the \(1\)-Bakry-\'{E}mery estimate \eqref{eq:BE1} that \(|DP_t\nchi_E|\leq e^{-Kt}P_t^*|D\nchi_E|\)
holds \(\mm\)-a.e., thus accordingly \(\nu_t\in L^\infty(TX)\) and
\(|\nu_t|\leq e^{-Kt}\) is satisfied in the \(\mm\)-a.e.\ sense.
Call 
\begin{equation*}
	\mathcal V:=\big\{v\in H^{1,2}_C(TX)
	\cap D({\rm div})\,\big|\,|v|\in L^\infty(\mm)\big\} 
\end{equation*}
and fix \(v\in\mathcal V\).
The Leibniz rule for the divergence ensures that
\(\varphi v\in D({\rm div})\) for any \(\varphi\in\Lip_b(X)\), so the usual
integration-by-parts formula yields
\begin{equation}\label{eq:def_Lv_aux}
\int P_t\nchi_E\,{\rm div}(\varphi v)\,\d\mm
=-\int\varphi\,\langle\nabla P_t\nchi_E,v\rangle\,\d\mm
=-\int\varphi\,\langle v,\nu_t\rangle\,\d\mu_t
\quad\text{ for all }\varphi\in\Lip_b(X).
\end{equation}
Moreover, observe that \(\langle v,\nu_t\rangle\in L^\infty(\mu_t)\)
and \(\big\|\langle v,\nu_t\rangle\big\|_{L^\infty(\mu_t)}\leq
e^{-Kt}\,\||v|\|_{L^\infty(\mm)}\) for every \(t>0\). 
Let us call \(\sigma_t:=\langle v,\nu_t\rangle\mu_t\) for all \(t>0\).
Fix any sequence \(t_n\searrow 0\). Since \(\mu_{t_n}\rightharpoonup|D\nchi_E|\)
in duality with \(C_b(X)\), we know that \((\mu_{t_n})_n\) is tight by Prohkorov
theorem. Given that \(\sup_n\|\langle v,\nu_{t_n}\rangle\|_{L^\infty(\mu_{t_n})}\)
is finite, we deduce that \((\sigma_{t_n})_n\) is tight as well. By using
Prohkorov theorem again, we can thus take a subsequence \((t_{n_i})_i\)
such that \(\sigma_{t_{n_i}}\rightharpoonup\sigma\) in duality with
\(C_b(X)\) for some finite (signed) Borel measure \(\sigma\) on \(X\).
Since \(\Lip_b(X)\) is dense in \(C_b(X)\) and the identity in
\eqref{eq:def_Lv_aux} gives
\[
\int\varphi\,\d\sigma=\lim_{i\to\infty}\int\varphi\,\d\sigma_{t_{n_i}}
=-\int_E{\rm div}(\varphi v)\,\d\mm\quad\text{ for every }\varphi\in\Lip_b(X),
\]
we see that \(\sigma\) is independent of the chosen sequence \((t_{n_i})_i\).
Hence, \(\sigma_t\rightharpoonup\sigma\) in duality with \(C_b(X)\)
as \(t\searrow 0\). Given any non-negative function \(\varphi\in C_b(X)\),
it thus holds that
\[
\bigg|\int\varphi\,\d\sigma\bigg|\leq
\lim_{t\searrow 0}\int\varphi\,|\langle v,\nu_t\rangle|\,\d\mu_t
\leq e^{|K|}\,\||v|\|_{L^\infty(\mm)}\,\lim_{t\searrow 0}\int\varphi\,\d\mu_t
=e^{|K|}\,\||v|\|_{L^\infty(\mm)}\int\varphi\,\d|D\nchi_E|,
\]
whence \(\sigma\ll|D\nchi_E|\) and its Radon-Nikod\'{y}m
derivative \(L(v):=\frac{\d\sigma}{\d|D\nchi_E|}\) belongs to
\(L^\infty(|D\nchi_E|)\). Consequently, taking into account \eqref{eq:def_Lv_aux}
we deduce that
\begin{equation}\label{eq:def_Lv_aux2}
\int_E{\rm div}(\varphi v)\,\d\mm=-\int\varphi\,L(v)\,\d|D\nchi_E|
\quad\text{ for every }v\in\mathcal V\text{ and }\varphi\in\Lip_b(X).
\end{equation}
Furthermore, one also has that
\begin{equation}\label{eq:def_Lv_aux3}
\lim_{t\searrow 0}\int\varphi\,\langle v,\nu_t\rangle\,\d\mu_t
=\int\varphi\,L(v)\,\d|D\nchi_E|\quad\text{ for every }v\in\mathcal V
\text{ and }\varphi\in\Lip_b(X).
\end{equation}
Observe that for any \(v\in\mathcal V\) and \(\varphi\in\Lip_b(X)\),
\(\varphi\geq 0\) it holds that
\[\begin{split}
\bigg|\int\varphi\,L(v)\,\d|D\nchi_E|\bigg|&\overset{\eqref{eq:def_Lv_aux3}}=
\lim_{t\searrow 0}\bigg|e^{Kt}\int\varphi\,\langle v,\nu_t\rangle\,\d\mu_t\bigg|\\
&\overset{\phantom{\eqref{eq:def_Lv_aux3}}}\leq
\lim_{t\searrow 0}\bigg(\|\varphi\|_{L^\infty(\mm)}\,\||v|\|_{L^\infty(\mm)}
\int\big|1-e^{Kt}|\nu_t|\big|\,\d\mu_t+
\int\varphi\,\big\langle v,\frac{\nu_t}{|\nu_t|}\big\rangle\,\d\mu_t\bigg)\\
&\overset{\eqref{eq:sci}}\leq\lim_{t\searrow 0}\int\varphi\,|v|\,\d\mu_t
\overset{\eqref{eq:equality_hf_per}}=
\lim_{t\searrow 0}\int\tr_E\big(P_t(\varphi|v|)\big)\,\d|D\nchi_E|\\
&\overset{\eqref{eq:approx_hf_per}}=\int\varphi\,\tr_E(|v|)\,\d|D\nchi_E|.
\end{split}\]
In the last two equalities we used the fact that \(|v|\in H^{1,2}(X)\).
By arbitrariness of \(\varphi\), we obtain that
\(|L(v)|\leq\tr_E(|v|)\) holds \(|D\nchi_E|\)-a.e.\ for all \(v\in\mathcal V\).
Let us now define \(\omega:\,\tr_E(\mathcal V)\to L^1(|D\nchi_E|)\) as
\begin{equation}\label{eq:def_Lv_aux5}
\omega\big(\tr_E(v)\big):=L(v)\quad\text{ for every }v\in\mathcal V.
\end{equation}
The operator \(L:\,\mathcal V\to L^\infty(|D\nchi_E|)\) is linear by
its very construction, whence by exploiting the inequality \(|L(v)|\leq\tr_E(|v|)\)
we can conclude that \(\omega\) is well-posed, linear and satisfying
\[
\big|\omega(v)\big|\leq|v|\;\;\;|D\nchi_E|\text{-a.e.}
\quad\text{ for every }v\in\tr_E(\mathcal V).
\]
Since \({\rm TestV}(X)\subset\mathcal V\) and \({\rm TestV}_E(X)\)
is dense in \(L^2_E(TX)\), we infer from \Cref{lem:density_TestVE}
that \(\bar\tr_E(\mathcal V)\) is a dense linear subspace of \(L^2_E(TX)\).
Therefore, we know from \cite[Proposition 1.4.8]{Gigli14} that \(\omega\)
can be uniquely extended to an element \(\omega\in L^2_E(T^*X):=L^2_E(TX)^*\)
satisfying \(|\omega|\leq 1\) in the \(|D\nchi_E|\)-a.e.\ sense.
We denote by \(\nu_E\in L^2_E(TX)\) the vector field corresponding to
\(\omega\) via the Riesz isomorphism. By combining \eqref{eq:def_Lv_aux2}
(with \(\varphi\equiv 1\)) and \eqref{eq:def_Lv_aux5}, we conclude that
\eqref{eq:Gauss-Green} is satisfied. It only remains to show that
\(|\nu_E|\geq 1\) holds \(|D\nchi_E|\)-a.e.. In order to do it, just observe
that \Cref{thm:repr_tv_RCD} yields
\[\begin{split}
|D\nchi_E|(X)&\leq
\sup_{\substack{v\in\mathcal V, \\ |v|\leq 1\;\mm\text{-a.e.}}}\int_E{\rm div}(v)\,\d\mm
\overset{\eqref{eq:Gauss-Green}}=
\sup_{\substack{v\in\mathcal V, \\ |v|\leq 1\;\mm\text{-a.e.}}}
-\int\big\langle\tr_E(v),\nu_E\big\rangle\,\d|D\nchi_E|
\leq\int|\nu_E|\,\d|D\nchi_E|\\
&\leq|D\nchi_E|(X),
\end{split}\]
whence each inequality must be an equality. This clearly forces the
\(|D\nchi_E|\)-a.e.\ equality \(|\nu_E|=1\). The element
\(\nu_E\) is uniquely determined by \eqref{eq:Gauss-Green} as the
space \(\tr_E(\mathcal V)\) is dense in \(L^2_E(TX)\). Finally,
the last part of the statement is an immediate consequence of \Cref{lem:density_TestVE}.
\end{proof}

\section{Uniqueness of tangents for sets of finite perimeter}\label{sec:uniqueness}
In this section we prove a uniqueness theorem (up to negligible sets) for blow-ups of sets with finite perimeter over $\RCD(K,N)$ metric measure spaces. This has to be considered as a further step in the direction of generalizing De Giorgi's theorem to the framework of $\RCD$ spaces.

Let us recall the notion of tangent to a set of finite perimeter that has been introduced in \cite{ABS18}.

\begin{definition}[Tangents to a set of finite perimeter]\label{def:tan}
	Let $(X,\sfd,\mm)$ be an $\RCD(K,N)$ m.m.s., $x\in X$ and let $E\subset X$ be a set of locally finite perimeter. 
	We denote by $\Tan_x(X,\sfd,\mm,E)$ the collection of quintuples $(Y,\varrho,\mu,y, F)$ satisfying the following two properties:
	\begin{itemize}
		\item[(a)] $(Y,\varrho,\mu,y)\in\Tan_x(X,\sfd,\mm)$ and $r_i\downarrow 0$ are such that the rescaled spaces $(X,r_i^{-1}\sfd,\mm_x^{r_i},x)$ converge to $(Y,\varrho,\mu,y)$ in the pointed measured Gromov-Hausdorff topology;
		\item[(b)] $F$ is a set of locally finite perimeter in $Y$ with $\mu(F)>0$ and, if $r_i$ are as in $\rm(a)$, then the sequence $f_i=\nchi_E$ converges in $L^1_{\rm loc}$ to $\nchi_F$ according to \Cref{def:L1 convergence of sets}.
	\end{itemize}
\end{definition}
Let us point out that, up to a $\abs{D\nchi_E}$-negligible set, one also has that the perimeter measures on the rescaled spaces $\abs{D^i\nchi_E}$ weakly converge to $\abs{D\nchi_F}$ in duality w.r.t.\ $\Cbs$. This statement, which is part of \cite[Corollary 4.10]{ABS18}, plays a role in the rest of the note. 

We are ready to state the main theorem of this section.

\begin{theorem}\label{th:uniqueness}
	Let $(X,\sfd,\mm)$ be an $\RCD(K,N)$ m.m.s.\ with essential dimension $1\le n\le N$, $E\subset X$ be a set of finite perimeter. Then, for $|D\nchi_E|$-a.e.\ $x\in X$, there exists $k=1,\ldots,n$ such that
	\begin{equation*}
	\Tan_x(X,\sfd,\mm,E)=\left\lbrace (\R^k,\sfd_{eucl},c_k\Leb^k,0^k,\left\lbrace x_k>0\right\rbrace )\right\rbrace.
	\end{equation*}
\end{theorem}

Let us explain the strategy of its proof. The starting point of our analysis is \cite[Theorem 4.3]{ABS18} that we state below.

\begin{theorem}\label{thm:tangenthalfspace}
	Let $(X,\sfd,\mm)$ be an $\RCD(K,N)$ m.m.s.\ and $E\subset X$ be a set of locally finite perimeter. 
	Then $E$ admits a Euclidean half-space as tangent at $x$ for $\abs{D\nchi_E}$-a.e.\ $x\in X$, that is to say
	\begin{equation*}
	\left(\R^k,\sfd_{eucl},c_k\Leb^k,0^k,\left\lbrace x_k>0\right\rbrace \right)\in \Tan_x(X,\sfd,\mm,E),\qquad\text{for some $k\in [1,N]$.}
	\end{equation*} 
\end{theorem}

After establishing \Cref{thm:tangenthalfspace} the state of the art in the theory of sets of finite perimeter was similar to that of the structure theory of $\RCD$ spaces after \cite{Gigli-Mondino-Rajala15}, where the authors proved existence of a Euclidean tangent space up to negligible sets. The content of this and of the next section instead can be seen as a counterpart in codimension 1 of the main results obtained by Mondino--Naber in \cite{Mondino-Naber14}.\\ Also the main ideas underlying the proofs of the uniqueness of tangents and the rectifiability result are quite similar to those implemented in \cite{Mondino-Naber14}. As in that case, the existence of a Euclidean tangent along a fixed scale is a regularity information which can be propagated at any location and scale up to a set which is small w.r.t.\ the relevant measure, yielding uniqueness of tangents.

From a technical point of view, our construction heavily relies on the use of the so-called \textit{harmonic $\delta$-splitting maps}, a kind of good replacement for coordinate functions within the theory of lower Ricci bounds, that played a crucial role in the development of the theory of Ricci limits (see \cite{Cheeger-Colding97I,Cheeger-Colding97II,Cheeger-Colding97III} and the more recent \cite{CheegerNaber15,CheegerJiangNaber18}). Since, up to our knowledge, this is the first time they are explicitly used in the $\RCD$ framework, we dedicate \Cref{section:splitting map} below to establish some of their properties. With this tool at our disposal, the \textit{propagation of regularity step} is a consequence of a weighted maximal argument which was suggested in \cite{CheegerNaber15}. Let us point out that, in order for the whole procedure to work, the fact that perimeter measures have codimension 1 (see \Cref{lemma:AssolutacontinuitaHauss}) and the fact that harmonic functions satisfy $L^2$ Hessian bounds play a key role. The strategy would completely fail if perimeter measures had codimension bigger or equal than 2.

\subsection{Splitting maps and propagation of regularity}\label{section:splitting map}
This subsection is devoted to the study of $\delta$-splitting maps. Let us recall that their introduction in the study of spaces with lower Ricci curvature bounds dates back to \cite{Cheeger-Colding96}.
\begin{definition}\label{def:deltasplitting}
	Let $(X,\sfd,\mm)$ be an $\RCD(-1,N)$ metric measure space, $x\in X$ and $\delta>0$ be fixed.
	We say that $u:=(u_1,\ldots,u_k):B_r(x)\to \R^k$ is a $\delta$-splitting map provided it is harmonic (meaning that $u_a\in D(\Delta, B_r(x))$ with $\Delta u_a=0$ for any $a=1,\dots,k$) and satisfies:
	\begin{itemize}
		\item[(i)] $u_a$ is $C_N$-Lipschitz for any $a=1,\dots,k$;
		\item[(ii)] $r^2\fint_{B_r(x)} |\Hess u_a|^2\d\mm <\delta$ for any $a=1,\dots,k$;
		\item[(iii)] $\fint_{B_r(x)} |\nabla u_a\cdot \nabla u_b-\delta_{a,b}|\d\mm <\delta$ for any $a,b=1,\dots,k$.
	\end{itemize}
\end{definition}

\begin{remark}\label{remark:localHessian}
	Let us clarify the meaning of $|\Hess u|$ when $u:B_r(x)\to \R$ is harmonic and not necessarily globally defined. For any ball $B_{2s}(y)\subset B_r(x)$ we take a good cut-off function $\eta$ according to \Cref{lemma:good cut-off} that satisfies $\eta=1$ in $B_s(y)$ and $\eta=0$ in $X\setminus B_{2s}(y)$.
	As we already remarked in \Cref{subsubsection:laplacian on balls}, one has $\eta u\in D(\Delta)$, therefore $\eta u\in H^{2,2}(X,\sfd,\mm)$ as a consequence of \eqref{eq:dominioLaplacianoH22}. We can now set $|\Hess u|:=|\Hess (\eta u)|$ in $B_s(y)$. Observe that this is a good definition thanks to the locality of the Hessian (see \Cref{prop:localityHessian}).
\end{remark}

\begin{remark}
With respect to the definition of $\delta$-splitting map which is nowadays adopted within the theory of Ricci limits (see for instance \cite[Definition 1.20]{CheegerNaber15}) the main difference is condition (i). Therein the sharper bound $\abs{\nabla u}\le 1+\delta$ is imposed in the definition though, as they observe, it can be obtained as a consequence of the bound $\abs{\nabla u}\le C_N$ and of the other defining properties (when working in the smooth framework).
\end{remark}

\subsubsection{$\delta$-splitting maps and $\eps$-closeness}

The power of $\delta$-splitting maps in the theory of lower Ricci bounds is that, roughly speaking, they allow to pass from analysis to geometry and vice-versa. Namely, the existence of a $\delta$-splitting map with $k$ components on a Riemannian manifold with Ricci bounded below by $-\delta$ can be turned into $\eps$-GH closeness (in the scale invariant sense) to a space which splits a factor $\R^k$ and vice-versa (see \cite{Cheeger-Colding96} and \cite[Lemma 1.21]{CheegerNaber15}).

Below we wish to provide rigorous statements of the above-mentioned results in the framework of $\RCD$ spaces. The convergence and stability results of \cite{AmbrosioHonda17,AmbrosioHonda18} allow us to argue by compactness avoiding the explicit constructions of \cite{Cheeger-Colding96}. The price we have to pay is that the results become less local in nature w.r.t.\ \cite[Lemma 1.21]{CheegerNaber15}. Still they are sufficient for our purposes.

The first result presented below, \Cref{prop:auxpropdelta}, corresponds to the rough statement ``the existence of a $\delta$-splitting map with $k$ components implies that the m.m.s.\ is $\eps$-close to a product $\R^k\times Z$''. The second one, \Cref{prop:epsclosedeltasplit}, ensures that, over an $\RCD(-\eps,N)$ space $\eps$-close to a product $\R^k\times Z$, one can build a $\delta$-splitting map with $k$ components.

In order to shorten the notation for the rest of the paper we write $(\R^k\times Z, (0^k,z))$ to denote the p.m.m.s.\ $(\R^k\times Z,\sfd_{eucl}\times \sfd_Z, \Leb^k\times \mm_Z, (0^k,z))$.

\begin{proposition}\label{prop:auxpropdelta}
Let $N>1$ be fixed. Then, for any $\eps>0$, there exists $\delta=\delta_{N,\varepsilon}>0$ such that, for any $\RCD(-\delta,N)$ m.m.s.\ $(X,\sfd,\mm)$ and for any $x\in X$, if there exists a map $u:B_{\delta^{-1}}(x)\to\R^k$ such that $u$ is a $\delta$-splitting map over $B_s(x)$ for any $0<s<\delta^{-1}$, then 
\begin{equation*}
\sfd_{pmGH}\left((X,\sfd,\mm,x),(\R^k\times Z, (0^k,z))\right)< \eps
\end{equation*}
for some pointed $\RCD(0,N-k)$ metric measure space $(Z,\sfd_{Z},\mm_Z,z)$.
\end{proposition}

\begin{proof}
We wish to prove the sought conclusion arguing by contradiction. To this aim let us suppose that, for any $n\ge 1$, there exist an $\RCD(-1/n,N)$ m.m.s. $(X_n,\sfd_n,\mm_n)$, a point $x_n\in X_n$ and a map $u_n:B_{n}(x_n)\to\R^k$ which is a $1/n$-splitting map when restricted to $B_s(x_n)$ for any $0<s<n$. Up to extracting a subsequence, that we do not relabel, we can assume that $(X_n,\sfd_n,\mm_n,x_n)$ converge in the pmGH-topology to an $\RCD(0,N)$ p.m.m.s.\ $(X_\infty,\sfd_\infty,\mm_\infty,x_\infty)$. Here we have used the stability and compactness property of $\RCD(K,N)$ spaces, cf.\ \Cref{remark:stability}. We claim that $X_\infty$ splits off a factor $\R^k$. Observe that, if this is the case, then we reach the sought contradiction. The rest of this proof is dedicated to establishing the claim.

We wish to prove that there exists a function $v:X_\infty\to\R^k$ such that, letting $v:=(v^1,\dots,v^k)$, it holds that $v^i$ is Lipschitz, harmonic and with vanishing Hessian for any $i=1,\dots,k$ and $\nabla v^i\cdot\nabla v^j=\delta_{ij}$ $\mm_\infty$-a.e.\ for any $i,j=1,\dots,k$. The function $v$ will be obtained as a limit function of the $1/n$-splitting maps $u_n:B_{n}(x_n)\to\R^k$. Indeed, since by the assumption in the defining condition of a $\delta$-splitting map the $u_n$ are $C_N$-Lipschitz for any $n\in\N$ and we can assume without loss of generality that $u_n(x_n)=0^k$ for any $n\in\N$, by a generalized version of the Ascoli--Arzel\`{a} theorem (\Cref{thm:AscoliArzela}) we can infer the existence of $v:X_\infty\to\R^k$ such that $u_n$ converge to $v$ locally uniformly on $B_R(x_n)$ for any $R>0$. As a consequence, it is easy to check that $u_n$ converge strongly in $L^2$ (see \Cref{def:L2convergence}) to $v$ on $B_R(x_n)$ for any $R>0$. Since the functions $u_n$ are harmonic on $B_{2R}(x_n)$, at least for $n$ sufficiently large, by \Cref{thm:stabilityLaplacian} and \Cref{prop:sum L1 convergence} it follows that $v$ is harmonic and that, for any $R>0$ and $i,j=1,\dots,k$,
\begin{equation*}
\fint_{B_R(x_\infty)}\abs{\nabla v^i\cdot\nabla v^j-\delta_{ij}}\d\mm_\infty=\lim_{n\to\infty}\fint_{B_R(x_n)}\abs{\nabla u^i_n\cdot\nabla u^j_n-\delta_{ij}}\d\mm_n=0.
\end{equation*}
Hence $\nabla v^i\cdot\nabla v^j=\delta_{ij}$ $\mm_\infty$-a.e.\ on $X_\infty$.

Since $(X_\infty,\sfd_\infty,\mm_\infty)$ is an $\RCD(0,N)$ m.m.s., from $\Delta v^i=0$ and $\abs{\nabla v^i}^2=1$ we infer by \eqref{eq:local Hessian estimate} that $\Hess v^i=0$, for any $i=1,\dots,k$. All in all we get by a standard argument (cf.\ the proof of \cite[Lemma 1.21]{ABS19}) that $X_\infty$ splits a factor $\R^k$, as we claimed.
\end{proof}

\begin{corollary}\label{cor:deltasplitclose}
Let $N>1$ and $K\in \R$ be fixed. For any $\varepsilon>0$ there exists $\delta>0$ such that, for any $r>0$, for any $\RCD(K,N)$ m.m.s.\ $(X,\sfd,\mm)$ and for any $x\in X$, if there exists $u:B_r(x)\to\R^k$ such that $u:B_s(x)\to\R^k$ is a $\delta$-splitting map for any $0<s<r$, then for any $(Y,\varrho,\mu,y)\in\Tan_x(X,\sfd,\mm)$ there exists an $\RCD(0,N-k)$ p.m.m.s.\ $(Z,\sfd_Z,\mm_Z,z)$ such that
\begin{equation*}
\sfd_{pmGH}\left((Y,\varrho,\mu,y),(Z\times\R^k,(z,0^k))\right)<\eps.
\end{equation*}
\end{corollary}

\begin{proof}
Choose $\delta=\delta(K,N,\varepsilon/2)$ given by \Cref{prop:auxpropdelta}.  
If $(Y,\varrho,\mu,y)\in\Tan_x(X,\sfd,\mm)$ then there exists $t>0$ such that $t^{-1}r>\delta^{-1}$, $t^2\abs{K}\le\delta$ and
\begin{equation}\label{eq:firstest}
\sfd_{pmGH}\left((X,t^{-1}\sfd,\mm_{x}^{t},x),(Y,\varrho,\mu,y)\right)<\varepsilon/2.
\end{equation} 
Thanks to \Cref{prop:auxpropdelta}, applied to $(X,t^{-1}\sfd,\mm_{x}^{t},x)$, there exists an $\RCD(0,N-k)$ p.m.m.s.\ $(Z,\sfd_Z,\mm_Z,z)$ such that
\begin{equation}\label{eq:scndest}
\sfd_{pmGH}\left((X,t^{-1}\sfd,\mm_{x}^{t},x),(Z\times\R^k,(z,0^k))\right)<\varepsilon/2.
\end{equation}
The conclusion follows from \eqref{eq:firstest} and \eqref{eq:scndest} by the triangle inequality.
\end{proof}

\begin{proposition}\label{prop:epsclosedeltasplit}
Let $N>1$ be fixed. For any $\delta>0$ there exists $\varepsilon=\eps_{N,\delta}>0$ such that, if $(X,\sfd,\mm)$ is an $\RCD(-\varepsilon,N)$ m.m.s., $x\in X$ and
\begin{equation*}
\sfd_{pmGH}\left(\left(X,\sfd,\mm,x\right),\left(\R^k\times Z,(0^k,z)\right)\right)< \eps
\end{equation*}
for some pointed $\RCD(0,N-k)$ metric measure space $(Z,\sfd_{Z},\mm_Z,z)$, then there exists a $\delta$-splitting map $u:B_5(x)\to\R^k$.
\end{proposition}

\begin{proof}
We are going to build upon the local convergence and stability results that we recalled in \Cref{subsubsection:stability results}, arguing by contradiction.

Suppose the conclusion to be false, then we could find a sequence of pointed $\RCD(-1/n,N)$ m.m.\ spaces $(X_n,\sfd_n,\mm_n,x_n)$ such that, for some $\RCD(0,N-k)$ pointed m.m.s.\ $(Z,\sfd_Z,\mm_Z,z)$ it holds that
\begin{equation*}
\sfd_{pmGH}\left(\left(X_n,\sfd_n,\mm_n,x_n\right),\left(\R^k\times Z,(0^k,z)\right)\right)<1/n
\end{equation*} 
for any $n\ge 1$. Furthermore there should be $\delta_0>0$ such that there is no $\delta_0$-splitting map over $B_5(x_n)$ for any $n\ge 1$.

Let $v:Z\times\R^k\to\R^k$ be defined by $v(p,x)=x$ and denote by $v^1,\dots,v^k$ its components (they are the coordinate functions of the split factor). Observe that $\Delta v^i=0$ for any $i=1,\dots,k$ and $\nabla v^i\cdot\nabla v^j=\delta_{ij}$ for any $i,j=1,\dots,k$. In particular, $v^i$ is harmonic on $B_{10}((z,0^k))$. Hence we can apply \Cref{prop:harmonic approximation} to get harmonic functions $v^i_n:B_9(x_n)\to\R$ that converge strongly in $H^{1,2}$ to $v^i$ on $B_9((z,0^k))$.

Observe that, thanks to \cite[Theorem 1.1]{Jang14},  we can assume that $v_n^i$ is $C_N$-Lipschitz for any $n\in\N$ and for any $i=1,\dots,k$. We wish to prove that $v_n=(v_n^1,\dots,v_n^k)$ is a $\delta_0$-splitting map on $B_5(x_n)$ for $n$ sufficiently big.

To this aim let us recall that \Cref{thm:stabilityLaplacian} yields strong $L^1$-convergence of $\nabla v_n^i\cdot\nabla v_n^j$ to $\delta_{ij}$ on $B_9((z,0^k))$ and on $B_5((z,0^k))$ for any $i,j=1,\dots,k$ (as a consequence of the $L^1$ convergence of $\nabla v_n^i\cdot\nabla v_n^i$ and of $\nabla (v_n^i+v_n^j)\cdot\nabla(v_n^i+v_n^j)$). In particular, due to the uniform boundedness of the gradients we obtained above, we get
\begin{equation*}
\lim_{n\to\infty}\fint_{B_R(x_n)}\abs{\nabla v_n^i\cdot\nabla v_n^j-\delta_{ij}}\d\mm_n=0,
\end{equation*}
for any $i,j=1,\dots,k$ and for any $R=5,9$. The choice $R=5$ gives that the second defining condition of $\delta$-splitting map is satisfied for $n$ sufficiently large and we are left with the verification of the third one.
We wish to prove that
\begin{equation*}
\lim_{n\to\infty}\int_{B_5(x_n)}\abs{\Hess v_n^i}^2\d\mm_n=0
\end{equation*}  
for any $i=1,\dots,k$. To this aim we choose cut-off functions $\eta_n$ for the pairs $B_5(x_n)\subset B_9(x_n)$ as in \Cref{lemma:good cut-off} and, taking into account \eqref{eq:local Hessian estimate}
\begin{equation}\label{eq:bochtohess}
\int_{B_9(x_n)}\Delta\eta_n\left(\abs{\nabla v_n^i}^2-1\right)\d\mm_n+C_N\frac{\mm_n(B_9(x_n))}{n}\ge\int_{B_5(x_n)}\abs{\Hess v_n^i}^2\d\mm_n
\end{equation}
for any $i=1,\dots,k$ and for any $n\ge 1$. Since, $|\Delta \eta_n|\le C_N$ by construction and as we already observed, $\abs{\nabla v_n^i}^2-1$ converge to $0$ in $L^1(B_9)$ and they are uniformly bounded, we get that the left-hand side in \eqref{eq:bochtohess} converges to $0$ as $n\to\infty$. Hence
\begin{equation*}
\lim_{n\to\infty}\fint_{B_5(x_n)}\abs{\Hess v_n^i}^2\d\mm_n=0,
\end{equation*}
as we claimed. 
\end{proof}

Arguing by scaling starting from \Cref{prop:epsclosedeltasplit}, it is possible to obtain the following statement.

\begin{corollary}\label{rm:epsclosedeltascale}
If $(X,\sfd,\mm)$ is an $\RCD(K,N)$ m.m.s., $r^2\abs{K}\le\varepsilon$ and
\begin{equation*}
\sfd_{pmGH}\left(\left(X,r^{-1}\sfd,\mm_x^r,x\right),\left(\R^k\times Z,(0^k,z)\right)\right)< \eps
\end{equation*}
for some pointed $\RCD(0,N-k)$ metric measure space $(Z,\sfd_{Z},\mm_Z,z)$, then there exists a $\delta$-splitting map $u:B_{5r}(x)\to\R^k$.
\end{corollary}

\subsubsection{Propagation of the $\delta$-splitting property}
In the next result we are concerned with the propagation of the property of being a $\delta$-splitting map. We are going to prove that, if $\alpha\in(0,2)$, outside a set of small codimension-$\alpha$ content any $\delta$-splitting map at a given scale is a $C_{N,\alpha}\delta^{1/4}$ splitting map at any scale. The proof is based on a weighted maximal function argument.

\begin{proposition}
	\label{proposition:propagation}
	 Let $\alpha\in (0,2)$ and $N>1$. There exist constants $C_N>0$ and $C_{N,\alpha}>0$ such that, for any $0<\delta<1$, any $\RCD(-1,N)$ m.m.s.\ $(X,\sfd,\mm)$, any $p\in X$ and for any $\delta$-splitting map $u:=(u_1,\ldots,u_k):B_2(p)\to\R^k$, there exists a Borel set $G\subset B_1(p)$ with $\haus_5^{h_{\alpha}}(B_1(p)\setminus G)<C_N\sqrt{\delta}\mm(B_2(p))$ such that for any $x\in G$ it holds
	\begin{equation}\label{eq: Hessian improvement}
	\sup_{0<r<1}r^{\alpha}\fint_{B_r(x)} |\Hess u_a|^2\d \mm \le \sqrt{\delta} 
	\qquad\text{for any}\quad a=1,\ldots,k,
	\end{equation}
	and
	\begin{equation}\label{eq: propagation of splitting}
	u: B_r(x)\to \R^k
	\qquad\text{is a}\ C_{N,\alpha}\delta^{1/4}\text{-splitting map for any $0<r<1/2$.}
	\end{equation}
\end{proposition}

\begin{proof}
	Let us start proving \eqref{eq: Hessian improvement}. To this aim fix any $a=1,\ldots,k$ and denote by $C_P$ and $C_D$ the Poincar\'{e} and the doubling constants over balls of radius $10$ of $(X,\sfd,\mm)$.
	To be more precise $C_P$ is a constant in the $(1,2)$-Poincar\'{e} inequality with $\lambda=2$ as in \eqref{eq:Poincare_ineq}. This inequality is available on $\RCD(K,N)$ m.m.s.\ (see for instance \cite[Theorem 30.26]{Villani09}) with constant depending only on $K$ and $N$. In particular, since $(X,\sfd,\mm)$ is an $\RCD(-1,N)$, $C_P$ depends only on $N$. The same conclusion holds for $C_D$ thanks to the Bishop-Gromov inequality \eqref{eq:BishopGromovInequality}.
	
	Set
	\begin{equation*}
	G:=\left\lbrace x\in B_1(p):\ \sup_{0<r<1} r^{\alpha}\fint_{B_r(x)} |\Hess u_a|^2\d \mm \le \sqrt{\delta}\right\rbrace. 
	\end{equation*}
	We claim that $\haus_5^{h_{\alpha}}(B_1(p)\setminus G)<C_N\sqrt{\delta}\mm(B_2(p))$. For any $x\in B_1(p)\setminus G$ we choose $\rho_x\in (0,1)$ satisfying
	\begin{equation}\label{z2}
	\rho_x^{\alpha}\fint_{B_{\rho_x}(x)} |\Hess u_a|^2\d \mm >\sqrt{\delta}.
	\end{equation}
	Observe that the family $\{B_{\rho_x}(x)\}_{x\in B_1(p)\setminus G}$ covers $B_1(p)\setminus G$.
	Using Vitali's covering lemma we can find a subfamily of disjoint balls $\left\lbrace  B_{\rho_i}(x_i)\right\rbrace _{i\in\N}$ such that $B_1(p)\setminus G\subset \cup_{i\in\N}B_{5\rho_i}(x_i)$. This gives the sought conclusion
	\begin{align*}
	\haus_5^{h_\alpha}(B_1(p)\setminus G )\le & \sum_{i\in\N} h_{\alpha}(B_{5\rho_i}(x_i))
	=  \sum_{i\in\N} \frac{\mm(B_{5{\rho_i}}(x_i))}{(5\rho_i)^{\alpha}}\\
	\le & C_N \sum_{i\in\N} \frac{\mm(B_{{\rho_i}}(x_i))}{\rho_i^{\alpha}}
	\le  C_N \sum_{i\in\N}\frac{1}{ \sqrt{\delta}}\int_{B_{\rho_i}(x_i)} |\Hess u_a|^2\d \mm\\
	\le & C_N\frac{1}{\sqrt{\delta}}\int_{B_2(p)} |\Hess u_a|^2 \d \mm
	\le  C_N \sqrt{\delta}\mm(B_2(p)),
	\end{align*}
	where we used the definition of $\haus^{h_{\alpha}}_5$, the Bishop-Gromov inequality, \eqref{z2} and the fact that $u$ is a $\delta$-splitting map. 
	
	In order to verify \eqref{eq: propagation of splitting} we just need to check that, for $a,b=1,\ldots,k$,
	\begin{equation*}
	\fint_{B_r(x)} |\nabla u_a\cdot \nabla u_b-\delta_{a,b}|\d \mm <C_{N,\alpha} \delta^{1/4}
	\quad\text{for any }x\in G,\ 0<r<1.
	\end{equation*}
	To this aim let us set $f_{a,b}:= |\nabla u_a\cdot \nabla u_b-\delta_{a,b}|$ and note that $|\nabla f_{a,b}|\le C_N(|\Hess u_a|+|\Hess u_b|)$ as a consequence of \Cref{def:deltasplitting}(i) and \eqref{eq:calculusrule}.
	Whence, the Poincar\'{e} inequality and \eqref{eq: Hessian improvement} yield
	\begin{align*}
	\abs{\fint_{B_{r}(x)}f_{a,b}\,\d \mm-\fint_{B_{r/2}(x)}f_{a,b}\,\d \mm}
	\le & C_P r\left(\fint_{B_{2r}(x)}|\nabla f_{a,b}|^2\,\d \mm\right)^{1/2} \\
	\le & C_N C_P\left( r^2\fint_{B_{2r}(x)} |\Hess u_a|^2\,\d \mm+r^2\fint_{B_{2r}(x)} |\Hess u_b|^2\,\d \mm\right)^{1/2}\\
	\le & C_N C_P\delta^{1/4}r^{1-\alpha/2}
	\end{align*}
	for any $0<r<1/2$. Applying a telescopic argument it is simple to see that
	\begin{equation*}
	\abs{\fint_{B_{2^{-1}}(x)}f_{a,b}\,\d\mm-\fint_{B_{2^{-k}}(x)}f_{a,b}\,\d \mm}\le C_{\alpha}C_N C_P \delta^{1/4},\qquad \text{for any}\ k>1.
	\end{equation*}
	Therefore, for any $0<r<1/2$ we take $k\in\N$ such that $2^{-k-1}<r\le 2^{-k}$ and using that $u:B_2(p)\to \R^k$ is a $\delta$-splitting map we get
	\begin{align*}
	\fint_{B_{r}(x)}f_{a,b}\,\d \mm \le & C_D2^N\fint_{B_{2^{-k}}(x)} f_{a,b}\,\d\mm\\
	\le & C_D2^N\abs{\fint_{B_{1/2}(x)}f_{a,b}\d\mm-\fint_{B_{r}(x)}f_{a,b}\,\d \mm}+C_D2^N\fint_{B_{1/2}(x)}f_{a,b}\,\d \mm\\
	\le & 2^NC_DC_{\alpha}C_N C_P \delta^{1/4}+ 8^NC_D^2\fint_{B_2(p)} f_{a,b}\,\d \mm\\
	\le & C_{N,\alpha}\delta^{1/4}.
	\end{align*}
\end{proof}

For the purposes of this paper we just need to consider the case $\alpha=1$ in \Cref{proposition:propagation}. This is related to the fact that boundaries of sets with finite perimeter are codimension one objects. In order to shorten the notation in the sequel we will write $h$ in place of $h_1$.\\
We are going to use several times the following scale invariant version of \Cref{proposition:propagation}. 

\begin{corollary}\label{cor:scale invariant version of propagation}
Let $(X,\sfd,\mm,p)$ be an $\RCD(K,N)$ p.m.m.s.\ and $u:B_{4r}(p)\to \R^k$ a $\delta$-splitting map for some $r>0$ such that $|K|r^2\le 4$ and $r<1/2$. Then there exists $G\subset B_{2r}(p)$ with 
	\begin{equation*}
	\haus_{5}^{h}(B_{2r}(p)\setminus G)
	\le \haus_{10 r}^{h}(B_{2r}(p)\setminus G)
	\le C_N \sqrt{\delta}\frac{\mm(B_{2r}(p))}{2r}
	\end{equation*}
	such that $u:B_s(x)\to \R^k$ is a $C_N\delta^{1/4}$-splitting map for any $x\in G$ and any $0<s<r$.
\begin{proof}
Apply \Cref{proposition:propagation} to the rescaled space $(X,(2r)^{-1}\sfd,\mm(B_{2r}(p))^{-1}\mm,p)$.
\end{proof}
\end{corollary}

\subsection{Uniqueness of tangents and consequences}
Let $(X,\sfd,\mm)$ be an $\RCD(K,N)$ metric measure space with essential dimension $n\le N$ (see \Cref{thm:structure theory}) and let $E\subset X$ be a set of locally finite perimeter. For any $k=1,\ldots,n$ we set
\begin{align*}
A_k:=\bigl\{&x\in X\;:\;\left(\R^k,\sfd_{eucl},c_k\Leb^k,0^k,\left\lbrace x_k>0\right\rbrace \right)\in \Tan_x(X,\sfd,\mm, E),\text{ but for no }(Y,\varrho,\mu,y)\ \text{with}\\
&\text{diam}(Y)>0\text{ it holds }(Y\times\R^k,\varrho\times\sfd_{eucl},\mu\times\Leb^k,(y,0^k), \left\lbrace x_k>0\right\rbrace )\in \Tan_x(X,\sfd,\mm,E)\bigl\}.
\end{align*}
Let us point out that, with arguments analogous to those in \cite[Lemma 6.1]{Mondino-Naber14} one can show that $A_k$ is a $\abs{D\nchi_E}$-measurable set for any $k=1,\ldots,n$.

Aiming at proving that the family $\left\lbrace A_k\right\rbrace_{k=1,\ldots,n}$ covers $X$ up to a $|D\nchi_E|$-negligible set we need to use the following result that has been proven in the appendix of \cite{ABS18}.

\begin{theorem}\label{th: tangent of tangent}
	Let $(X,\sfd,\mm)$ be an $\RCD(K,N)$ m.m.s.\ and let $E\subset X$ be a set of locally finite perimeter. 
	Then for $|D\nchi_E|$-a.e.\ $x\in X$ the following property holds true:
	for every $(Y,\varrho,\mu,y, F)\in \Tan_x(X,\sfd,\mm,E)$ one has
	\begin{equation*}
	\Tan_{y'}(Y,\varrho,\mu,F)\subset \Tan_x(X,\sfd,\mm,E)\qquad\text{for every $y'\in \supp |D\nchi_F|$.}
	\end{equation*}
\end{theorem}

\begin{lemma}\label{lemma:Ak}
	Under the assumptions above
	\begin{equation*}
	|D\nchi_E|\left(X\setminus \bigcup_{k=1}^n A_k\right)=0.
	\end{equation*}
\end{lemma}
\begin{proof}
	As a consequence of \Cref{thm:tangenthalfspace} we have
	\begin{equation*}
	|D\nchi_E|\left(X\setminus \bigcup_{k=1}^n A'_k\right)=0,
	\end{equation*}
	where
	\begin{align*}
	A'_k:=\bigl\{ x  \in X:\ &\left(\R^k,\sfd_{eucl},c_k\Leb^k,0^k,\left\lbrace x_k>0\right\rbrace \right)\in \Tan_x(X,\sfd,\mm,E)\ \text{but}\\
	& \left(\R^m,\sfd_{eucl},c_m\Leb^m,0^m,\left\lbrace x_m>0\right\rbrace \right)\notin \Tan_x(
	X,\sfd,\mm,E)\ \text{for any}\ m>k \bigl\}.
	\end{align*}
	The measurability of the $A'_k$'s can be verified as in the case of the $A_k$'s.
		
	It is clear that $A_k\subset A'_k$, let us prove $|D\nchi_E|(A'_k\setminus A_k)=0$. We argue by contradiction. If the claim is false we can find $x\in A'_k\setminus A_k$ such that the iterated tangent property of \Cref{th: tangent of tangent} holds true.
	Since $x\in A'_k\setminus A_k$ we can find $(Y,\varrho,\mu,y)\in \RCD(0,N-k)$ with $\diam(Y)>0$ such that
	\begin{equation*}
	(Y\times\R^k,\varrho\times\sfd_{eucl},\mu\times\Leb^k,(y,0^k),\left\lbrace x_k>0\right\rbrace )\in \Tan_x(X,\sfd,\mm,E).
	\end{equation*}
	Moreover $\Tan_{(y',x,0)}(Y\times \R^k,\varrho\times\sfd_{eucl},\mu\times\Leb^k,\left\lbrace x_k>0\right\rbrace )\subset\Tan(E,x)$ for any $(y',x)\in Y\times \R^{k-1}$, thanks to \Cref{th: tangent of tangent}. Thus, choosing $(y',x,0)\in Y\times\R^k$ such that \Cref{thm:tangenthalfspace} holds and $y'$ is regular in $Y$ we get the sought contradiction, since the essential dimension of $Y$ is bigger or equal than one (otherwise $\diam(Y)=0$).
\end{proof}

    We are now in a position to conclude the proof of \Cref{th:uniqueness}. 
\begin{proof}[Proof of \Cref{th:uniqueness}]    
    In light of \Cref{lemma:Ak} it is enough to prove that $A_k$ coincides up to a $|D\nchi_E|$-negligible set with
    \begin{equation*}
		\left\lbrace x\in X:\ \Tan_x(X,\sfd,\mm,E)=\left\lbrace (\R^k,\sfd_{eucl},c_k\Leb^k,0^k,\left\lbrace x_k>0\right\rbrace )\right\rbrace \right\rbrace .
    \end{equation*}
	Let us assume without loss of generality that $A_k\subset B_2(p)$ for some $p\in X$.
	We claim that, for any $\eta>0$, there exists $G^{\eta}\subset A_k$ with
	\begin{equation}\label{z4}
		\haus_{5}^{h}(A_k\setminus G^{\eta})\le C_N \eta \Per(E, B_2(p))
	\end{equation}
	such that, for any $x\in G^{\eta}$ and for any $(Y,\varrho,\mu,y)\in\Tan_x(X,\sfd,\mm)$, there exists a pointed $\RCD(0,N-k)$ m.m.s.\ $(Z,\sfd_Z,\mm_Z,z)$ satisfying
	\begin{equation}\label{z3}
		\sfd_{pmGH}((Y,\varrho,\mu,y),(\R^k\times Z,(0,z))\le \eta.
	\end{equation}
	Observe that the claim implies our conclusion. Indeed if we fix $\eta>0$ and set $\eta_i:=\eta 2^{-i}$ then $G_{\eta}:=\cup_{i\in \N}G^{\eta_i}$ satisfies
    $\haus_5^{h}(A_k\setminus G_{\eta})=0$ and thus $\Per(E,A_k\setminus G_{\eta})=0$ thanks to \Cref{lemma:AssolutacontinuitaHauss}. Moreover, for any $x\in G_{\eta}$, \eqref{z3} holds. We conclude observing that $G:=\cap_{k\in \N} G_{2^{-k}}$ still satisfies $\Per(E, A_k\setminus G)=0$ and any tangent cone at $x\in G$ splits off a factor $\R^k$. By definition of $A_k$ we deduce that the only tangent at $x\in G$ is the Euclidean space of dimension $k$.
    
	Let us pass to the verification of the claim. 
	Fix $\delta\in (0,1/2)$ and take $\eps>0$ as in \Cref{prop:epsclosedeltasplit}. Of course we can assume $\eps\le \delta$.
    We wish to prove that there exists a disjoint family of balls $\left\lbrace B_{r_i}(x_i)\right\rbrace_{i\in\N}$ such that $r_i^2|K|\le \eps$ for any $i\in\N$ and
	\begin{itemize}
	    \item[(i)] $A_k\cap B_1(p)\subset\cup_{i\in \N}B_{5r_i}(x_i)$;
		\item[(ii)] $\sfd_{pmGH}\left((X,r_i^{-1}\sfd ,\mm_x^{r_i},x_i),(\R^k,\sfd_{eucl},c_k\Leb^k,0^k)\right)\le \eps$;
		\item[(iii)] $\frac{\omega_{k-1}}{\omega_{k}}(1-\eps)\frac{\mm(B_{r_i}(x_i))}{r_i}\le \Per(E,B_{r_i}(x_i))\le \frac{\omega_{k-1}}{\omega_{k}}(1+\eps)\frac{\mm(B_{r_i}(x_i))}{r_i}$.
	\end{itemize}	
    Indeed, for any $x\in A_k$ there exists a sequence of radii $r_i\to 0$ such that
	\begin{equation*}
	\lim_{i\to\infty}\sfd_{pmGH}\big((X,r_i^{-1}\sfd ,\mm_x^{r_i},x),(\R^k,\sfd_{eucl},\Leb^k,0^k)\big)=0
	\quad\text{and}\quad \lim_{i\to \infty}\frac{r_i\Per(E,B_{r_i}(x))}{\mm(B_{r_i}(x))}=\frac{\omega_{k-1}}{\omega_{k}},
	\end{equation*}
    as a consequence of \Cref{thm:tangenthalfspace}, see also \eqref{eq:asympotic of different normalization}. Therefore, for any $x\in A_k$ we can choose $r_x^2|K|\le \eps$ such that the pair $(x,r_x)$ satisfies (ii) and (iii). In order to get a disjoint family of balls satisfying (i) we have just to apply Vitali's Lemma to $\left\lbrace B_{r_x}(x)\right\rbrace_{x\in A_k\cap B_1(p)}$. 

    Let us now focus the attention on a single ball $B_{20r_i}(x_i)\subset X$.
	\Cref{rm:epsclosedeltascale} yields the existence of a $\delta$-splitting map
	\begin{equation*}
		u^i:B_{5r_i}(x_i)\to \R^k.
	\end{equation*}
	Thanks to \Cref{cor:scale invariant version of propagation} we can find $G_i\subset B_{5r_i}(x_i)$ with 
	\begin{equation}\label{z5}
	\haus_5^{h}(B_{5r_i}(x_i)\setminus G_i)\le C_N\sqrt{\delta}\frac{\mm(B_{5r_i}(x_i))}{5r_i}
	\end{equation}
    and such that $u^i:B_s(x)\to \R^k$ is a $C_N\delta^{1/4}$-splitting map for any $x\in G_i$ and any $0<s<5r_i$.
	Applying \Cref{cor:deltasplitclose}, up to assuming $\delta$ small enough, we deduce that at any $x\in G_i$ \eqref{z3} holds true.\\
	To conclude let us verify that $G:=\cup_{i\in\N} G_i$ satisfies \eqref{z4}. 
	Using (iii), \eqref{z5} and the Bishop-Gromov inequality \eqref{eq:BishopGromovInequality} we get
    \begin{align*}
    \haus_5^{h}(A_k\setminus G)\le & \sum_{i\in\N}\haus_{5}^{h_1}(B_{5r_i}(x_i)\setminus G_i)
    \le \sum_{i\in\N}C_N\sqrt{\delta}\frac{\mm(B_{5r_i}(x_i))}{5r_i}\\
    \le & C_N\sqrt{\delta}\sum_{i\in\N} \frac{\mm(B_{r_i}(x_i))}{r_i}
    \le C_N\sqrt{\delta}\sum_{i\in\N} \Per(E, B_{r_i}(x_i))\\
    \le & C_N\sqrt{\delta} \Per(E, B_2(p)).
    \end{align*}
   Since we can assume $\delta<\eta^2$ we get the sought estimate.
\end{proof}

Let $(X,\sfd,\mm)$ be an $\RCD(K,N)$ metric measure space and $E\subset X$ a set of locally finite perimeter. 
For any $k=1,\ldots,n$, where $n$ is the essential dimension of $(X,\sfd,\mm)$, we set 
\begin{equation*}
\mathcal{F}_kE:=\left\lbrace x\in X\;:\;\Tan_x(X,\sfd,\mm,E)=\left\lbrace (\R^k,\sfd_{eucl},c_k\Leb^k,0^k,\left\lbrace x_k>0\right\rbrace )\right\rbrace \right\rbrace .
\end{equation*}
We know thanks to \Cref{th:uniqueness} that $\Per(E,\cdot)$ is concentrated on $\mathcal{F}E:=\cup_{k=1}^n\mathcal{F}_kE$ and, from now on, we shall call $\mathcal{F}E$ the reduced boundary of $E$. 

The result about uniqueness of tangents that we just proved allows to obtain a representation formula for the perimeter measure in terms of the codimension-$1$ Hausdorff measure. 

\begin{corollary}\label{cor:reprper}
Let $(X,\sfd,\mm)$ be an $\RCD(K,N)$ m.m.s.\ with essential dimension $n$. Let $E\subset X$ be a set of locally finite perimeter. Then 
\begin{equation}\label{eq:reprperimeter}
\abs{D\nchi_E}=\sum_{k=1}^{n}\frac{\omega_{k-1}}{\omega_k}\haus^{h}\res\mathcal{F}_kE.
\end{equation}
\end{corollary}
\begin{proof}
The proof can be obtained as in the case of the representation formula for the perimeter on non-collapsed spaces obtained in \cite[Corollary 4.7]{ABS18} relying on \cite[Theorem 3]{Magnani15} in place of \cite[Theorem 5]{Magnani15}. We just report here the key computation.

If $x\in\mathcal{F}_kE$, then we can compute 
\begin{align*}
\lim_{r\to 0}\frac{r\abs{D\nchi_E}(B_r(x))}{\mm(B_r(x))}=&\lim_{r\to 0}\frac{r\abs{D\nchi_E}(B_r(x))}{C(x,r)}\cdot\frac{C(x,r)}{\mm(B_r(x))}=\lim_{r\to 0}\frac{\abs{D^r\nchi_E}(B_1(x))}{\mm_x^r(B_1(x))}\\
=&\frac{\haus^{k-1}(B_1(0))}{\haus^{k}(B_1(0))}=\frac{\omega_{k-1}}{\omega_k},
\end{align*}
where the regularity of the point and the weak convergence of the rescaled perimeter measures to the perimeter measure of a half-space play a role.\\
This computation, together with the rigid structure of the tangent, allows then to infer, arguing as in the non-collapsed case, that
\begin{equation*}
\lim_{r\to 0}\sup_{x\in B_s(y),\ s\le r}\frac{s\abs{D\nchi_E}(B_s(y))}{\mm(B_s(y))}=\frac{\omega_{k-1}}{\omega_k},
\end{equation*}
which is the needed density estimate in order to obtain the representation formula \eqref{eq:reprperimeter}.
\end{proof}

\section{Rectifiability of the reduced boundary}\label{sec:bdryrect}
The main achievement of this section is a rectifiability result for the reduced boundary of sets with finite perimeter. With this theorem we complete the picture about the generalization of De Giorgi's theorem to the framework of $\RCD(K,N)$ spaces.

\begin{theorem}\label{th: rectifiability}
	Let $(X,\sfd,\mm)$ be an $\RCD(K,N)$ m.m.s.\ and $E\subset X$ be a set of locally finite perimeter. Then, for any $k=1,\dots,n$, $\mathcal{F}_k E$ is $\big(\abs{D\nchi_E},(k-1)\big)$-rectifiable.
\end{theorem}
Let us recall that a set is $\big(\abs{D\nchi_E},\ell\big)$-rectifiable if up to a $\abs{D\nchi_E}$-negligible set it can be covered by $\cup_{i\in \N} A_i$ where any $A_i$ is bi-Lipschitz equivalent to a Borel subset of $\R^\ell$.

When specialized to the non-collapsed case (see \cite{DPG18}),
where the only non-empty regular set is the top dimensional one, \Cref{th: rectifiability} turns into:
\begin{corollary}\label{cor:Noncollapsed}
Let $(X,\sfd,\mm)$ be a $\ncRCD(K,N)$ m.m.s.\ and $E\subset X$ a set of locally finite perimeter. Then $\mathcal{F}E=\mathcal{F}_NE$ is $\left(\abs{D\nchi_E},N-1\right)$-rectifiable (equivalently, $\left(\mathcal{H}^{N-1},N-1\right)$-rectifiable, where $\mathcal{H}$ denotes the $(N-1)$-dimensional Hausdorff measure). Furthermore
\begin{equation}
\abs{D\nchi_E}=\mathcal{H}^{N-1}\res\mathcal{F}E.\footnote{In \cite{ABS18} it was proved that $\abs{D\nchi_E}=\mathcal{S}^{N-1}\res\mathcal{F}E$, where $\mathcal{S}$ denotes the spherical Hausdorff measure. Coincidence with the Hausdorff measure $\mathcal{H}$ is a consequence of rectifiability.}
\end{equation}
\end{corollary}

\begin{remark}
	We point out that, given any \(\eps>0\), the maps providing rectifiability of the reduced boundary in \Cref{th: rectifiability} and \Cref{cor:Noncollapsed} can be taken $(1+\eps)$-bi-Lipschitz (compare with the analogous statement in the case of \cite{Mondino-Naber14}).

In particular, when $(X,\sfd,\mm)$ is non-collapsed,  $(X,\sfd,\abs{D\nchi_E})$ is a strongly $|D\nchi_E|$-rectifiable m.m.s.\ according to \cite{GP16}.  
\end{remark}

\begin{remark}
	It is worth mentioning that \Cref{th: rectifiability} is stronger than \cite[Theorem 1.1]{Mondino-Naber14}. Indeed, given an $\RCD(K,N)$ m.m.s.\ $(Z,\sfd_Z,\mm_Z)$ we can consider $X:=Z\times\R$ endowed with the product structure, and the set of finite perimeter $E:=\{(z,t)\in Z\times \R\,:\,t>0\}$. Applying \Cref{th: rectifiability} to $E\subset X$ we get the rectifiability result for $Z$.
\end{remark}

Let us outline the strategy of the proof of \Cref{th: rectifiability}.\\
First of all, up to intersecting with a ball and thanks to the locality of perimeter and tangents, we can assume that $E$ has finite measure and perimeter.\\
The bi-Lipschitz maps from subsets of $\mathcal{F}_kE$ to $\R^{k-1}$ providing rectifiability are going to be \textit{suitable approximations} of the $(k-1)$ coordinate maps over the hyperplane where the perimeter concentrates after the blow-up. Better said, they will be the first $(k-1)$ components of a $(k,\delta)$-splitting map ``$\delta$-orthogonal to the exterior normal $\nu_E$ to the boundary of $E$''.\\ 
Proving existence of these maps requires some technical work which builds upon the Gauss--Green formula \Cref{thm:Gauss-Green}. The rigorous statement is as follows.
\begin{proposition}\label{prop:good aproximation}
	Let $(X,\sfd,\mm)$ be an $\RCD(K,N)$ m.m.\ space and $E\subset X$ a set of finite perimeter and measure. For any $\delta>0$, $r_0>0$ and $|D\nchi_E|$-a.e.\ $x\in \mathcal{F}_kE$ there exist $r=r_{x,\delta}<r_0$ and a $\delta$-splitting map $u=(u_1,\ldots,u_{k-1}):B_r(x)\to \R^{k-1}$ such that
	\begin{equation*}
	\frac{r}{\mm(B_r(x))}\int_{B_r(x)}|\nu \cdot \nabla u_{\alpha}|\,\d|D\nchi_E|<\delta,
	\quad \text{for}\ \alpha=1,\ldots,k-1.
	\end{equation*}
\end{proposition}
The second step in the proof of \Cref{th: rectifiability} is showing that the map built in \Cref{prop:good aproximation} is indeed bi-Lipschitz with its image if restricted to suitable subsets of $\mathcal{F}_kE$ (see \Cref{prop:rectifiability} below for the rigorous statement). These subsets are obtained collecting points $x\in \mathcal{F}_kE$ such that $B_s(x)\cap E$ is $\eps$-close, in a suitable sense, to $B_s(0^k)\cap\left\lbrace x_k>0\right\rbrace$ for any $s\le r_0$, where $r_0>0$ is a fixed radius. 

\begin{definition}\label{def:quantitative boundary}
Given $\eps>0$ and $r_0>0$, we define $(\mathcal{F}_kE)_{r_0,\eps}$ as the set of points $x\in \mathcal{F}_kE$ satisfying
\begin{itemize}
	\item[(i)] $\sfd_{pmGH}\left( \left(X,s^{-1}\sfd,\frac{\mm}{\mm(B_s(x))},x\right), \left( \R^k,\sfd_{\text{eucl}}, \frac{1}{\omega_k}\Leb^k, 0^k\right)\right)<\eps$ for any $s\le r_0$;
	\item [(ii)]
	\begin{equation}
	\abs{\frac{\mm(B_s(x)\cap E)}{\mm(B_s(x))}-\frac{1}{2}}+\abs{\frac{s\abs{D\nchi_E}(B_s(x))}{\mm(B_s(x))}-\frac{\omega_{k-1}}{\omega_{k}}}<\eps\quad \text{for any $s\le r_0$}.
	\end{equation}
\end{itemize}
\end{definition}
Observe that, as a consequence of \Cref{th:uniqueness} and \Cref{remark: limit of different normalization}, for any $\eps>0$ we have 
\begin{equation*}
	\mathcal{F}_kE=\bigcup_{0<r<1}(\mathcal{F}_kE)_{r,\eps}
	\quad\text{and}\quad
	(\mathcal{F}_kE)_{r,\eps}\subset(\mathcal{F}_kE)_{r',\eps}\ \text{for}\ r'<r.
\end{equation*}
Hence for any $\eta>0$ there exists $r=r(\eta)>0$ such that
\begin{equation}\label{z11}
	\abs{D\nchi_E}\big(\mathcal{F}_k E\setminus (\mathcal{F}_kE)_{s,\eps}\big)<\eta,
	\quad \text{for any }\ 0<s<r.
\end{equation}
\begin{proposition}\label{prop:rectifiability}
	Let $N>1$, $K\in\R$ and $k\in [1,N]$ be fixed. For any $\eta>0$ there exists $\eps=\eps(\eta,N)<\eta$ such that, if $(X,\sfd,\mm)$ is an $\RCD( K,N)$ m.m.s., $E\subset X$ is a set of finite perimeter and finite measure, $p\in (\mathcal{F}_k E)_{2s,\eps}$ for some $s\in (0,|K|^{-1/2})$ and there exists an $\varepsilon$-splitting map $u:B_{2s}(p)\to\R^{k-1}$ such that
	\begin{equation}\label{eq:ort}
	\frac{s}{\mm(B_{2s}(x))}
	\int_{B_{2s}(x)}\abs{\nu\cdot \nabla u_a}\,\d|D\nchi_E|<\eps,
	\quad\text{for any
		$a=1,\ldots,k-1$,}
	\end{equation}
	then there exists $G\subset B_s(p)$ that satisfies:
	\begin{itemize}
		\item[(i)] $G\cap (\mathcal{F}_k E)_{2s,\eps}$ is bi-Lipschitz to a Borel subset of $\R^{k-1}$. More precisely,
		\begin{equation}\label{z1}
		\big||u(x)-u(y)|-\sfd(x,y)\big|\le C_N \eta\,\sfd(x,y),
		\quad \forall x,y\in (\mathcal{F}_kE)_{2s,\eps}\cap G;
		\end{equation}
		\item[(ii)] $\haus_5^{h}(B_s(p)\setminus G)< C_N\eta \frac{\mm(B_s(p))}{s}$.
	\end{itemize}
\end{proposition}

Let us now prove \Cref{th: rectifiability} assuming \Cref{prop:good aproximation} and \Cref{prop:rectifiability}.

\begin{proof}[Proof of \Cref{th: rectifiability}]
	Assume without loss of generality that $E$ has finite perimeter and measure, and that $\mathcal{F}_k E\subset B_2(p)$ for some $p\in X$. We claim that, for any $\eta>0$, we can decompose $\mathcal{F}_k E=G^{\eta}\cup B^{\eta}\cup R^{\eta}$, where $G^{\eta}$ is $(k-1)$-rectifiable and 
	\begin{equation}\label{z8}
		\haus^h_5(B^{\eta})+\abs{D\nchi_E}(R^{\eta})\le C_{N,K}\abs{D\nchi_E}(B_2(p))\eta + \eta.
	\end{equation}
	Observe that the claim easily gives the sought conclusion. Indeed, setting $\eta_i:=\eta 2^{-i}$, $G_{\eta}:=\cup_i G^{\eta_i}$ and $R_{\eta}:=\cup_{i\in\N} R^{\eta_i}$, $G_{\eta}$ is still $(k-1)$-rectifiable and it holds
	\begin{equation*}
		\haus_{5}^h((\mathcal{F}_kE\setminus G_{\eta})\setminus R_{\eta})=0,
	\end{equation*}
    hence, as a consequence of \Cref{lemma:AssolutacontinuitaHauss}, $\abs{D\nchi_E}(\mathcal{F}_kE\setminus G_{\eta})\setminus R_{\eta})=0$. Therefore
    \begin{equation*}
    	\abs{D\nchi_E}(\mathcal{F}_k E\setminus G_{\eta})\le \abs{D\nchi_E}(R_{\eta})\le C_N\abs{D\nchi_E}(B_2(p))\eta +\eta.
    \end{equation*}
    Setting $G:=\cup_{i\in\N} G_{2^{-i}}$, we get that $G$ is still $(k-1)$-rectifiable and coincides with $\mathcal{F}_kE$ up to a $\abs{D\nchi_E}$-negligible set.
      	
	Let us now prove the claim. To this aim fix $r>0$ and $\eps>0$. We cover $(\mathcal{F}_kE)_{r,\eps}$ with balls of radius smaller than $r/5$ with center in $(\mathcal{F}_kE)_{r,\eps}$ such that the assumptions of \Cref{prop:rectifiability} are satisfied. The possibility of building such a covering is a consequence of \Cref{th:uniqueness} and of \Cref{prop:good aproximation}. By Vitali's lemma, we can extract a disjoint family $\left\lbrace B_{r_i/5}(x_i)\right\rbrace _{i\in\N}$ such that $(\mathcal{F}_kE)_{r,\eps}\subset \cup_i B_{r_i}(x_i)$. Applying \Cref{prop:rectifiability} above, for any $i\in\N$ we can find $G_i\subset B_{r_i}(x_i)$ such that $G_i\cap (\mathcal{F}_k E)_{r,\eps}$ is $(k-1)$-rectifiable and $\haus_5^{h}(B_{r_i}(x_i)\setminus G_i)< C_N\eta \frac{\mm(B_{r_i}(x_i))}{r_i}$. Set $G_r^{\eta}:= (\mathcal{F}_k E)_{r,\eps}\cap (\cup_i G_i)$ and observe that
	\begin{align*}
	\haus_5^{h}((\mathcal{F}_k E)_{r,\eps}\setminus G^{\eta}_r)\le & \sum_{i\in\N}\haus_{5}^{h}(B_{r_i}(x_i)\setminus G_i)
	\le \sum_{i\in\N}C_N\eta\frac{\mm(B_{r_i}(x_i))}{r_i}\\
	\le & C_N\eta\sum_{i\in\N} \frac{\mm(B_{r_i/5}(x_i))}{r_i/5}
	\le C_{N,K}\eta \sum_{i\in\N} \abs{D\nchi_E}(B_{r_i/5}(x_i))\\
	\le & C_{N,K}\eta \abs{D\nchi_E}(B_2(p)),
	\end{align*}
	where we used the Bishop-Gromov inequality \eqref{eq:BishopGromovInequality} and
	\begin{equation*}
	\frac{\mm(B_{r_i/5}(x_i))}{r_i/5}\le C(k) \abs{D\nchi_E}(B_{r_i/5}(x_i)),
	\end{equation*}
	that holds true provided $\eps$ is small enough.\\
	Setting $B_r^{\eta}:=(\mathcal{F}_k E)_{r,\eps}\setminus G^{\eta}_r$, the argument above gives the decomposition
	\begin{equation*}
		(\mathcal{F}_k E)_{r,\eps}=G_r^{\eta}\cup B_r^{\eta},
	\end{equation*}
	where $G_r^{\eta}$ is $(k-1)$-rectifiable and $\haus_{5}^h(B_r^{\eta})\le C_{N,K}\eta \abs{D\nchi_E}(B_2(p))$. Let us now choose $r>0$ small enough to have \eqref{z11}. This allows us to write
	\begin{equation*}
		\mathcal{F}_k E=G_r^{\eta}\cup B_r^{\eta}\cup (\mathcal{F}_kE\setminus(\mathcal{F}_kE)_{r,\eps})
		=: G^{\eta}\cup B^{\eta}\cup R^{\eta}
	\end{equation*}
	and to conclude the proof.
\end{proof}

\subsection{Proof of \texorpdfstring{\Cref{prop:good aproximation}}{Prop}}\label{subsection:regularityextiriornormal}

Let us start by recalling that one of the main results of the previous part of the note was proving that the exterior normal is indeed an element of $L^2_E(TX)$ (see \Cref{thm:Gauss-Green}). In the following, to simplify the notation, we shall write $v$ in place of $\tr_E(v)$ for any $v\in H^{1,2}_C(TX)\cap D(\dive)$.

\begin{definition}\label{def:approximationoftheboundary}
	Let $(X,\sfd,\mm)$ be an $\RCD(K,N)$ m.m.s.\ and $E\subset X$ a set of finite perimeter. Given $x\in X$ and a sequence $r_i\downarrow 0$ we say that $\left\lbrace u^{r_i}:=(u_1^{r_i},\ldots,u_{k-1}^{r_i}):B_{r_i}(x)\to \R^{k-1}\right\rbrace _{i\in\N}$ is a good approximation of the boundary of $E$ at $x$ if the following conditions hold true:
	\begin{itemize}
		\item[(i)] there exists a sequence $\delta_i\to 0$ such that $u^{r_i}:B_{r_i}(x)\to \R^{k-1}$ is a $\delta_i$-splitting map with $u^{r_i}(x)=0$;
		\item[(ii)] there exists $(Z,\sfd_Z)$ that realizes the convergences
		\begin{equation*}
		(X,r_i^{-1}\sfd,\mm_x^{r_i},x)\to (\R^k,\sfd_{\text{eucl}},c_k\Leb^k,0^k)
		\quad \text{and}\quad E_{r_i}\to \left\lbrace x_k>0\right\rbrace \ \text{locally strongly in $BV$}
		\end{equation*}
		and $r_i^{-1} u_{\alpha}^{r_i}\to x_{\alpha}$ in $H^{1,2}$-strong on $B_1(0^k)$ along the sequence
		\begin{equation*}
		(X,r_i^{-1}\sfd,\mm_x^{r_i},x)\to (\R^k,\sfd_{eucl}, c_k\Leb^k,0^k),
		\end{equation*}
		for any $\alpha=1,\ldots,k-1$.
	\end{itemize}
\end{definition}
\begin{lemma}\label{lemma:appr}
	Let $(X,\sfd,\mm)$ be an $\RCD(K,N)$ m.m.\ space and $E\subset X$ a set of finite perimeter and finite measure. Then for any $p\in X$ and for any $\eps>0$ there exists $V\in \TestV(X)$ such that
	\begin{equation*}
	\int_{B_2(p)}|\nu-V|^2\,\d|D\nchi_E|\le \eps,
	\end{equation*}
	where $\nu$ is the exterior normal of $E$.\\
	Moreover, there exists $G\subset B_1(p)$ with $\haus^h(B_1(p)\setminus G)\le C_{K,N} \sqrt{\eps}$ and such that, for any $x\in G$, it holds
	\begin{equation*}
	\limsup_{r\to 0}\frac{r}{\mm(B_r(x))}\int_{B_r(x)}|\nu-V|^2\,\d|D\nchi_E|\le \sqrt{\eps}.
	\end{equation*}
\end{lemma}

\begin{proof}
	The first conclusion follows from \Cref{thm:Gauss-Green}, where we proved that the normal is an element of $L^2_E(TX)$, and \Cref{lem:density_TestVE}, yielding density of $\tr_E\left(\TestV(X)\right)$ in $L^2_E(TX)$.
	
	To prove the second part of the statement we set 
	\begin{equation*}
	G:=\left\lbrace x\in B_1(p)\;:\;\limsup_{r\downarrow 0}\frac{r}{\mm(B_r(x))}\int_{B_r(x)}\abs{\nu-V}^2\,\d|D\nchi_E|\le\sqrt{\eps}\right\rbrace. 
	\end{equation*}
	Then, for any $r_0>0$ and for any $x\in B_1(p)\setminus G$, there exists $r_x<r_0$ such that
	\begin{equation*}
	\frac{r_x}{\mm(B_{r_x}(x))}\int_{B_{r_x}(x)}\abs{\nu-V}^2\,\d|D\nchi_E|>\sqrt{\eps}.
	\end{equation*}
	Hence, applying Vitali's covering theorem we can find a disjoint set of balls $\left\lbrace B_{r_i}(x_i)\right\rbrace $ such that $\left\lbrace B_{5r_i}(x_i) \right\rbrace $ is a covering of $B_1(p)\setminus G$. Now we can estimate, for any $r_0>0$,
	\begin{align*}
	\haus_{5r_0}^h(B_1(p)\setminus G)&\le \sum_{i=0}^{\infty}\frac{\mm(B_{5r_i}(x_i))}{5r_i}
	\le C_{K,N}\sum_{i=0}^{\infty}\frac{\mm(B_{r_i}(x_i))}{r_i}\\
	&\le\frac{C_{K,N}}{\sqrt{\eps}}\sum_{i=0}^\infty\int_{B_{r_i}(x_i)}\abs{\nu-V}^2\,\d|D\nchi_E|\le \frac{C_{K,N}}{\sqrt{\eps}}\int_{B_2(p)}\abs{\nu-V}^2\,\d|D\nchi_E|\\
	&\le C_{K,N}\sqrt{\eps}.
	\end{align*}
	The conclusion follows letting $r_0\downarrow 0$.
\end{proof}

\begin{proof}[Proof of \Cref{prop:good aproximation}]
	The proof is divided in 3 steps. Aim of the first one is to prove that good approximations of the boundary are regular enough to guarantee that the scalar product between their gradient and the gradient of any given test function leaves a well-defined trace over the reduced boundary of $E$. In the second step we combine the outcome of the first one, the approximation result of \Cref{lemma:appr} and the \textit{orthogonality in weak sense} between the normal vector and the coordinates of its orthogonal hyperplane guaranteed by the Gauss--Green formula, to get that gradients of good approximations of the boundary leave a trace even when coupled with the normal to the boundary and that this trace is $0$. In the last step we prove existence of good approximations of the boundary and combine it with steps 1 and 2 to get the sought conclusion.

	\textbf{Step 1.} Observe that it suffices to restrict the attention to the ball $B_1(p)\subset X$, for any $p\in X$. 
	
	We claim that for any function $\phi\in\Test(X,\sfd,\mm)$ there exists a $\abs{D\nchi_E}$-negligible set ${\rm N}\subset\mathcal{F}_kE\cap B_1(p)$ such that, for any $x\in \mathcal{F}_kE\cap B_1(p)\setminus{\rm N}$ and any good approximation of the boundary of $E$ at $x$ with radii $r_i\downarrow 0$ and maps $\left\lbrace u^{r_i}:=(u^{r_i}_1,\ldots,u^{r_i}_{k-1})\,:\,B_{r_i}(x)\to \R^{k-1}\right\rbrace _{i\in \N}$, there exist a subsequence $r_{i_j}$ and $c(x)=(c_1(x),\dots,c_{k-1}(x))\in\R^{k-1}$ such that
	\begin{equation}\label{z21}
	\lim_{j\to\infty}\frac{r_{i_j}}{\mm(B_{r_{i_j}}(x))}\int_{B_{r_{i_j}}(x)}\abs{\nabla u^{r_{i_j}}_{\alpha}\cdot\nabla\phi-c_{\alpha}(x)}^2\,\d|D\nchi_E|=0,\quad\text{for any $\alpha=1,\dots,k-1$.}
	\end{equation} 
	
	Assume without loss of generality that $\abs{\nabla\phi}\le 1$. 
	Let us fix also $\alpha\in\left\lbrace 1\ldots,k-1\right\rbrace$ and set $g_{i}:=\nabla u^{r_i}_{\alpha}\cdot \nabla \phi$. We have
	\begin{itemize}
		\item[(i)] $\norm{g_i}_{L^{\infty}(B_{r_i}(x))}\le C_N$;
		\item[(ii)] $r_i^2\fint_{B_{r_i}(x)} |\nabla g_i|^2\,\d\mm\le 2\delta_i+C_N r_i^2\fint_{B_{r_i}(x)} |\Hess \phi |^2\,\d\mm$, where $\delta_i$ is as in \Cref{def:approximationoftheboundary}.
	\end{itemize}
	Since $\Hess \phi\in L^2(B_2(p),\mm)$, by \Cref{lem:H_theta_null} and \Cref{lemma:AssolutacontinuitaHauss}, we deduce that
	\begin{equation*}
		\lim_{r\to 0}r^2\fint_{B_{r}(x)} |\Hess \phi |^2\,\d\mm=0
	\end{equation*}
	for any $x\in X$ outside a $|D\nchi_E|$-negligible set depending only on $\phi$.
	Therefore we can assume that $x$ does not belong to this set obtaining
	\begin{equation}\label{z20}
	\lim_{i\to\infty} r_i^2\fint_{B_{r_i}(x)} |\nabla g_i|^2\d\mm=0.
	\end{equation}
	This gives that, up to subsequence, $g_i\to c_{\alpha}(x)$ in $H^{1,2}$-strong on $B_1(0^k)$ along the sequence in \Cref{def:approximationoftheboundary}(ii). Here we have used \eqref{eq:asympotic of different normalization}.
	 Taking into account \Cref{prop:sum L1 convergence}, it follows that $(g_i-c_{\alpha}(x))\to 0$ in $H^{1,2}$-strong on $B_1(0^k)$ and thus, reading the convergence in the starting space,
	\begin{equation}\label{z19}
	\fint_{B_{r_{i_j}}(x)} |g_{i_j} -c_{\alpha}(x)|^2\,\d\mm+r_{i_j}^2\fint_{B_{r_{i_j}}(x)} |\nabla g_{i_j}|^2\,\d\mm=:\eps_j\to 0\quad \text{as $j\to\infty$}.
	\end{equation}
	We wish to prove that, up to excluding another $|D\nchi_E|$-negligible set depending only on $E$, \eqref{z19} gives \eqref{z21}.
	More precisely we are going to prove that \eqref{z19} implies \eqref{z21} at any $x\in X$ such that $x\in E_{r_0,C}$ for some $r_0>0$ and $C>1$, where
	\begin{equation}\label{z22}
	E_{r_0,C}:=\left\lbrace y\in X:\ C^{-1}\le \frac{r|D\nchi_E|(B_r(y))}{\mm(B_r(y))}\le C \  \forall r<r_0\right\rbrace,
	\end{equation}
	and	
	\begin{equation}\label{eq:w1}
	\lim_{r\to 0}\frac{|D\nchi_E|(E_{r_0,C}\cap B_r(x))}{|D\nchi_E|(B_r(x))}=1.
	\end{equation}
	Observe that \eqref{z22} and \eqref{eq:w1} are satisfied at $|D\nchi_E|$-a.e. point in
	 $\mathcal{F}E$ thanks to \Cref{th:uniqueness}, the asymptotic doubling property of $\abs{D\nchi_E}$ and elementary considerations. In order to keep notations short, from now on we set $r_j:=r_{i_j}$ and $g_j:=g_{i_j}$.\\
	We claim that, for any $j$ such that $r_j\le  r_0/5$, it holds
	\begin{equation}\label{eq:claim}
	|D\nchi_E|\big(E_{r_0,C}\cap B_{r_j}(x)\cap \left\lbrace |g_j-c_{\alpha}(x)|^2\ge \sqrt{\eps_j}\right\rbrace\big)\le C C_{N,K}\sqrt{\eps_j}\frac{\mm(B_{r_j}(x))}{r_j},
	\end{equation}
	where $\eps_j$ is as in \eqref{z19} and $r_0$ and $C$ are as in \eqref{z22}.\\  
	Notice that \eqref{eq:claim}, together with the Chebyshev inequality, (i) and \eqref{eq:w1}, give \eqref{z21}.\\
	Let us see how to establish \eqref{eq:claim}.
	Fix any $j$ such that $r_j\le r_0/5$ and let us set 
	\begin{equation*}
	(X_j, \sfd_j, \mm_j, x):=\left(X,r_j^{-1}\sfd, \frac{\mm}{\mm(B_{r_j}(x))},x\right).
	\end{equation*}
	With a slight abuse of notation we use the notations $E_{r_0,C}$ and $g_j$ also in $X_j$.
	Let us observe that, when read in $X_j$, \eqref{z19} turns into
	\begin{equation*}
	\fint_{B^j_{1}(x)} |g_j-c_{\alpha}(x)|^2\,\d\mm_j+\fint_{B^j_{1}(x)}|\nabla g_j|^2\,\d\mm_j\le \eps_j.
	\end{equation*}
	Moreover, a telescopic argument as in the proof of \Cref{proposition:propagation} gives
	\begin{align*}
	B_1^j(x)\cap E_{r_0,C}&\cap \left\lbrace |g_j-c_{\alpha}(x)|^2\ge C_{N,K}\sqrt{\eps_j}\right\rbrace \\
	\subset\,&B_1^j(x)\cap E_{r_0,C}\cap\left\lbrace z\;:\;\sup_{0<s<1} s\fint_{B_s^j(z)}|\nabla g_j|^2\,\d\mm_j>\sqrt{\eps_j}\right\rbrace .
	\end{align*}
	Using Vitali's lemma we can find a disjoint family $\left\lbrace B_{s_i}^j(z_i)\right\rbrace _{i\in \N}$ with $s_i\le 1$ and $z_i\in B_1^j(x)\cap E_{r_0,C}\cap \left\lbrace z\,:\,\sup_{0<s<1} s\fint_{B_s^j(z)}|\nabla g_j|^2\,\d\mm_j>\sqrt{\varepsilon_j}\right\rbrace $ for any $i\in\N$ such that
	\begin{equation*}
	B_1^j(x)\cap E_{r_0,C}\cap \left\lbrace z\;:\;\sup_{0<s<1} s\fint_{B_s^j(z)}|\nabla g_j|^2\,\d\mm_j>\sqrt{\varepsilon_j}\right\rbrace   \subset	\bigcup_{i\in\N} B_{5 s_i}^j(z_i).
	\end{equation*}
	Taking into account \eqref{z22} and the defining identities
	\begin{equation*}
	B^j_{s_i}(z_i)=B_{r_js_i}(z_i),
	\quad
	\mm_j=\frac{\mm}{\mm(B_{r_j}(x))},
	\end{equation*}
	we get
	\begin{align*}
	\frac{r_j}{\mm(B_{r_j}(x))} &|D\nchi_E|\big(E_{r_0,C}\cap B_{r_j}(x)\cap \left\lbrace |g_j-c_{\alpha}(x)|^2\ge \sqrt{\eps_j}\right\rbrace\big)
	\le  \frac{r_j}{\mm(B_{r_j}(x))} \sum_{i\in\N} |D\nchi_E|(B_{5r_j s_i}(z_i))\\
	\le & \frac{C C_{N,K} r_j}{\mm(B_{r_j}(x))} \sum_{i\in\N}\frac{\mm(B_{ r_js_i}(z_i))}{r_js_i}
	=  C C_{N,K} \sum_{i\in\N} \frac{\mm_j(B_{s_i}^j(z_i))}{s_i}
	\le \frac{CC_{N,K}}{\sqrt{\eps_j}}\int_{B_1^j(x)} |\nabla g_j|^2\,\d\mm_j\\
	\le & CC_{N,K}\sqrt{\eps_j}.
	\end{align*}
	% % % %
	
	\textbf{Step 2.} We wish to prove that, for $|D\nchi_E|$-a.e.\ $x\in \mathcal{F}_k E$ and any good approximation of the boundary of $E$ at $x$ with radii $r_i\downarrow0$ and maps $\left\lbrace u^{r_i}:=(u^{r_i}_1,\ldots,u^{r_i}_{k-1})\,:\,B_{r_i}(x)\to \R^{k-1}\right\rbrace _{i\in \N}$, there exists a subsequence $r_{i_j}\to 0$ such that
	\begin{equation}\label{eq:claimstep2}
	\lim_{j\to\infty} \frac{r_{i_j}}{\mm(B_{r_{i_j}}(x))}\int_{B_{r_{i_j}}(x)}|\nu \cdot \nabla u_{\alpha}^{r_{i_j}}|\,\d|D\nchi_E|=0
	\quad \text{for any }\alpha=1,\ldots,k-1.
	\end{equation}
	Let us restrict our attention as above to $\mathcal{F}_kE\cap B_1(p)$.\\
	We claim that, for any $\eps>0$, there exists $G_{\eps}\subset B_1(p)\cap \mathcal{F}_k E$ with $\haus^h(B_1(p)\cap \mathcal{F}_k E\setminus G_{\eps})\le C_{N,K} \sqrt{\eps}$ and such that, for any $x\in G_{\eps}$, and any $\left\lbrace u^{r_i}:=(u^{r_i}_1,\ldots,u^{r_i}_{k-1})\,:\,B_{r_i}(x)\to \R^{k-1}\right\rbrace _{i\in \N}$ good approximation of the boundary of $E$ at $x$, there exists a subsequence $r_{i_j}\to 0$ satisfying
	\begin{equation}\label{z14}
	\limsup_{j\to\infty} \frac{r_{i_j}}{\mm(B_{r_{i_j}}(x))}\int_{B_{r_{i_j}}(x)}|\nu \cdot \nabla u_{\alpha}^{r_{i_j}}|\,\d|D\nchi_E|\le C_{N,K}\eps^{1/4}
	\quad \text{for any }\alpha=1,\ldots,k-1.
	\end{equation}
	
	Before then proving the claim let us see how it implies \eqref{eq:claimstep2}.\\ 
	Fix $\eps>0$, set $\eps_i:=\eps 2^{-i}$ and take $G^{\eps}:=\cup_{i\in\N} G_{\eps_i}$. Then we have $|D\nchi_E|(B_1(p)\cap\mathcal{F}_k E\setminus G^{\eps})=0$, thanks to \Cref{lemma:AssolutacontinuitaHauss}, and \eqref{z14} holds for any $x\in G^{\eps}$. Therefore the set $\cap_{i\in\N} G^{\eps_i}$ has full $|D\nchi_E|$-measure in $B_1(p)\cap \mathcal{F}_kE$ and has the sought property.
	
	The remaining part of this step is devoted to the proof of \eqref{z14}.
	Let $\eps>0$ be fixed, take $G$ and $V$ as in \Cref{lemma:appr}. Recalling that any test vector field can be represented as $\sum_{i=1}^m \eta_i\nabla \phi_i$ with $\eta_i,\phi_i\in\Test(X,\sfd,\mm)$ for some $m\in\N$ and using Step 1, we conclude that there exists $G_{\eps}\subset G\cap \mathcal{F}_kE$ with $|D\nchi_E|(G\cap\mathcal{F}_kE\setminus G_{\eps})=0$ and the property that, for any $x\in G_{\eps}$ and $\left\lbrace u^{r_i}:=(u_1^{r_i},\ldots,u_{k-1}^{r_i})\,:\,B_{r_i}(x)\to \R^{k-1}\right\rbrace _{i\in\N}$ good approximation of the boundary of $E$ at $x$, there exists $c(x):=(c_1(x),\ldots,c_{k-1}(x))\in \R^{k-1}$ and a subsequence $r_{i_j}\to 0$ such that
	\begin{equation}\label{z16}
	\lim_{j\to\infty}\frac{r_{i_j}}{\mm(B_{r_{i_j}}(x))}\int_{B_{r_{i_j}}(x)} |\nabla u^{r_{i_j}}_{\alpha}\cdot V -c_{\alpha}(x)|^2\,\d|D\nchi_E|=0
	\quad\text{for}\ \alpha=1,\ldots,k-1.
	\end{equation}
	In order to conclude the proof it suffices to show that
	\begin{equation}\label{z15}
	|c(x)|\le C_{K,N}\eps^{1/4}.
	\end{equation}
	Indeed, in that case, one has
		\begin{align*}
	&\limsup_{j\to\infty}\frac{r_{i_j}}{\mm(B_{r_{i_j}}(x))}\int_{B_{r_{i_j}}(x)} |\nu\cdot \nabla u^{r_{i_j}}_{\alpha}|\,\d|D\nchi_E|	\\
	\le\,&C_N\limsup_{j\to\infty}\bigg(\frac{r_{i_j}}{\mm(B_{r_{i_j}}(x))}\int_{B_{r_{i_j}}(x)} |\nu-V|^2\,\d|D\nchi_E|\bigg)^{1/2}
	+ \limsup_{j\to\infty}\frac{C_N r_{i_j}}{\mm(B_{r_{i_j}}(x))}\int_{B_{r_{i_j}}(x)} | \nabla u^{r_{i_j}}_{\alpha}\cdot V|\,\d|D\nchi_E|\\
	\le\,&C_N\eps^{1/4}+ \lim_{j\to\infty}C_N\bigg(\frac{r_{i_j}}{\mm(B_{r_{i_j}}(x))}\int_{B_{r_{i_j}}(x)} |\nabla u^{r_{i_j}}_{\alpha}\cdot V -c_{\alpha}(x)|^2\,\d|D\nchi_E|\bigg)^{1/2}+|c_{\alpha}(x)|\frac{r_{i_j}|D\nchi_E|(B_{r_{i_j}}(x))}{\mm(B_{r_{i_j}}(x))}\\
	\le\,&C_{K,N}\eps^{1/4},
	\end{align*}
	where we used \eqref{z16}, \eqref{z15} and the fact that $x\in \mathcal{F}_kE$.\\
	In order to prove \eqref{z15} we simplify the notation setting $r_{i_j}=:r_j$. Choose a smooth function $\psi_{\infty}:\R^k\to\R$ with compact support in $B_1(0^k)$ and such that $\int_{\left\lbrace x_k=0\right\rbrace }\psi_{\infty}\,\d\Leb^{k-1}=:C_k>0$. Then we consider a sequence $\psi_j\in \Lip(X,\sfd)$ with $\supp(\psi_j)\subset B_{r_j}(x)$, $\norm{\psi_j}_{L^\infty}\le 2$ and $\psi_j\to \psi_{\infty}$ strongly in $H^{1,2}$ along the sequence in \Cref{def:approximationoftheboundary}(ii), whose existence is proved in \Cref{lemma:approxWithLipschitz}.
	Observe now that
	\begin{equation}\label{z17}
	\lim_{j\to\infty} \frac{r_j}{\mm(B_{r_j}(x))}\int_E \nabla \psi_j\cdot \nabla u^{r_j}_{\alpha}\,\d\mm
	=c_k\int_{\left\lbrace x_k>0\right\rbrace } \nabla \psi_{\infty}\cdot e_{\alpha}\,\d\Leb^k=0,
	\quad\text{for $\alpha=1,\ldots,k-1$},
	\end{equation}
	and 
	\begin{equation}\label{z18}
	\lim_{j\to\infty} \frac{r_j}{\mm(B_{r_j}(x))}\int \psi_j V\cdot\nabla u^{r_j}_{\alpha} \d|D\nchi_E|
	=C_k c_{\alpha}(x),
	\quad\text{for $\alpha=1,\ldots,k-1$},
	\end{equation}
	where the last equality in \eqref{z17} is obtained integrating by parts and to prove \eqref{z18} we used \eqref{z16}.
	Building upon \eqref{z17}, \eqref{z18}, \Cref{thm:Gauss-Green} and \Cref{lemma:appr}, we get \eqref{z16}:
	\begin{align*}
	C_k|c_{\alpha}(x)| = & \abs{ \lim_{j\to\infty} \frac{r_j}{\mm(B_{r_j}(x))}\left(  \int_E \nabla \psi_j\cdot \nabla u^{r_j}_{\alpha} \d\mm+
		\int \psi_j V\cdot\nabla u^{r_j}_{\alpha}\,\d|D\nchi_E|\right)  }\\
	= &  
	\abs{ \lim_{j\to\infty} \frac{r_j}{\mm(B_{r_j}(x))}\left(  -\int \psi_j\nu \cdot \nabla u^{r_j}_{\alpha}\,\d|D\nchi_E|+ \int \psi_j V\cdot\nabla u^{r_j}_{\alpha}\,\d|D\nchi_E|\right)}\\
	\le &\limsup_{j\to\infty}\frac{C_Nr_j}{\mm(B_{r_j}(x))}\int_{B_{r_j}(x)} |\nu-V|\,\d|D\nchi_E|\\
	\le & C_{N,K}\eps^{1/4}.
	\end{align*}
	Note that in order to apply the Gauss--Green formula in the previous estimate the fact that $u_{\alpha}^{r_j}$ is locally the restriction of a $H^{2,2}(X,\sfd,\mm)$ function (see \Cref{remark:localHessian}) plays a role.
	
	\textbf{Step 3.} In order to conclude the proof we just observe that, since 
	\begin{equation*}
	\sfd_{pmGH}\left((X,r^{-1}\sfd,\mm_x^r,x),(\R^k,\sfd_{eucl},c_k\Leb^k,0^k)\right)\to 0
	\end{equation*}
	as $r\downarrow 0$ and the blow-up of the set of finite perimeter is a half-space (in the sense of $BV_{\rm loc}$ convergence, as we pointed out after \Cref{def:tan}), a slight modification of \Cref{prop:epsclosedeltasplit}\footnote{With the splitting functions defined on balls of radius $1$ in place of $5$.} provides, for any sequence $r_i\downarrow0$, existence of a good approximation of the boundary of $E$ at $x$ with maps $\left\lbrace u^{r_i}:=(u^{r_i}_1,\ldots,u^{r_i}_{k-1})\,:\,B_{r_i}(x)\to \R^{k-1}\right\rbrace _{i\in \N}$ (observe that \Cref{prop:epsclosedeltasplit} gives $\delta_i$-splitting maps defined on the balls of radius $1$ of the rescaled spaces for a sequence $\delta_i\downarrow0$ and then rescale these functions). The sought conclusion follows now from what we obtained in the previous step.
\end{proof}

\subsection{Proof of \texorpdfstring{\Cref{prop:rectifiability}}{Prop}}
The proof is divided in three steps.\\ 
Aim of the first one is to provide a bridge between analysis and geometry suitable for this context. We prove that, whenever at a certain location and scale the set of finite perimeter is quantitatively close to a half-space in a Euclidean space and there is a $(k-1,\delta)$-splitting map which is also $\delta$-orthogonal to the normal vector in the sense of \eqref{eq:ort}, then the $(k-1,\delta)$-splitting map is an $\eta$-isometry (in the scale invariant sense) when restricted to the support of the perimeter.\\ 
The second step is analytic and dedicated to the propagation of the $\delta$-orthogonality condition.\\ In the last one we get the bi-Lipschitz property relying on the observation that a map which is an $\eta$-isometry (in the scale invariant sense) at any location and scale is bi-Lipschitz.

\textbf{Step 1.}
	Let $N>0$, $K\in\R$ and $k\in [1,N]$ be fixed. 
	We claim that, for any $\eta>0$, there exists $\delta=\delta_{\eta,N}\le \eta$ such that, for any pointed $\RCD(K,N)$ m.m.s.\ $(X,\sfd,\mm,x)$ and for any set of finite perimeter and finite measure $E\subset X$ such that, for some $0<r<\abs{K}^{-1/2}$,
	\begin{itemize}
	\item[(i)] $\sfd_{pmGH}\left(\left(X, (2r)^{-1}\sfd,\frac{\mm}{\mm(B_{2r}(x))},x\right),\left(\R^k,\sfd_{eucl},\frac{1}{\omega_k}\Leb^k,0^k\right)\right)<\delta$;
	\item [(ii)]
	\begin{equation}
	\abs{\frac{\mm(B_t(x)\cap E)}{\mm(B_t(x))}-\frac{1}{2}}+\abs{\frac{t\abs{D\nchi_E}(B_t(x))}{\mm(B_t(x))}-\frac{\omega_{k-1}}{\omega_k}}<\delta
	\quad \text{for any $t\le 2r$};
	\end{equation}
	\item[(iii)] there exists $u:=(u_1,\ldots,u_{k-1}):B_{2r}(x)\to \R^{k-1}$ a $\delta$-splitting map satisfying
		\begin{equation}\label{eq:orthogonality}
			\frac{r}{\mm(B_{2r}(x))}
			\int_{B_{2r}(x)}\abs{\nu\cdot \nabla u_a}\,\d|D\nchi_E|<\delta,
			\quad\text{for any
			$a=1,\ldots,k-1$,}
		\end{equation}
   \end{itemize}
   then $u:\supp\abs{D\nchi_E}\cap B_{r}(x)\to B_r^{\R^{k-1}}(u(x))$ is an $\eta r$-GH isometry.
   \smallskip
   
   By scaling it is enough to prove the claim when $r=1/2$ and $|K|\le 4$. Let us argue by contradiction. Then we could find $\eta>0$, a sequence $(X_n,\sfd_n,\mm_n,E_n, x_n)$, points $z_1^n,z_2^n\in \supp |D\nchi_{E_n}|\cap B_{1/2}(x_n)$, and $1/n$-splitting maps $u^n:B_1(x_n)\to \R^{k-1}$ satisfying (i), (ii) and (iii) with $\delta=1/n$, $u^n(x_n)=0$ and
	\begin{equation}\label{z12}
		\big| |u^n(z_1^n)-u^n(z_2^n)|-\sfd_n(z_1^n,z_2^n)\big| \ge \eta,
		\quad \forall n\in\N.
	\end{equation}
	Notice that $\sfd_n(z_1^n,z_2^n)\ge\min\{\eta/(C_N-1),\eta\}$ since $u^n$ is $C_N$-Lipschitz.
	
	Observe that, by (i), $(X_n,\sfd_n,\mm_n,x_n)$ converge in the pmGH topology to $\left(\R^k,\sfd_{eucl},\frac{1}{\omega_k}\Leb^k,0^k\right)$. We can assume the existence of a metric space $(Z,\sfd_Z)$ realizing this convergence (cf.\ \Cref{subsubsection:stability results}).
	Since $E_n$ satisfies the bound
	\begin{equation}\label{z23}
		\abs{\frac{\mm_n(E_n\cap B_t(x_n))}{\mm_n(B_t(x_n))}-\frac      {1}{2}}+\abs{\frac{t|D\nchi_{E_n}|(B_t(x_n))}{\mm_n(B_t(x_n))}-\frac{\omega_{k-1}}{\omega_k}}<1/n
		\quad\text{for any $t\le 1$},	
	\end{equation}
	up to extracting a subsequence, $E_n\cap B_1(x_n)\to F\cap B_1(0^k)$ in $L^1$-strong, where $F$ is of locally finite perimeter in $B_1(0^k)$ thanks to \Cref{prop:compactness sets with finite perimeter}.

   Up to extracting again a subsequence we can   assume $u^n\to u^{\infty}$ strongly in $H^{1,2}$ on $B_1(0^k)$, where $u^{\infty}: B_1^{\R^k}(0)\to \R^{k-1}$ is the restriction of an orthogonal projection, as a consequence of \Cref{thm:AscoliArzela} and \Cref{thm:stabilityLaplacian}. 
    We assume, without loss of generality, that $u^{\infty}(x)=(x_1,\ldots,x_{k-1})$ for any $x\in B_1(0^k)$.
   
   We claim that $\Leb^k\left(\left(F\cap B_1(0^k)\right)\Delta\left(\left\lbrace x_k>0 \right\rbrace\cap B_1(0^k)\right)\right)=0$ and 
   \begin{equation}\label{z28}
   	\int g\,\d |D \nchi_{E_n}|\to\int g\,\d |D\nchi_{\left\lbrace x_k>0\right\rbrace}|
   	\quad\text{for any $g\in C(Z)$ with $\supp(g)\subset B_{1/2}(0^k)$}.
   \end{equation}
   This would imply that $z_1^{\infty}, z_2^{\infty}\in \left\lbrace x_k=0 \right\rbrace$, therefore $|u^{\infty}(z_1^{\infty})-u^{\infty}(z_2^{\infty})|=\sfd_{eucl}(z_1^{\infty},z_2^{\infty})$ that contradicts \eqref{z12}.\\   
   In order to verify the claim we argue as in the proof of the second step of \Cref{prop:good aproximation}.
   We choose a smooth function $\psi_{\infty}:\R^k\to\R$ with compact support in $B_1(0^k)$.
   Then we consider a sequence $\psi_n\in \Lip(X_n,\sfd_n)$ with $\supp(\psi_n)\subset B_1(x_n)$, $\norm{\psi_n}_{L^\infty}+\norm{|\nabla \psi_n|}_{L^{\infty}}\le 4$ and $\psi_n\to \psi_{\infty}$ strongly in $H^{1,2}$ along
  the sequence $(X_n,\sfd_n,\mm_n,x_n)$, whose existence is proved in \Cref{lemma:approxWithLipschitz}.
  Observe now that
  \begin{equation*}
  	\nabla \psi_n\cdot \nabla u_a^n\to \nabla \psi_{\infty}\cdot e_a=\frac{\partial \psi_{\infty}}{\partial x_a}
  	\quad \text{in $L^2$-strong, for any $a=1,\ldots,k-1$},
  \end{equation*}
  by \Cref{prop:sum L1 convergence}(i) and \Cref{prop:sum L1 convergence}(iii).
  This observation, along with \Cref{prop:sum L1 convergence}(ii) and \Cref{remark:convergence of sets}, gives
  \begin{equation}\label{z24}
    \int_F\frac{\partial \psi_{\infty}}{\partial x_a}\,\d\frac{\Leb^k}{\omega_k} 
    = \lim_{n\to \infty} \int_{E_n}\nabla\psi_n\cdot\nabla u_a^n\,\d\mm_n.
  \end{equation}
  We can now use \eqref{z24}, Theorem \ref{thm:Gauss-Green} and (iii) to conclude that
  \begin{align*}
  \left|\int_F\frac{\partial \psi_{\infty}}{\partial x_a}\,\d\frac{\Leb^k}{\omega_k} \right|
  =& \lim_{n\to \infty} \left|\int_{E_n}\nabla\psi_n\cdot\nabla u_a^n\,\d\mm_n\right|\\
  =&\lim_{n\to\infty}\left|\int\psi_n\nabla u_a^n\cdot\nu_{E_n}\,\d|D\nchi_{E_n}| \right|\\
  \le&\lim_{n\to\infty}\int\abs{\psi_n}\abs{\nabla u_a^{n}\cdot\nu_{E_n}}\,\d|D\nchi_{E_n}|=0,
  \end{align*}
  for $a=1,\ldots,k-1$. Since $\psi_{\infty}\in C^{\infty}_c(B_1(0^k))$ is arbitrary we obtain that 
  \begin{equation*}
  	\Leb^k\left(\left(F\cap B_1(0^k)\right)\Delta\left(\left\lbrace x_k>\lambda \right\rbrace\cap B_1(0^k)\right)\right)=0
  	\quad \text{for some $\lambda\in \R$.}
  \end{equation*}
  Using again \eqref{z23} we get $\Leb^k(F\cap B_1(0^k))=\omega_k/2$ that forces $\lambda=0$.\\ 
  Let us finally prove \eqref{z28}. To this end we use again \eqref{z23} with $t=1/2$ obtaining that
  \begin{equation*}
  	\lim_{n\to \infty}\abs{D\nchi_{E_n}}(B_{1/2}(x_n))=\frac{\omega_{k-1}}{2^{k-1}}=\abs{D\nchi_{\{x_k>0\}}}(B_{1/2}(0^k)).
  \end{equation*}
  We can now apply the third conclusion of \Cref{prop:compactness sets with finite perimeter} and conclude.
 
\textbf{Step 2.}
	By assumption there exists an $\eps$-splitting map $u:B_{2s}(p)\to\R^{k-1}$ such that \eqref{eq:ort} holds true.
	% % % % % %
    We wish to propagate now both the $\eps$-splitting condition and the orthogonality condition \eqref{eq:ort} at any scale and point outside a set of small $\haus_{5}^h$-measure. More precisely we are going to prove that there exists a set $G\subset B_s(p)$ with $\haus_{5}^h(B_s(p)\setminus G)\le C_N\sqrt{\eps}\frac{\mm(B_{s}(p))}{s}$ such that
    \begin{itemize}
    	\item[(i)] for any $x\in G$, $0<r<s$, $u:B_r(x)\to \R^{k-1}$ is a $C_N\eps^{1/4}$-splitting map;
    	\item[(ii)] for any $x\in G$, $0<r<s$, it holds
    	\begin{equation}\label{eq:ort1}
    	\frac{r}{\mm(B_r(x))}
    	\int_{B_r(x)} \abs{\nu\cdot \nabla u_a}\,\d|D\nchi_E|<\sqrt{\eps},
    	\quad
    	\text{for}\ a=1,\ldots,k-1.
    	\end{equation}
    \end{itemize}
     We can find a set $G'$ satisfying the measure estimate and (i) applying \Cref{cor:scale invariant version of propagation}. Hence it is enough to find a set $G''$ satisfying the measure estimate and (ii) and to take $G:=G'\cap G''$.\\ 
     To do so we apply a standard maximal argument. Let us fix $a=1,\ldots,k-1$ and set
     \begin{equation*}
     	M(x):=\sup_{0<r<s}
     	\frac{r}{\mm(B_r(x))}
     \int_{B_r(x)} \abs{\nu\cdot\nabla u_a}\,\d|D\nchi_E|.
     \end{equation*}
     We claim that $G'':=\left\lbrace x\in B_s(p)\,:\,M(x)<\sqrt{\eps}\right\rbrace $ has the sought properties.\\     
     Indeed, for any $x\in B_s(p)\setminus G''$, there exists $\rho_x\in (0,s)$ such that
     \begin{equation}\label{z10}
     	\frac{\rho_x}{\mm(B_{\rho_x}(x))} \int_{B_{\rho_x}(x)} \abs{\nu\cdot\nabla u_a}\,\d|D\nchi_E| \ge\sqrt{\eps}.
     \end{equation}
     Applying Vitali lemma to the family $\left\lbrace B_{\rho_x}(x)\right\rbrace _{x\in B_s(p)\setminus G''}$ we find a disjoint subfamily $\left\lbrace B_{r_i}(x_i)\right\rbrace _{i\in\N}$ such that $B_s(p)\setminus G''\subset \cup_i B_{5r_i}(x_i)$. Taking into account the disjointedness of the covering, \eqref{z10}, \eqref{eq:ort} and the Bishop-Gromov inequality, we can compute
     \begin{align*}
     \haus_{5}^h(B_s(p)\setminus G'')\le & \sum_{i\in\N} h(B_{5 r_i}(x_i))
     =  \sum_{i\in\N} \frac{\mm(B_{5r_i}(x_i))}{5r_i}\\
     \le & C_N \sum_{i\in\N} \frac{\mm(B_{r_i}(x_i))}{r_i}
     \le  C_N \sum_{i\in\N}\eps^{-1/2}\int_{B_{r_i}(x_i)}\abs{\nu\cdot \nabla u_a}\,\d|D\nchi_E|\\
     \le& C_N \eps^{-1/2}\int_{B_{2s}(p)}\abs{\nu\cdot\nabla u_a}\,\d|D\nchi_E|
     \le C_N\sqrt{\eps}\,\frac{\mm(B_{2s}(p))}{s}.
     \end{align*}   
     \textbf{Step 3.} We claim now that for any $\eta>0$ there exists $\eps=\eps_{\eta,N}>0$ small enough such that for any $0<r<s$ and $x\in G\cap (\mathcal{F}_kE)_{2s,\eps}$ the map
     \begin{equation}\label{eq:claimconcl}
     	u=(u_1,\ldots,u_{k-1}):\supp |D\nchi_E|\cap B_r(x) \to\R^{k-1}
     	\qquad\text{is an $r\eta$-GH isometry.}
     \end{equation}
     The claim is a consequence of Step 1. Indeed, for any $x \in G\cap (\mathcal{F}_kE)_{2s,\eps}$ and any $r\in (0,s)$, the conditions (i) and (ii) of Step 1 are satisfied by definition of $(\mathcal{F}_kE)_{2s,\eps}$. Moreover $u$ is a $C_N \eps^{1/4}$-splitting map on $B_r(x)$ satisfying \eqref{eq:ort1}, hence also the assumption (iii) of Step 1 is satisfied for $\eps$ small enough.
     \smallskip
     
     In order to conclude the proof we have just to check the conclusion (i) in the statement of \Cref{prop:rectifiability}, since the conclusion (ii) follows from Step 2 choosing $\eps$ small enough so that $\sqrt{\eps}<\eta$. To this aim, take $x,y\in G\cap (\mathcal{F}_kE)_{2s,\eps}$ and choose $r:=\sfd(x,y)$. Our claim \eqref{eq:claimconcl} ensures that
     \begin{equation*}
     	\big||u(x)-u(z)|-\sfd(x,z)\big|\le r\eta 
     	\qquad\text{for any}\ z\in \supp |D\nchi_E|\cap B_r(x),
     \end{equation*}
     therefore we can take $z=y$ and conclude. 
\def\cprime{$'$} \def\cprime{$'$}

\end{document}